\documentclass[10pt]{amsart}

\usepackage[utf8]{inputenc} 
\usepackage[T1]{fontenc}   
\usepackage{lmodern}       
\usepackage{microtype}     
\usepackage[french,english]{babel}

\usepackage{geometry}
\usepackage{amsmath}
\usepackage{amsfonts}
\usepackage{amssymb}
\usepackage{enumitem}
\usepackage{hyperref}
\usepackage{dsfont}

\usepackage{url}
\usepackage{xcolor}
\colorlet{darkgreen}{green!50!black}
\usepackage{amsbsy}
\usepackage{amsmath}

\hypersetup{
    colorlinks=true,
    urlcolor=blue,
    linkcolor=blue,
    breaklinks=true,
    citecolor=darkgreen
}

\usepackage{graphicx}
\usepackage{caption}
\usepackage{subcaption}

\newcommand{\m}{\tilde x_b}
\newcommand{\M}{{x_b}}
\let  \tilde = \widetilde
\def\R{{\mathbb R}}

\def\C{{\mathbb C}}

\def\e{{\epsilon}}
\def\1{{1\!\!\!1}}

\def\E{{\mathbb E}}

\def\P{{\mathbb P}}

\newcommand{\fc}{\mathds{1}}
\newcommand{\ind}{\mathds{1}}
\newcommand\inter[1]{\overset{{}_\circ}{#1}}

\let \phi=\varphi

\newcommand{\be}{\begin{equation}}
\newcommand{\ee}{\end{equation}}
\numberwithin{equation}{section}

\newtheorem{theo}{Theorem}

\newtheorem{proposition}{Proposition}[section]

\newtheorem{corollary}[proposition]{Corollary}
\newtheorem{lemma}[proposition]{Lemma}

\newtheorem{hypothesis}{Hypothesis}

\theoremstyle{definition} 
\newtheorem{defi}[proposition]{Definition}

\theoremstyle{remark}     
\newtheorem{rem}[proposition]{Remark}

\usepackage{oubraces}

\newcommand{\nb}[1]{#1}
\newcommand{\nr}[1]{}
\begin{document}

\title[On the Green's functions and  Martin boundary of a diffusion in a discontinuous medium]{On the Green's functions and  Martin boundary structure of a planar diffusion in a discontinuous layered medium}

\author{Sandro Franceschi}
\address{Institut Polytechnique de Paris, T\'el\'ecom SudParis, Laboratoire SAMOVAR, 19 place Marguerite Perey, 91120 Palaiseau, France}
\email{sandro.franceschi@telecom-sudparis.eu}

\author{Irina Kourkova}
\address{Sorbonne Universite, Laboratoire de Probabilités, Statistiques et Modélisation, 
UMR 8001, 4 place Jussieu, 75005 Paris, France. 
}
\email{irina.kourkova@sorbonne-universite.fr}

\author{Maxence Petit}
\address{Sorbonne Universite, Laboratoire de Probabilités, Statistiques et Modélisation, 
UMR 8001, 4 place Jussieu, 75005 Paris, France. }
\email{maxence.petit@ens-rennes.fr}

\thanks{This project has received funding from Agence Nationale de la Recherche, ANR JCJC programme 
under the Grant Agreement ANR-22-CE40-0002 (ANR RESYST).}

\begin{abstract} We consider a two-dimensional diffusion process in a two-layered plane, governed by distinct covariance matrices in the upper and lower half-planes and by two drift vectors pointed away from the $x$-axis.
We first analyze the case where the generator of the process is in divergence form, that is, when the flux is continuous across the interface. Then we extend the study to a broader class of processes whose behavior at the interface forms an oblique two-dimensional analogue of the skew Brownian motion.

We provide a detailed theoretical analysis of this transient process. Our main results are as follows:
(i) we derive explicit Laplace transforms of the Green's functions;
(ii) we compute exact asymptotics of the Green's functions along all possible trajectories in the plane; (iii) 
We determine all positive harmonic functions, identifying the full and minimal Martin boundaries, which turn out to be distinct.
The nonminimality of the Martin boundary is a noteworthy phenomenon for diffusions.

  To obtain an analytical description of the process, we fully develop a three-variable version of the so-called kernel method by deriving and exploiting a functional equation involving unknown Laplace transforms of Green's functions and two known
  kernels $\gamma_+(x,y)$ and $\gamma_{-}(x,z)$.
  The introduction of independent auxiliary variables $y$ and $z$, associated with each half-plane, is a key idea.



\end{abstract}
\maketitle
\setcounter{tocdepth}{1} 
\tableofcontents

\section{Introduction and main results}

\subsection{Context, goals and strategy}

The study of diffusion processes in multi-dimensional discontinuous or layered media is a challenging problem. 
Diffusion processes with discontinuous coefficients arise naturally in various applied contexts, such as geophysics, ecology, and astrophysics, where the medium may be composed of layers with different diffusion and advection properties (see \cite{Lejay_2006,Lejay2011} and references therein).
In such models, discontinuities of the diffusion matrix act as permeable interfaces, leading to complex boundary behaviors. 

In \cite{Lejay2011}, Lejay proposed a simulation algorithm for 2-dimensional diffusion processes in layered media. More precisely, the process is governed by different diagonal covariance matrices in the upper  and lower half-plane of  $\R^2$, the value of the singular drift on the axis is chosen to have a generator in the divergence form. 
In the present paper, we consider the generalisation of this process with arbitrary covariance matrices $\Sigma^+$ and $\Sigma^-$  in the upper and lower half-planes and drifts $\mu^+$ and $\mu^-$ directed away from the $x$-axis. This model can be interpreted as a “bang-bang” type process, switching between two regimes, a structure reminiscent of stochastic control models~\cite{Shreve1981}. 
Our main focus is on the case where the generator of the process is in divergence form, that is, when the flux is continuous across the interface. In the last section of this paper, we also consider a broader class of processes whose behavior can be viewed as a natural two-dimensional generalization of a Skew Bang-Bang Brownian Motion (SBBBM), with oblique interaction at the interface. 

A rich literature exists for skew Brownian motion in dimension one, see the foundational paper by Harrison and Shepp \cite{HarrisonShepp1981} and the survey by Lejay \cite{Lejay_2006}. See also Fernholz et al. for a detailed introduction and study of one-dimensional SBBBM
\cite{FernholzIchibaKaratzas2013} and Lenotre who studies and simulates this type of process \cite{Lenotre2015}.
Some results also exist in higher dimensions, a recent proof of weak existence and uniqueness was provided by Atar and Budhiraja~\cite{AtarBudhiraja2015}.
The first study of multidimensional skew Brownian motion was conducted by Portenko who constructed a diffusion process whose infinitesimal generator involves a singular drift concentrated on a smooth hypersurface~$S$~\cite{Portenko1979a,Portenko1979b}.
Zaitseva studied a stochastic differential equation describing a multidimensional Brownian motion 
with an oblique skew penetration through a hyperplane, under the assumption that both  the skewness coefficient and the direction are constant \cite{Zaitseva2005}. The extension to the general case of an oblique skew diffusion with non-constant coefficients remains a challenging open question \cite{AryasovaPavlyukevichPilipenko2024,PortenkoKopytko2012}.

The skew Brownian motion is deeply linked to diffusion processes generated by
\begin{equation}\label{eq:L_general}
\mathcal{L}
= \frac{\rho(x)}{2}\,\frac{d}{dx}\!\left(a(x)\,\frac{d}{dx}\right)
+ \mu(x)\,\frac{d}{dx},
\end{equation}
where $a$ and $\rho$ have discontinuities~\cite{Etore2006}. 
Indeed, the diffusion process $X$ generated by~$\mathcal{L}$ can be described as the solution of an SDE with local time:
\begin{equation}\label{eq:SDE_local_time}
dX_t = \sigma(X_t)\, dW_t + \mu(X_t)\,dt + \int_{\mathbb{R}} \nu(dx)\, dL_t^x(X),
\end{equation}
where $\sigma^2=\rho a$ and $\nu$ is a finite signed measure whose mass is concentrated at the points where $a$ or $\rho$ are discontinuous. See Legall's work \cite{Legall1984} in dimension one and Stroock's work \cite{Stroock1988} in higher dimension. The process studied in this paper satisfies an equation of type~\eqref{eq:SDE_local_time}.

From a structural viewpoint, our model also connects with discrete analogues such as random walks or queueing processes evolving under different dynamics in distinct regions. The semi-martingale reflecting Brownian motion (SRBM) in a cone and its discrete counterparts are now well understood in dimension two and beyond (see, e.g., \cite{FIM17,Williams95}), but models involving different generators in several cones sharing a common boundary have been much less studied. Among discrete examples, the joining-the-shortest-queue (JSQ) model (studied by Kurkova and Suhov \cite{Kurkova2003}, Kobayashi et al. \cite{Kobayashi2013}) and the doubly two-sided quasi-birth-and-death (DQBD) process (studied by Miyazawa \cite{Miyazawa2009QBD,Miyazawa2009,Miyazawa2015}) display related features, although symmetries sometimes reduce them to single-kernel functional equations. In contrast, our diffusion model genuinely involves two distinct kernels corresponding to the two half-planes. 
Some time-space continuous models for the multi-level single server queue are also studied in 
\cite{kobayashi2025, miyazawa2024multilevelRBM}.

\medskip
Our aim is to provide a detailed theoretical analysis of the transient process presented above.
The main results of the present paper are as follows:

\begin{enumerate}[label=(\roman*)]
\item\label{(i)} We make explicit the Laplace transforms of the Green's functions. 
\item\label{(ii)} We derive the exact asymptotics of the Green's functions along all possible trajectories in $\mathbb{R}^2$.
\item\label{(iii)} As a consequence, we obtain explicit expressions for all positive harmonic functions and determine the full Martin boundary. We also identify the minimal Martin boundary, which turns out to be strictly smaller than the full one.
\end{enumerate}
The nonminimality of the full Martin boundary reveals a genuinely new phenomenon, which contrasts with many known results for diffusions with regular coefficients, whether free, killed or confined in cones or in other domains, see \cite{ernst_franceschi_asymptotic_2021,Franceschi_2024} and DeBlassie's papers
\cite{Deblassie1999,Deblassie2008,Deblassie2010}.
\nr{This phenomenon can be also expected for diffusions in two-dimensional cones with special assumptions on drifts making the process escape to infinity along the boundaries, see \cite{IgKurRa,KurMa} for the proof in the case of the analogous discrete random walks. However, the minimal Martin boundary in the last example is reduced to only two points, while for our process it is homeomorphic to a union of arcs of a circle.}
\nb{Examples of Markov processes for which the minimal Martin boundary differs from the full Martin boundary are scarce in the literature. To the best of our knowledge, they mainly concern certain discrete Markov chains. Namely, in
\cite{Lau2011} and \cite{Kesseb_hmer_2020}, random walks on the set of finite words over a three-letter alphabet have a Martin boundary homeomorphic to the Sierpiński gasket, while the minimal Martin boundary consists of three points. In \cite{KurMa} and \cite{IgKurRa}, random walks in the quarter-plane with escape to infinity along the two axes have a two-point minimal Martin boundary, while the full Martin boundary also contains certain linear combinations of these points. Similar results to those in \cite{KurMa} can be expected for diffusion analogues of these walks. Nevertheless, in these examples, the minimal Martin boundary is reduced to a finite set of points, whereas in the present article it is homeomorphic to a union of arcs of a circle.}
 
\medskip
The main tool of our analysis is a version of the so-called kernel method usually used to study random processes in cones; \nb{see \cite{Malyshev1972} and \cite{FIM17} for the construction of explicit solutions, \cite{Malyshev1973} for the foundations of asymptotic analysis, \cite{IgnatioukRobert2025}
for a recent and very successful combination of the kernel method with other techniques or \cite{Zhao2022} for a modern survey and references therein}\nr{see Zhao \cite{Zhao2022} for a survey}. The kernel $\gamma(x, y)$ of the SRBM is usually given by the sum of one half of the quadratic form associated with the covariance matrix and the linear form corresponding to the drift. In our model, we consider two kernels, $\gamma^+(x, y)$ and $\gamma^-(x, z)$, which share the same first variable and have distinct second variables. The crucial idea of our method is to introduce the independent variables $y$ and $z$, which correspond to the two half-planes.

This approach was first fruitfully applied in the large deviation analysis of random walks in $\mathbb{Z}^{d+1}$ with a discontinuity on a hyperplane; as discussed by Ignatyuk et al. \cite{IAIgnatyuk_1994}. It was also employed in \cite{Miyazawa2009}, where the exponential decay rate of the stationary distribution for a DQBD process along the axes is determined by certain extremal points of the domains bounded by two kernels.

In this paper, we extend the kernel method to three variables by fully developing it through the derivation and exploitation of a functional equation involving the Laplace transforms of Green's functions, see \eqref{eq fonctionnelle} below. To the best of our knowledge, this is the first example of a three-variable implementation of the method, leading to a complete analytical characterization of the model.

\medskip
\textbf{Detailed outline of the article}
In Section~\ref{sec:defprocess}, we introduce the main process. In Section~\ref{subsec:defprocess2}, we define the Green's functions and state the three-variable kernel functional equation for their Laplace transforms. In Section~\ref{sec:main_results} we state our main results. In Section~\ref{sec:main_results1}, we first present  result~\ref{(i)}, providing explicit formulas for the Laplace transforms of the Green's functions in Theorem~\ref{prop:laplace explicitee}. Then, in Section~\ref{sub:theorems}, we state result~\ref{(ii)}, which provides the exact asymptotics of the Green's functions $g^{v_0}$ along all trajectories in the plane $\mathbb{R}^2$. In other words, we consider the limit of $g^{v_0}(r\cos\alpha, r\sin\alpha)$ as $r \to \infty$ and $\alpha \to \alpha_0$, where $\alpha_0 \in [0, 2\pi)$ and $v_0 \in \mathbb{R}^2$ denotes the initial state. 
Theorem \ref{alpha=0} states the asymptotics along the $x$-axis, when $\alpha=0$ or $\pi$. We obtain asymptotics of the form
$$
g^{v_0}(r, 0) \underset{r\to+\infty}{\sim} C_0 f_0(v_0)e^{-r\M}r^{-3/2}
$$
where $C_0$ is a constant, $f_0$ a harmonic function, and $x_b$ an explicit branching point of one of the two-valued functions $Y^{\pm}(x)$ or $Z^{\pm}(x)$ defined by the kernels as $\gamma_+(x, Y^{\pm}(x))=0$ and $\gamma_{-}(x, Z^{\pm}(x))=0$ respectively, see  Definition~\eqref{def:xb}.
We then highlight two key explicit angles for the asymptotic study, $$\alpha_b \text{ and } \widetilde \alpha_b \in [0,2\pi]$$
which are related to branching points of $Y^{\pm}(x)$ and $Z^{\pm}(x)$
as in Definitions~\eqref{def:alphabb} and \eqref{def:alphabbtilde}.
These angles define the domain $\mathcal M$ which takes the following form depending on the configuration (see Figure~\ref{anglesalphab} below) :
$$
\mathcal M= 
[\alpha_b, \tilde \alpha_b]\cup [\pi,2\pi]
\text{ or }
[0, \tilde \alpha_b] \cup [\pi,\alpha_b]
\text{ or }
[0,\pi]\cup [\tilde\alpha_b, \alpha_b]
\text{ or }
[\alpha_b, \pi] \cup [\tilde\alpha_b, 2\pi] .
$$
Theorem~\ref{thm:1} states the typical asymptotic result: when $\alpha_0 \in \inter{\mathcal M} $, 
we have
$$
   g^{v_0}(r\cos(\alpha), r\sin(\alpha)) \underset{r \to \infty \atop \alpha \to \alpha_0}{\sim}
   C(\alpha_0) h_{\alpha_0} (v_0) r^{-1/2}
   {e^{-r\big(\cos(\alpha)\, x(\alpha) + \sin(\alpha)\times\, \substack{y(\alpha) \text{ when } \alpha_0 \in (0,\pi) \\ \text{or} \\ 
   \, z(\alpha) \text{ when } \alpha_0 \in  (\pi,2\pi)} \big)}}
$$
where $C(\alpha)$ is a constant depending on $\alpha$, $h_\alpha$ is a harmonic function, and $x(\alpha)$, $y(\alpha)$ and $z(\alpha)$ are explicit points depending on $\alpha$. 
Theorem~\ref{thm:2} concerns the asymptotics as $\alpha\in\mathcal M$ tends to $\alpha_0=0$ or $\pi$. In this case a competition between $\alpha$ and $1/r$ appears in the asymptotics, which take the form
$$
g^{v_0}(r\cos(\alpha), r\sin(\alpha))
\underset{r\to\infty \atop\alpha\to0, \alpha \in \mathcal M}{\sim}
 C_0f_0(v_0)\left(\kappa\alpha + r^{-1} \right) r^{-1/2} {e^{-r \big( \cos(\alpha)x(\alpha) + \sin(\alpha)
 \times\, \substack{y(\alpha) \text{ when } \alpha>0 \\ \text{or} \\ 
 z(\alpha) \text{ when } \alpha<0}
 \big)}}.
$$
Theorem~\ref{thm3} gives the asymptotics for $\alpha_0= \alpha_b$ or $\tilde{\alpha}_b$, which involve a linear combination of two harmonic functions, $h_{\alpha_b}$ and $f_0$, with coefficients $C$ depending on the rate at which $\alpha$ converges to $\alpha_0$ as $r \to \infty$. We obtain an asymptotic expression of the form
$$
 g^{v_0}(r\cos(\alpha), r\sin(\alpha))
    \underset{r\to\infty \atop \alpha\to\alpha_b}{\sim} \left( C(\alpha_b) h_{\alpha_b}(v_0) + C f_0(v_0) \right) r^{-1/2}{e^{-r\big(\cos(\alpha)x(\alpha) + \sin(\alpha)
   \times\,  \substack{y(\alpha) \text{ when } \alpha_b \in  (0,\pi) \\ \text{or} \\ z(\alpha) \text{ when } \alpha_b \in  (\pi,2\pi) }
    \big)}}.
    $$
\nr{This is where the difference between the full Martin boundary and the minimal one appears.}
Theorem~\ref{thm:4} deals with the asymptotics 
when $\alpha_0\notin \mathcal M$ (or $\alpha_0 = 0, \pi$ and $\alpha \to \alpha_0$ while $\alpha \notin \mathcal M$).
We obtain an asymptotic expression of the form
$$
g^{v_0}(r\cos(\alpha), r\sin(\alpha))
\underset{r\to\infty \atop\alpha\to\alpha_0, \alpha\notin \mathcal M }{\sim}  C'_{br}(\alpha)f_0(v_0)r^{-3/2}e^{-r\big(\cos(\alpha)x(\alpha_b) + \sin(\alpha)
\times\, \substack{y(\alpha_b) \text{ when } \alpha \in  (0,\pi)   \\ \text{or} \\ z(\alpha_b) \text{ when } \alpha \in  (\pi,2\pi)  }
\big)}.
$$

\nb{The asymptotics stated in Theorems~\ref{alpha=0}--\ref{thm:2} can be readily obtained using the techniques developed in \cite{Franceschi_2024}, where a diffusion process in a cone in ${\mathbb{R}^2}$, with a single kernel in the interior and reflections at the boundaries, is studied. They are essentially proved by the saddle-point method for parameter-dependent integrals and take a form similar to that obtained in \cite{Franceschi_2024}.
The asymptotic results in Theorems~\ref{thm3} and~\ref{thm:4} 
are completely new and specific to our model with two kernels. They are derived from the analysis of integrals in which the saddle point approaches a branch point of the integrand. Theorem~\ref{thm3}, in particular, gives rise to the difference between the full and the minimal Martin boundaries.}

In Section~\ref{subsec:results_Martin}, we state result~\ref{(iii)} in Theorem~\ref{thm5}, concerning the harmonic functions, the full Martin boundary, and the minimal Martin boundary which is homeomorphic to $\mathcal M$. 

In Section~\ref{sec:process_functionalequation}, we study the process in greater detail, establish its existence and uniqueness in the case where the generator is in divergence form and prove the main functional equation. In Section~\ref{sec:alpha=0}, we prove Theorem~\ref{prop:laplace explicitee} using the functional equation from which we derive Theorem~\ref{alpha=0} 
applying Tauberian theorems; in Section~\ref{sec:proofstheorems}, we prove Theorems~\ref{thm:1}–\ref{thm:4}, the proofs use variants of the saddle-point method; and in the short Section~\ref{sec:5}, we prove Theorem~\ref{thm5} which follows from the asymptotic results. Finally, in Section~\ref{sec:skewdiff}, we extend our analysis to a two-dimensional oblique skew diffusion, where the flux is not necessarily continuous:
substantially different types of asymptotic under this assumption 
are found in Theorems \ref{thm:7}, \ref{thm:8} and \ref{th3new}.

\subsection{Planar skew bang-bang diffusion with piecewise constant coefficients}\label{sec:defprocess}
\label{subsec:defprocess1}

We now define a two-dimensional generalization of the SBBBM, which acts obliquely at the interface between two half-planes. The diffusion coefficients and the drift vector are piecewise constant, taking different values in the upper and lower half-planes.
Let $\Sigma^+, \Sigma^-$ be \nb{positive-definite} covariance matrices and $\mu^+, \mu^-$ be vectors in $\mathbb{R}^2$. We define the matrices $\Sigma(y)$ and $\sigma(y)$ and the vector $\mu(y)$ such that
$$
\Sigma(y)= \sigma(y)\sigma^\top (y) = \Sigma^- \fc_{y<0} + \Sigma^+ \fc_{y\geq0} 
\quad\text{and}\quad 
\mu(y) = \mu^- \fc_{y<0} + \mu^+ \fc_{y\geq0}.
$$  
Let $q = (q_1,q_2) \in \mathbb{R}\times(-1,1)$. This vector represents a singular drift on the axis $\{y = 0\}$. The parameter $q_2=2\beta-1$ corresponds to the classical skewness intensity at the boundary, when the process hits the axis $y = 0$, it restarts on the $y > 0$ side with probability $\beta=(1+q_2)/2$, and on the $y < 0$ side with probability $1-\beta=(1-q_2)/2$. The parameter $q_1$ fixes the singular drift component along the axis and thus introduces the oblique direction of the skew behaviour at the interface. 
\begin{defi}[Planar skew bang-bang diffusion with piecewise constant coefficients]
We define a planar skew bang-bang diffusion associated with coefficients $(\Sigma^+, \Sigma^-, \mu^+, \mu^-)$ and with the skew vector $q$, as a continuous adapted process $(V_t)_{t\geq0} = (A_t, B_t)_{t\geq0}$ on $\mathbb{R}^2$, 
    such that $V$ can be expressed 
    almost surely as
    \begin{equation}\label{EDS}
        V_t= (A_t,B_t) = (a_0,b_0) + \int_0^t {\sigma}(B_s)\, dW_s + \int_0^t \mu(B_s)\, ds + q L^0_t(B),\quad t \geq 0,
    \end{equation}
    where:
    \begin{enumerate}
    \item[(i)] $v_0=(a_0,b_0)\in\mathbb{R}^2$ is the starting point;
        \item[(ii)] $(W_t)_{t\geq0}$ is an adapted two-dimensional standard Brownian motion;
        \item[(iii)] 
      $(L^0_t(B))_{t\geq0}$ is the symmetric local time at $0$ of $B$ defined by
        \begin{equation}
            L^0_t(B) = \lim_{\varepsilon \to 0} \frac{1}{2\varepsilon} \int_0^t \fc_{[-\varepsilon,\varepsilon]}(B_s)\, d \langle B\rangle_s .
        \end{equation}
    \end{enumerate}
\end{defi}

\begin{figure}[t]
    \centering
    \includegraphics[width=0.31\linewidth]{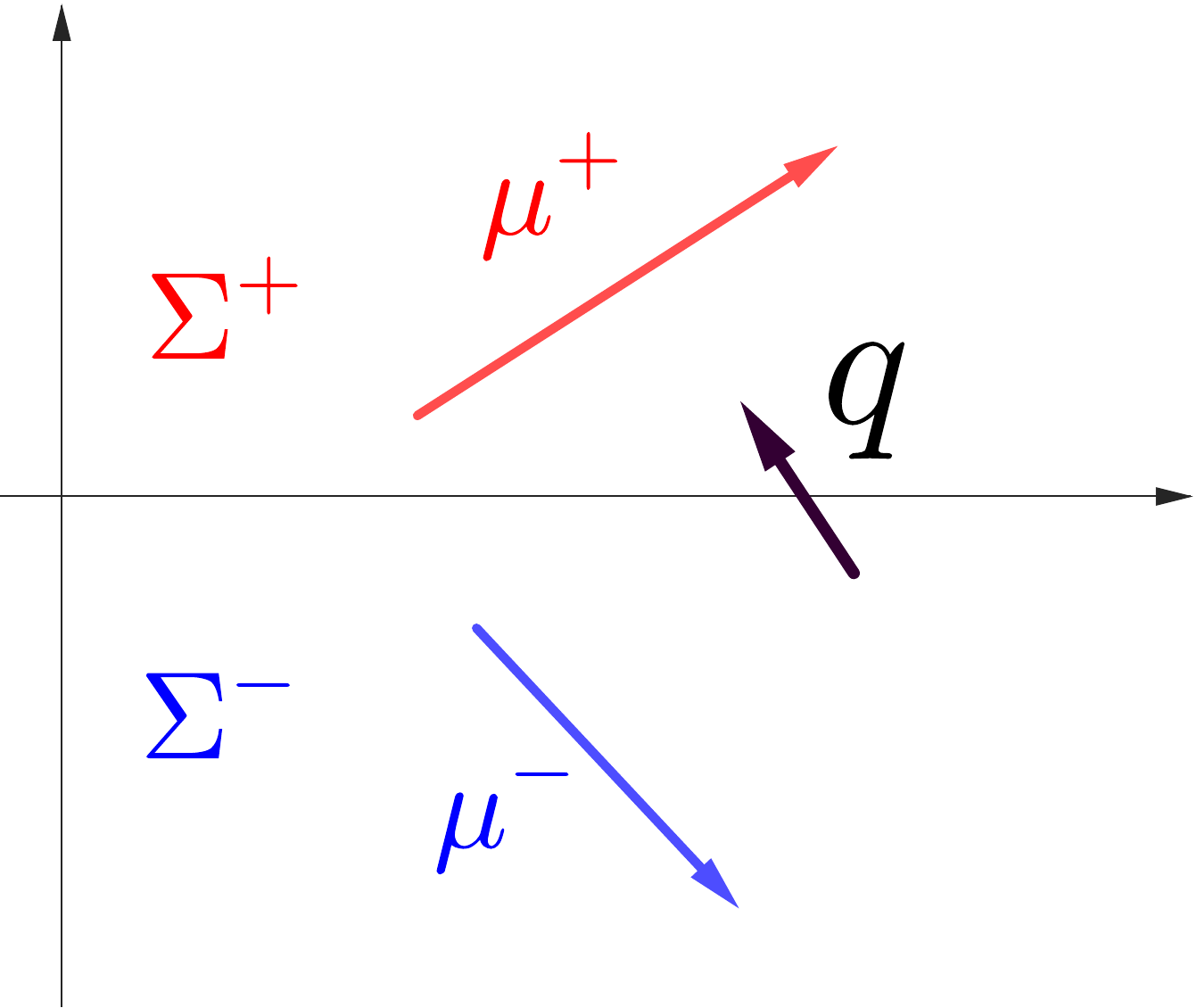}
    \includegraphics[width=0.3\linewidth]{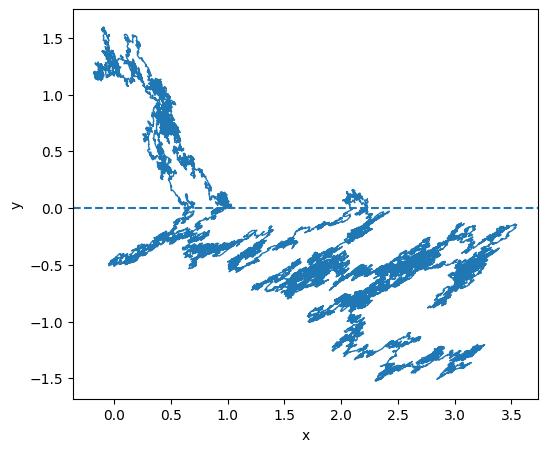}
    \caption{Parameters and path of the process. Locally, the process $V$ behaves like a planar Brownian motion with covariance $\Sigma^+$ (resp. $\Sigma^-$) and drift $\mu^+$ (resp. $\mu^-$) in the upper (resp. lower) half-plane. When the process hits the axis $\{y = 0\}$, it is subject to a singular drift in the direction of $q$.}
    \label{parametres}
\end{figure}

Let $q_0$ be the vector defined by
\begin{equation}\label{eq:q}
q_0 = 
\left(
\frac{\Sigma_{12}^+ - \Sigma_{12}^-}{\Sigma_{22}^+ + \Sigma_{22}^-} 
,
\frac{\Sigma_{22}^+ - \Sigma_{22}^-}{\Sigma_{22}^+ + \Sigma_{22}^-} 
\right).
\end{equation}
In all sections of the article, except Section~\ref{sec:skewdiff}, we will assume that
$$
q=q_0.
$$
In this case, the generator, understood in the sense of Strook \cite[Theorem II.3.1]{Stroock1988}, can be written in divergence form; see \eqref{divergence-form} below. This form is relevant in physical applications~\cite{Lejay2011}. 
The transmission conditions then ensure the continuity of the probability flux across the boundary, see the partial differential equation \eqref{eq:pdeharm}. 
The following proposition is proved in Section~\ref{sub:existence}.
\begin{proposition}[Existence, uniqueness and divergence form]\label{prop:existence_deuxnoyaux}
There exists a unique solution $V_t$ (both pathwise and in law) to the stochastic differential equation~\eqref{EDS} associated with the skew vector $q=q_0$.
This process defines a strong Markov process 
and its generator can be expressed as follows:
\begin{equation}\label{divergence-form}
    \mathcal{L} = \frac{1}{2}\nabla\cdot (\Sigma \nabla) + \mu\cdot\nabla .
\end{equation}
\end{proposition}

\nr{For the remainder of this article, we will work under the following hypothesis regarding the drift:}\nb{For the remainder of the paper, we assume that the process $(V_t)_{t\geq0}=((A_t, B_t))_{t\geq0}$ is transient.
It may escape to infinity almost surely in several ways depending on the parameters : via a single half-plane if $\mu _{2}^+$ and $\mu _{2}^-$ have the same sign;
by hitting the $x$-axis at arbitrarily large times (which could happen under $\mu_2^+ < 0<\mu_2^-$);
or via one half-plane or the other with positive complementary probabilities, where the process no longer touches the $x-$axis after some finite time if $\mu_2^+>0$, $\mu_2^-<0$.
Although the method presented in this paper conceptually works in all these cases, because of space constraints, we restrict our attention to the last one.
Therefore, we assume throughout this article that:}
\begin{equation}\label{hyp:drift}
   \nr{ \mu_1^+ > 0, \quad} \mu_2^+ > 0\quad \nr{\mu_1^- > 0,\quad}\text{and}\quad \mu_2^- < 0,
\end{equation}
see Figure \ref{parametres}. \nr{In particular, this hypothesis means that the process $V_t=(A_t,B_t)$ is transient and eventually no longer touches the horizontal axis.} 
\nb{Let us note that the second coordinate $(B_{t})_{t\geq0}$ of the two-dimensional process $(V_t)_{t\geq0}=((A_t,B_t))_{t\geq0}$ is a one-dimensional skewed bang-bang process with parameter $q_{2}$ at $0$. Under \eqref{hyp:drift}, it has a positive drift on the nonnegative real axis and a negative drift on the negative real axis. This one-dimensional process is studied in full detail in Appendix~\ref{annexe:dim1}, see also \cite[Section 1.2.2.2]{Petit2026-fc}, where the Green functions are made explicit and their exact asymptotics are computed. Furthermore, it escapes to $+\infty$ or $-\infty$ with positive complementary probabilities found in Corollary~\ref{cor:escape+-infty}.}
\nr{The process $B$ is a one dimensional skewed bang-bang process of parameter $q_2$ at $0$ with positive drift on the nonnegative real axis and negative drift on the negative real axis.}

\subsection{Green's functions and functional equation}\label{subsec:defprocess2}

We have just seen that, when $q=q_0$, the generator can be expressed in divergence form. From Strook \cite[Theorem II.3.8]{Stroock1988}, we deduce that for all $t>0$, the law of $V_t$ has a density $p_t^{v_0}(v)$ with respect to the Lebesgue measure on $\R^2$. Furthermore, the function $(t,v_0,v) \longmapsto p_t^{v_0}(v)$ is continuous on $(0, +\infty) \times \R^2 \times \R^2$. 
\begin{defi}[Green's measures and Laplace transforms]
For $v_0=(a_0,b_0) \in \R^2$, we define the Green's measure on the plane $\R^2$ by
\begin{equation}
G(v_0,S) = \E_{v_0}\left[\int_0^{+\infty}\fc_{S}(V_t)\,dt\right] = \iint_S g^{v_0}(a,b)\,da \,db, \quad \forall S \in \mathcal{B}(\R^2)
\end{equation}
where its density $g^{v_0}$ is called the Green's function and is equal to
\begin{equation}
g^{v_0}(a,b) = \int_0^{+\infty} p_t^{v_0}(a,b)\,dt, \quad \forall (a,b)\in\mathbb{R}^2.
\end{equation}
For $x, y, z \in \R$, we also define the Laplace transforms of $G(v_0,\cdot)$ restricted to the half-planes $\R \times (0,+\infty)$ and $\R \times (-\infty,0)$ by
$$
\phi_+^{v_0}(x,y) = \E_{v_0}\left[\int_0^{+\infty} e^{xA_t + yB_t} \fc_{B_t > 0}\,dt\right]
=\int_{\mathbb{R}}\int_0^\infty g^{v_0}(a,b)e^{xa+yb} 
\,da \,db, 
$$
$$
\phi_-^{v_0}(x,z) = \E_{v_0}\left[\int_0^{+\infty} e^{xA_t + zB_t} \fc_{B_t < 0}\,dt\right]
=\int_{\mathbb{R}}\int_{-\infty}^0 g^{v_0}(a,b)e^{xa+zb} 
\,da \,db.
$$
Finally, we define the Green's measure on the axis $\R \times \{0\}$  by
\begin{equation}\label{eq:H(z0,A)}
H(v_0, S) = \E_{v_0}\left[\int_0^{+\infty} \fc_S(A_t)\,dL^0_t(B)\right], \quad \forall S \in \mathcal{B}(\R)
\end{equation}
and its Laplace transform $\phi^{v_0}$ as
\begin{equation}\label{def:phi}
\phi^{v_0}(x) = \E_{v_0}\left[\int_0^{+\infty} e^{xA_t}\,dL^0_t(B)\right].
\end{equation}
\end{defi}

For $x, y, z \in \C$, we set:
\begin{equation}\label{kernels}
\begin{cases}
        \gamma_+(x, y) = \frac{1}{2}(x, y)\cdot\Sigma^+(x, y)^T + (x, y)\cdot\mu^+ = \frac{1}{2}(\Sigma^+_{11}x^2 + 2\Sigma^+_{12}xy + \Sigma_{22}^+y^2) + \mu_1^+ x + \mu^+_2 y\\
        \gamma_-(x, z) = \frac{1}{2}(x, z)\cdot\Sigma^-(x, z)^T + (x, z)\cdot\mu^- = \frac{1}{2}(\Sigma^-_{11}x^2 + 2\Sigma^-_{12}xz + \Sigma_{22}^-z^2) + \mu_1^- x + \mu^-_2 z\\
\gamma(x,y,z) = q_1x + \frac{1}{2}(y(1 + q_2) + z(q_2 - 1)).
    \end{cases}
\end{equation}
These polynomials are the coefficients of the following three-variable kernel functional equation, which is at the heart of this article. The following proposition is proven in Sections~\ref{sub:formalproof} and~\ref{sub:detailedproof}.
\begin{proposition}[Functional Equation]\label{prop:eqfonc}
Suppose that \eqref{hyp:drift} holds. 
Then, \nr{there exists $\eta > 0$ such that for all $x \in (-\eta, 0)$}\nb{for all $x$, $y$ and $z$ with $Re(x) = 0$,} $Re(y) < 0$, and $Re(z) > 0$, the Laplace transforms $\phi^{v_0}_-(x, z)$, $\phi^{v_0}_+(x, y)$, and $\phi^{v_0}(x)$ are finite and satisfy the following functional equation:
\begin{equation}\label{eq fonctionnelle}
\gamma_-(x,z)\phi^{v_0}_-(x,z) + \gamma_+(x,y)\phi^{v_0}_+(x,y) +\gamma(x,y,z)\phi^{v_0}(x) = -e^{xa_0 + yb_0\fc_{b_0> 0} + zb_0 \fc_{b_0<0}}.
\end{equation}
\end{proposition}
The polynomials $\gamma_+(x,y)$ and $\gamma_-(x,z)$ are called the kernels. We are now going to define some important quantities that will be useful for stating the results.
We first determine the complex branches $Y^\pm(x)$ and $Z^\pm(x)$ which satisfy $\gamma_+(x,Y^\pm(x))=0$ and $\gamma_-(x,Z^\pm(x))=0$. By elementary considerations, these branches are given by
\begin{equation}\label{Ypm}
Y^{\pm} (x)= \frac{1}{\Sigma^+_{22}}\Big(-\Sigma^+_{12} x -\mu^+_2 \pm \sqrt{({\Sigma^+_{12}}^2 -\Sigma^+_{11} \Sigma^+_{22}) x^2 + 2  (\mu^+_2 \Sigma^+_{12}- \mu^+_1 \Sigma^+_{22})x + {\mu^+_2}^2} \Big)
\end{equation}
and
\begin{equation}\label{Zpm}
Z^\pm(x) = \frac{1}{\Sigma^-_{22}}\Big(-\Sigma^-_{12} x -\mu^-_2 \pm \sqrt{({\Sigma^-_{12}}^2 -\Sigma^-_{11} \Sigma^-_{22}) x^2 + 2  (\mu^-_2 \Sigma^-_{12}- \mu^-_1 \Sigma^-_{22})x + {\mu^-_2}^2} \Big).
\end{equation}
Furthermore, $Y^\pm(x)$ have branching points
$x^+_{min} < 0$ and $x^+_{max} > 0$ given by
\begin{equation*}
x^+_{min} = \frac{\mu^+_2\Sigma^+_{12} - \mu^+_1\Sigma^+_{22} - \sqrt{D_1}}{\det(\Sigma^+)}, \quad x^+_{max} = \frac{\mu^+_2\Sigma^+_{12} - \mu_1^+\Sigma^+_{22} + \sqrt{D_1}}{\det(\Sigma^+)}
\end{equation*}
with $D_1 = (\mu^+_2\Sigma^+_{12} - \mu_1{\Sigma^+_{22}}^2)^2 + {\mu^+_2}^2\det(\Sigma^+)$. Similarly, $Z^\pm(x)$ have branching points
$x^-_{min} < 0$ and $x^-_{max} > 0$ with symmetric formulas. Finally, we denote the maximin and the minimax values of the branching points as follows:
\begin{equation}\label{def:xb}
    \m= \max(x^+_{min}, x^-_{min}) < 0, \quad \M = \min(x^+_{max}, x^-_{max}) > 0 .
\end{equation}
These points are illustrated in Figure \ref{fig:rrr}.
\begin{figure}[t]
\centering
\begin{subfigure}[b]{0.45\linewidth}
  \centering
  \includegraphics[width=\linewidth]{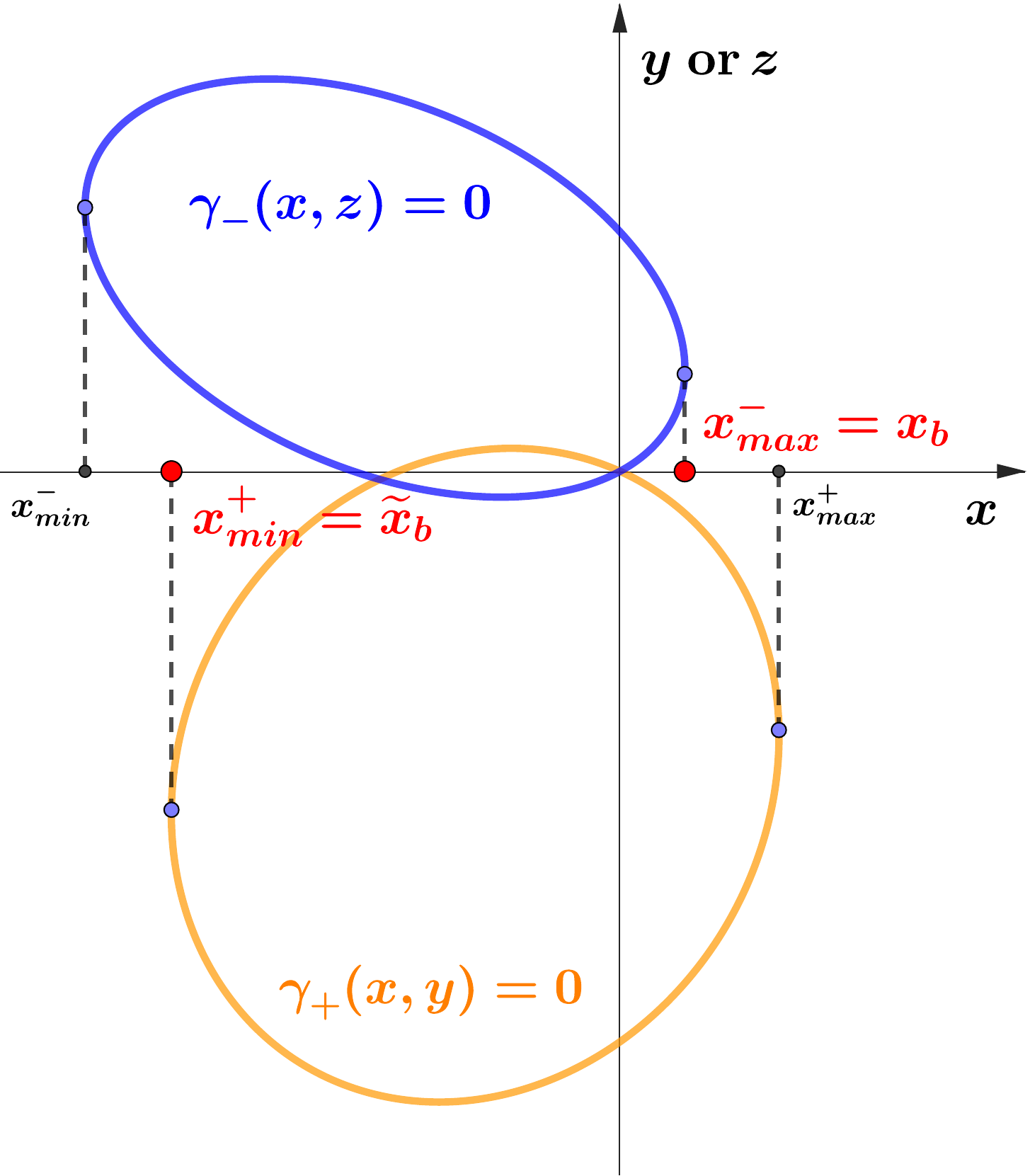}
  \caption{Branching points $\m$ and $\M$.}
  \label{111}
\end{subfigure}
\hfill
\begin{subfigure}[b]{0.45\linewidth}
  \centering
  \includegraphics[width=\linewidth]{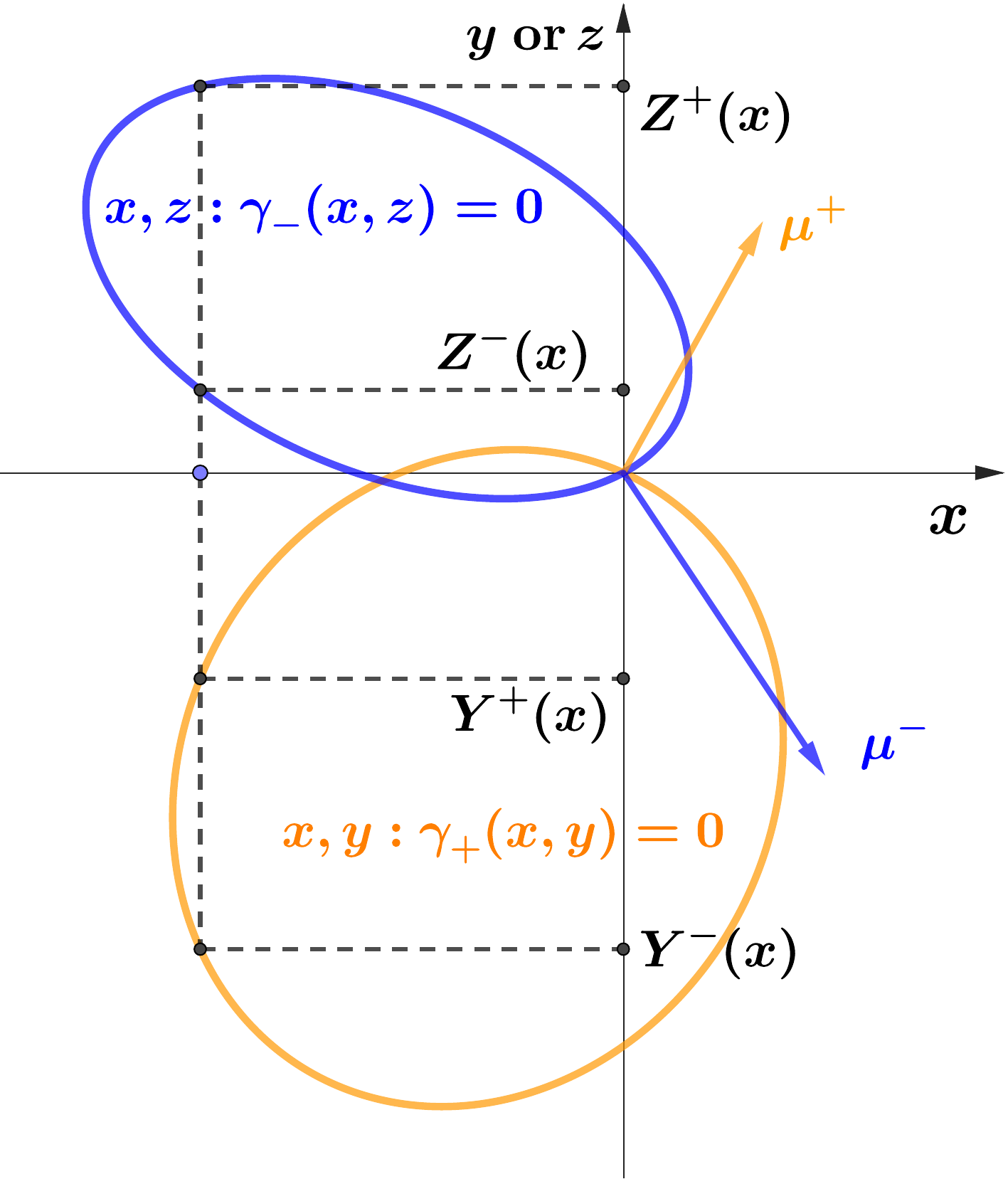}
  \caption{Branches $Y^\pm(x)$ and $Z^\pm(x)$.}
  \label{Z+Y-}
\end{subfigure}
\caption{Ellipses $\{(x,y)\in\mathbb{R}^2:\gamma_+(x,y) = 0\}$ and $\{(x,z)\in\mathbb{R}^2:\gamma_-(x,z) = 0\}$.}
\label{fig:rrr}
\end{figure}

\subsection{Main results} \label{sec:main_results}

\subsubsection{Explicit expression of the Laplace transforms}
\label{sec:main_results1}

The first theorem provides explicit formulas for the Laplace transforms $\phi^{v_0}$ and $\phi^{v_0}_\pm$. The proof is given in Section~\ref{sec:alpha=0} by substituting $y = Y^-(x)$ and $z = Z^+(x)$ into the functional equation \eqref{eq fonctionnelle} to cancel the kernels.
\begin{theo}\label{prop:laplace explicitee}
The Laplace transform $\phi^{v_0}(x)$ \nb{converges on the strip $\{x \in \C\mid \m < Re(x) < \M\}$}. It extends to a holomorphic function on $\C \backslash \big( (-\infty,\m] \cup [\M, +\infty) \big)$ and can be expressed as follows:
\begin{equation}\label{laplace explicitee}
\phi^{v_0}(x) = -\frac{e^{xa_0 + Y^-(x)b_0\fc_{b_0> 0} +  Z^+(x)b_0\fc_{b_0<0}}}{\gamma(x,Y^-(x), Z^+(x))}.
\end{equation}
The points $\M$ and $\m$ are the branching points of $\phi^{v_0}$.
Furthermore\nb{, for all $x \in i\R$ and $Re(y) < 0$, the Laplace transform $\phi^{v_0}_+(x,y)$ converges and is given by:}
\begin{equation}\label{laplace explicite}
\phi^{v_0}_+(x,y) = \frac{-\gamma(x,y,Z^+(x))\phi^{v_0}(x) - e^{xa_0 + yb_0 \fc_{b_0 > 0} + Z^+(x) b_0 \fc_{b_0 < 0}}}{\gamma_+(x,y)}.
\end{equation} 
Similarly, \nr{a symmetric expression holds for $\phi^{v_0}_-$.}\nb{if $x \in i\R$ and $Re(z) > 0$, $\phi^{v_0}_-(x,z)$ converges and is given by:  
\begin{equation}
\phi^{v_0}_-(x,z) = \frac{-\gamma(x,Y^-(x),z)\phi^{v_0}(x) - e^{xa_0 + Y^-(x)b_0 \fc_{b_0 > 0} + z b_0 \fc_{b_0 < 0}}}{\gamma_-(x,z)}.
\end{equation} }
\end{theo}

\nb{\noindent\textbf{Remark.} (Convergence of $\phi_+(x,y)$ and $\phi_-(x,z)$ for $x$ in a strip). 
The Laplace transform $\phi_+(x,y)$ may be holomorphically extended (and convergent)  on a strip $\{-\eta \leq Re(x) \leq 0\} \times \{Re(y) < 0\}$ (resp. $\{0 \leq Re(x) \leq\eta\}\times \{Re(y) < 0\}$) if $\mu^+_1 > 0$ (resp.  $\mu^+_1 < 0$) with $\eta > 0$. However, the convergence cannot hold on the whole strip $-\eta <Re(x) < \eta$. The same remark applies to $\phi_-(x,z)$.
}

\subsubsection{Asymptotics of Green's functions}\label{sub:theorems}
Before stating the main asymptotic theorems, a bit more notation needs to be introduced.
For $\alpha \in (0, \pi)$, we define
\begin{equation}\label{xalphayalpha}
    (x(\alpha), y(\alpha)) = \underset{\gamma_+(x,y) = 0}{\textnormal{argmax}} \left(\cos(\alpha) x + \sin(\alpha) y \right),
\end{equation}
and for $\alpha \in (\pi, 2\pi)$, we set
\begin{equation}
    (x(\alpha), z(\alpha)) = \underset{\gamma_-(x,z) = 0}{\textnormal{argmax}} \left(\cos(\alpha) x + \sin(\alpha) z \right).
\end{equation}
See Figure~\ref{fig:col} for a geometric interpretation. In the rest of this article, these will be the key saddle points we look at.
Note that the functions
\[
\alpha \in (0, \pi) \longmapsto (x(\alpha), y(\alpha)) \in \{(x,Y^+(x)) \;:\; x^+_{min} < x < x^+_{max}\}
\]
and
\[
\alpha \in (\pi, 2\pi) \longmapsto (x(\alpha), z(\alpha)) \in \{(x,Z^-(x)) \;:\; x^-_{min} < x < x^-_{max}\}
\]
are $C^\infty$ diffeomorphisms.  
Note also that $\mu^+$ is orthogonal to the ellipse $\{(x,y) \;:\; \gamma_+(x,y) = 0\}$ at $(0,0)$, see Figure~\ref{Z+Y-}. Hence, the directions \nb{\[
(\cos\alpha_{\mu^\pm},\sin\alpha_{\mu^\pm})
=\frac{\mu^\pm}{\|\mu^\pm\|},
\qquad
\alpha_{\mu^+}\in(0,\pi),\quad
\alpha_{\mu^-}\in(\pi,2\pi).
\] 
satisfy $(x(\alpha_{\mu^\pm}), y(\alpha_{\mu^\pm})) = (0,0).$}
\nr{$\alpha_{\mu^+}\nr{ = \arctan(\mu_2^+ / \mu_1^+) \in (0, \pi/2)}$ of the drift $\mu^+$ satisfies
$(x(\alpha_{\mu^+}), y(\alpha_{\mu^+})) = (0,0).$
Similarly, {if $\alpha_{\mu^-} := \arctan(\mu_2^- / \mu_1^-) + 2\pi \in (3\pi/2, 2\pi)$
is }the direction \nb{$\alpha_{\mu^-}$} of the drift $\mu^-${, then} {satisfies} $
(x(\alpha_{\mu^-}), z(\alpha_{\mu^-})) = (0,0).$}

\begin{figure}[t]
\centering
\begin{subfigure}[b]{0.45\linewidth}
  \centering
  \includegraphics[width=\linewidth]{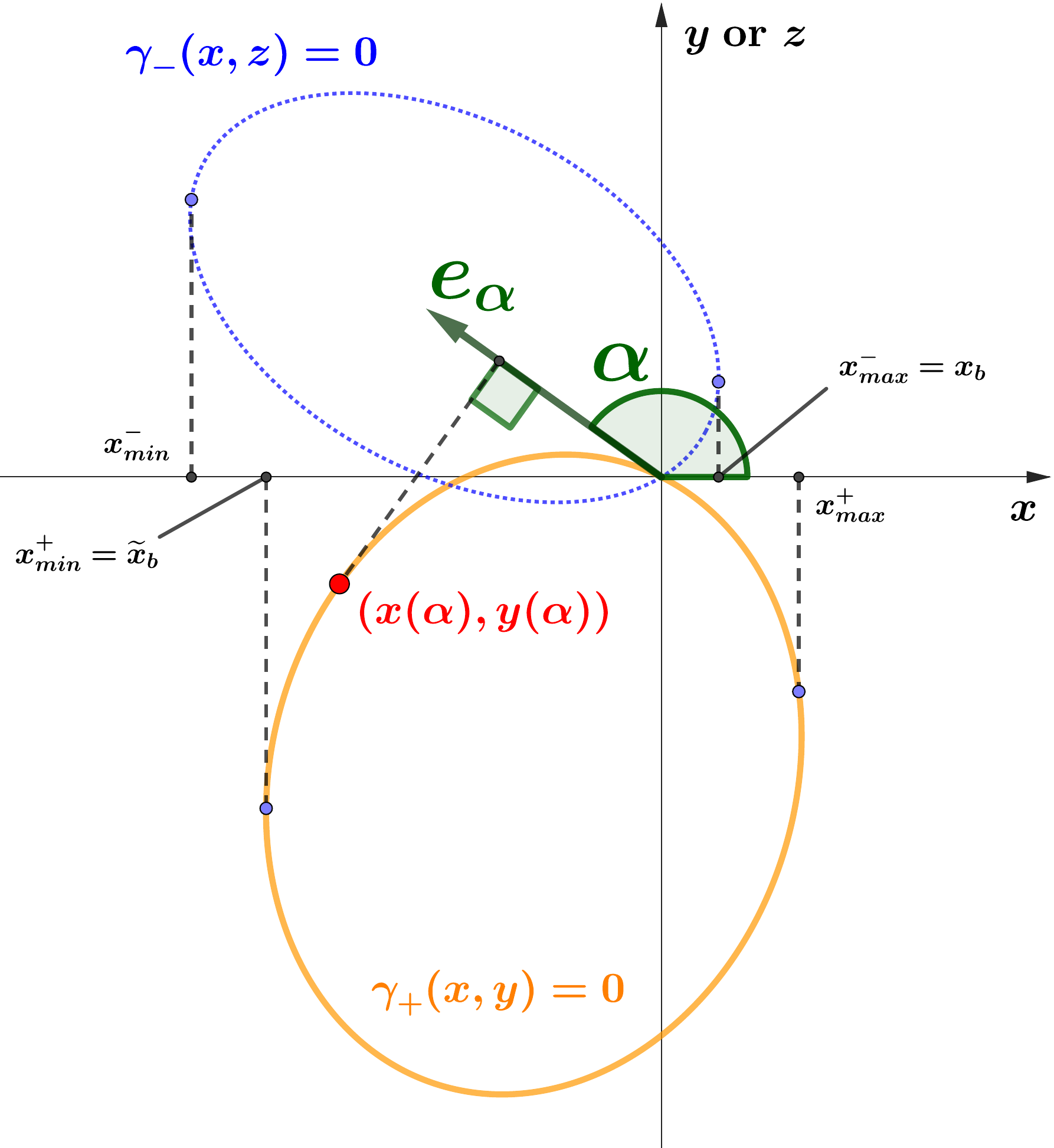}
  \caption{If $\alpha \in (0,\pi)$, the construction is done on the ellipse $\{(x,y)\,:\,\gamma_+(x,y)=0\}$.}
\end{subfigure}
\hfill
\begin{subfigure}[b]{0.45\linewidth}
  \centering
  \includegraphics[width=\linewidth]{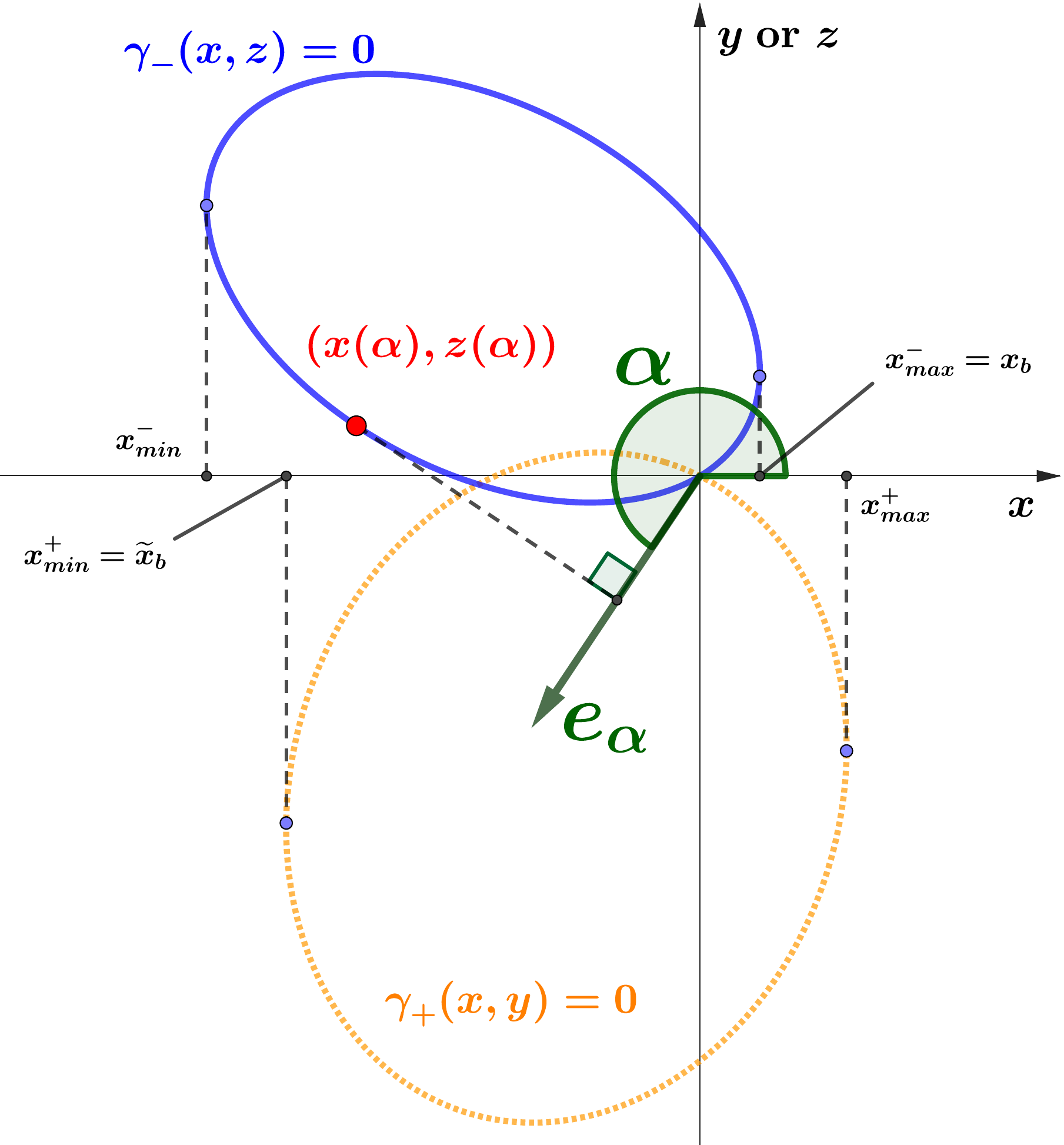}
  \caption{If $\alpha \in (\pi,2\pi)$, the construction is done on the ellipse $\{(x,z)\,:\,\gamma_-(x,z)=0\}$.}
\end{subfigure}
\caption{Geometric construction of $(x(\alpha), y(\alpha))$ and $(x(\alpha), z(\alpha))$. The vector $e_\alpha$ is defined by $e_\alpha = (\cos(\alpha), \sin(\alpha))$.}
\label{fig:col}
\end{figure}

If $x_{max}^+ > x_{max}^-$, 
we define $\alpha_b$ as the unique angle $\alpha_b \in (0,\pi)$ 
such that
\begin{equation}\label{def:alphabb}
    (x(\alpha_b), y(\alpha_b)) = (\M, Y^+(\M)), 
\end{equation}
see Figure~\ref{fig:defalphabA}. In this case, we have $0 < \alpha_b < \alpha_{\mu^+}$ since $0 = x(\alpha_{\mu^+}) < \M$.

If $x_{max}^+ < x_{max}^-$, we define $\alpha_b$ as the unique angle $\alpha_b \in (\pi, 2\pi)$ such that
\begin{equation}\label{def:alphabbtilde}
    (x(\alpha_b), z(\alpha_b)) = (\M, Z^-(\M)). 
\end{equation}
In this case, we have $\alpha_{\mu^-} < \alpha_b < 2\pi$ since $0 = x(\alpha_{\mu^-}) < \M$.

\begin{figure}[t]
\centering
\begin{subfigure}[b]{0.45\linewidth}
  \centering
  \includegraphics[width=\linewidth]{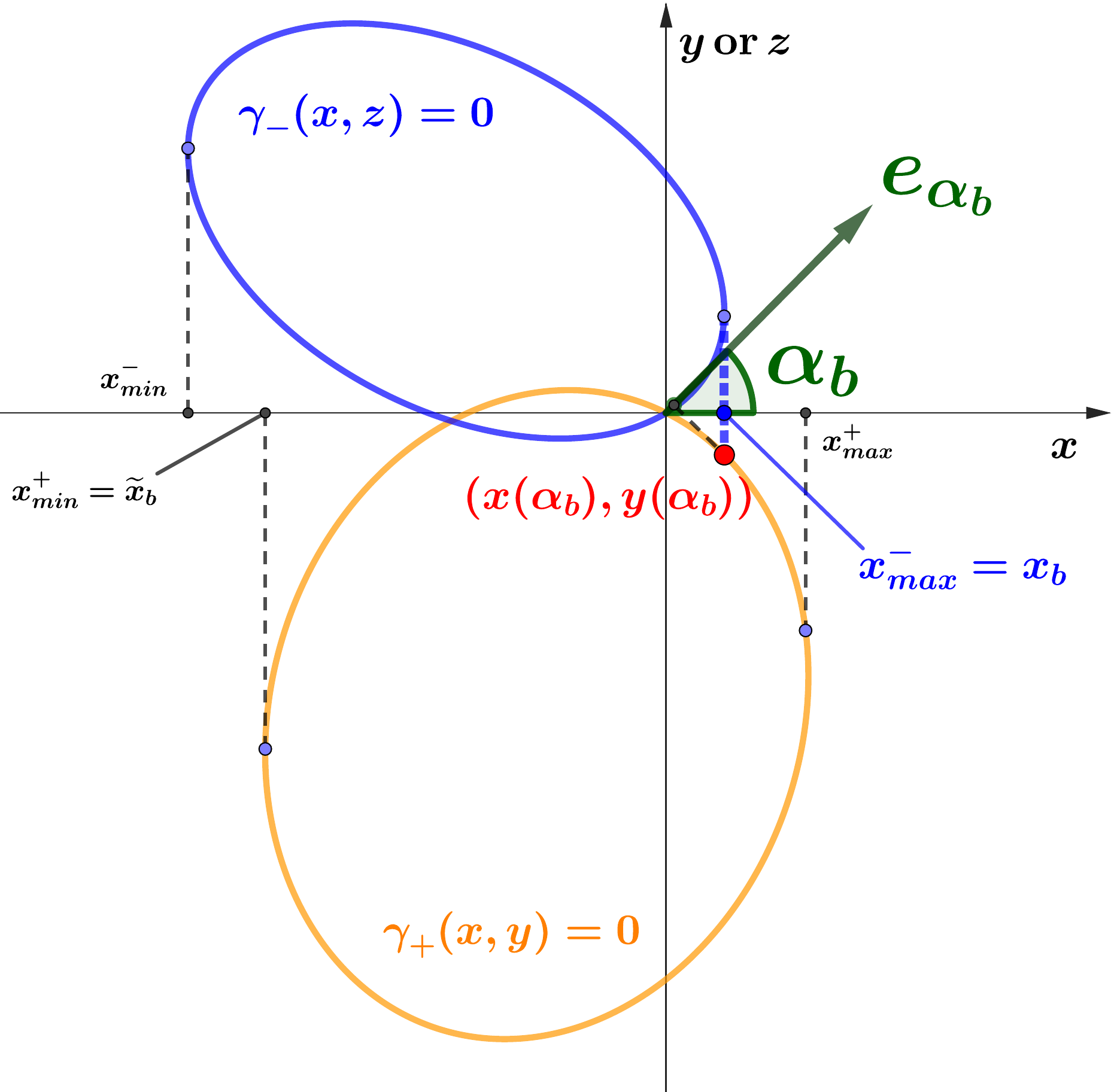}
  \caption{Definition of $\alpha_b$: here $\M = x_{max}^- = x(\alpha_b)$.}
  \label{fig:defalphabA}
\end{subfigure}
\hfill
\begin{subfigure}[b]{0.45\linewidth}
  \centering
  \includegraphics[width=\linewidth]{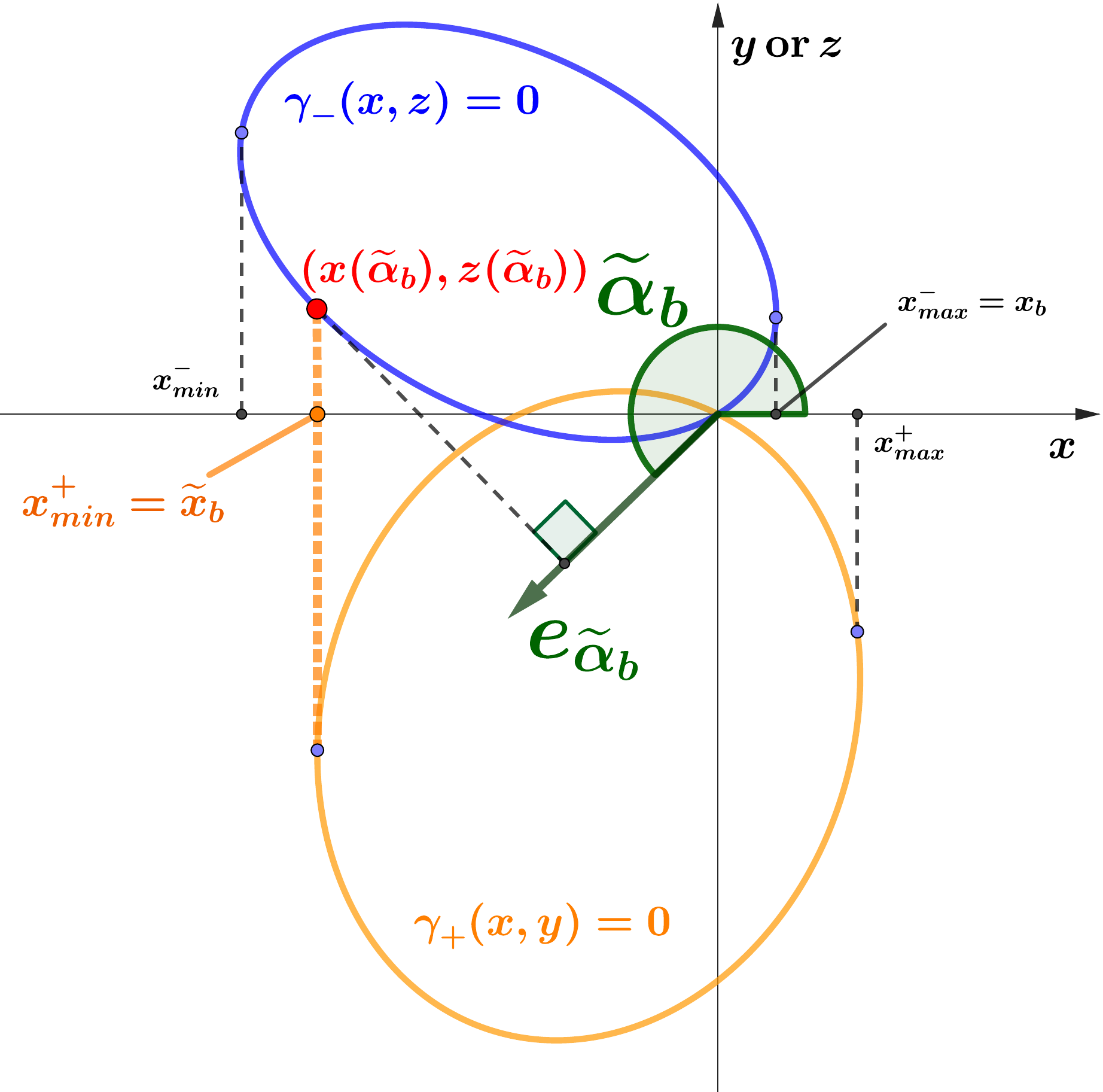}
  \caption{Definition of $\tilde\alpha_b$: here $\m = x_{min}^+ = x(\tilde\alpha_b)$.}
  \label{fig:defalphabB}
\end{subfigure}
\caption{Definitions of $\alpha_b$ and $\tilde \alpha_b$ in the cases $\M = x_{max}^-$ and $\m = x_{min}^+$.}
\end{figure}

Next, considering $\m$, we similarly define $\tilde{\alpha}_b$ by the condition $x(\tilde{\alpha}_b) = \m$, where $\tilde{\alpha}_b$ belongs to $(0,\pi)$ if $x_{min}^+ < x_{min}^-$ and to $(\pi, 2\pi)$ if $x_{min}^+ > x_{min}^-$, see Figure~\ref{fig:defalphabB}. Figure~\ref{anglesalphab} provides an overview of all possible cases: $A, B, C,$ and $D$. Then, we define 
$$
\mathcal M=\{\alpha\in [0,2\pi] :\m \leqslant x(\alpha) \leqslant \M \}
$$
which can be expressed as follows for the cases shown in Figure~\ref{anglesalphab}:
$$
\mathcal M = \begin{cases} 
[\alpha_b, \tilde \alpha_b]\cup [\pi,2\pi] \quad &\textnormal{in the case}\; A :\quad x_{max}^+ > x_{max}^- \;\;\textnormal{and}\;\;
x_{min}^+ < x_{min}^- ,\\
[0, \tilde \alpha_b] \cup [\pi,\alpha_b]\quad &\textnormal{in the case}\; B : \quad x_{max}^+ < x_{max}^- \;\;\textnormal{and}\;\;
x_{min}^+ < x_{min}^-  ,\\
[0,\pi]\cup [\tilde\alpha_b, \alpha_b] \quad &\textnormal{in the case}\; C : \quad x_{max}^+ < x_{max}^- \;\;\textnormal{and}\;\;
x_{min}^+ > x_{min}^-  ,\\
[\alpha_b, \pi] \cup [\tilde\alpha_b, 2\pi] \quad &\textnormal{in the case}\; D :\quad x_{max}^+ > x_{max}^- \;\;\textnormal{and}\;\;
x_{min}^+ > x_{min}^-. 
\end{cases}
$$
The set $\mathcal M$ is represented by the red arcs in Figure~\ref{anglesalphab}. As we will see in Theorem~\ref{thm5}, this is the minimal Martin boundary.

\begin{figure}[t]
\centering
\begin{tabular}{|c|c|c|c|}
\hline
Case A&
Case B&
Case C&
Case D\\

\hline

$x_{max}^+ > x_{max}^- $ & $x_{max}^+ < x_{max}^-$ & $x_{max}^+ < x_{max}^-$ & $x_{max}^+ > x_{max}^-$ \\

$x_{min}^+ < x_{min}^-$ & $x_{min}^+ < x_{min}^-$ & $x_{min}^+ > x_{min}^-$ & $x_{min}^+ > x_{min}^-$ \\
\hline
& & & \\
\includegraphics[scale=0.35]{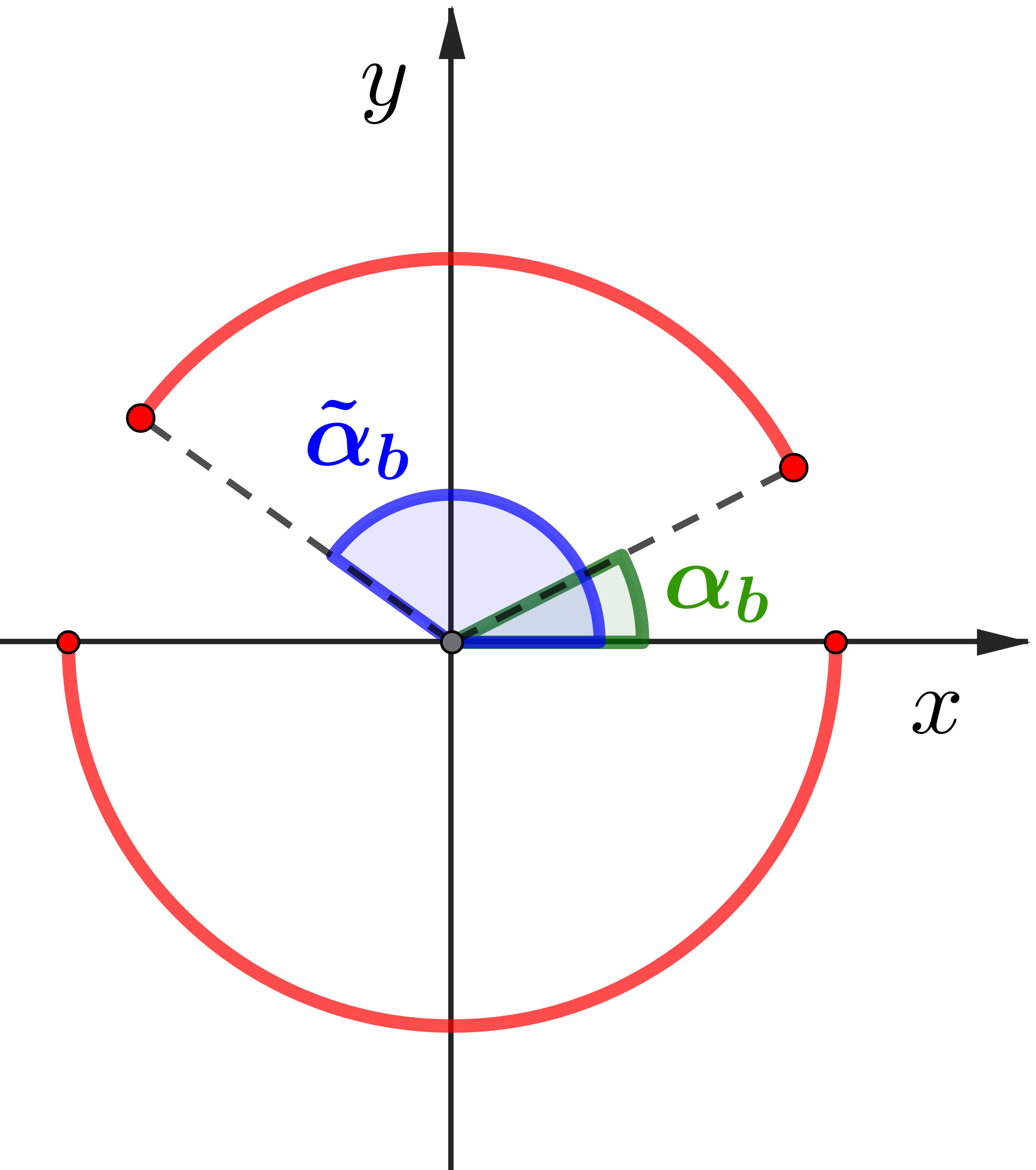} & \includegraphics[scale=0.35]{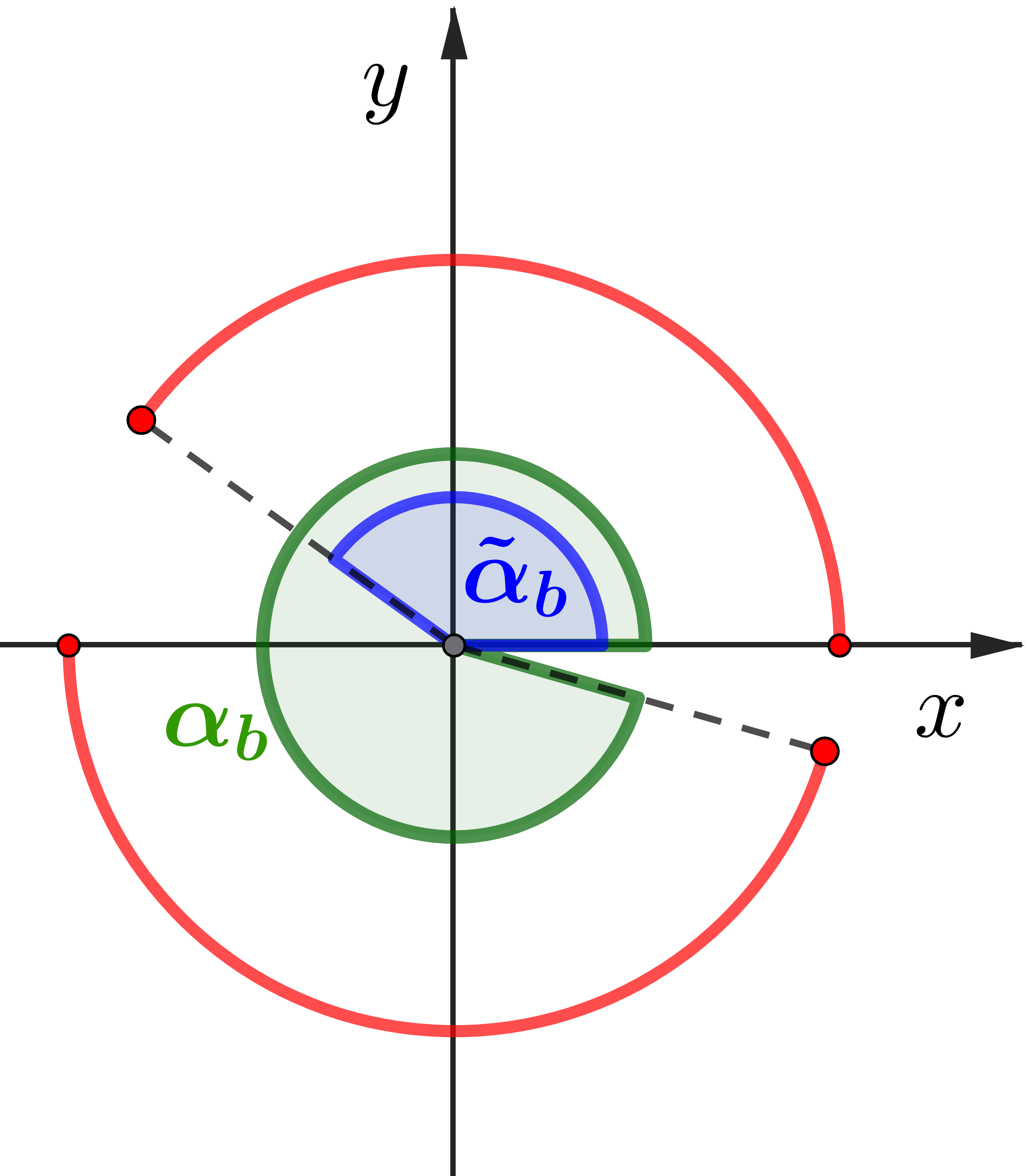} & \includegraphics[scale=0.35]{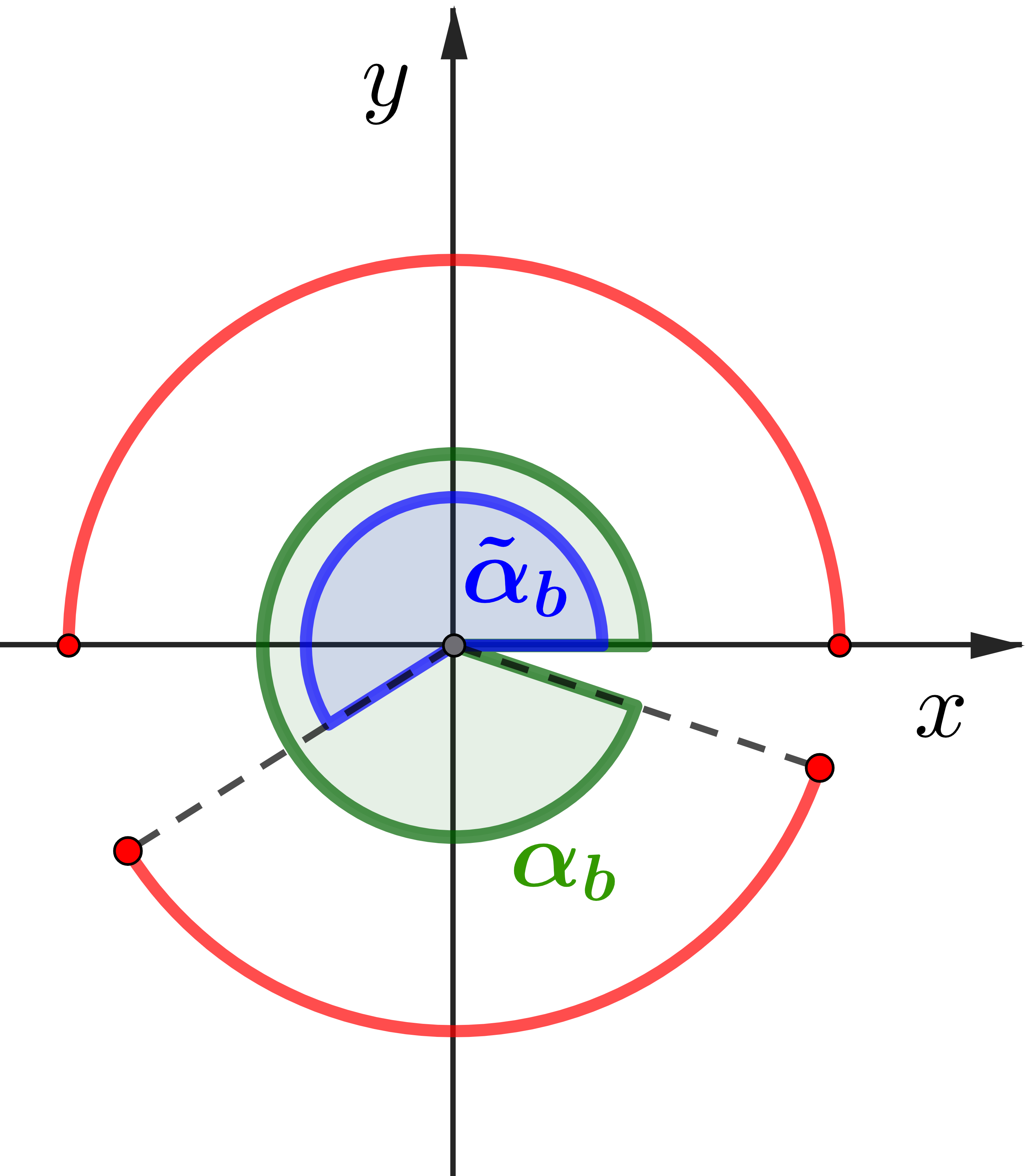} & \includegraphics[scale=0.35]{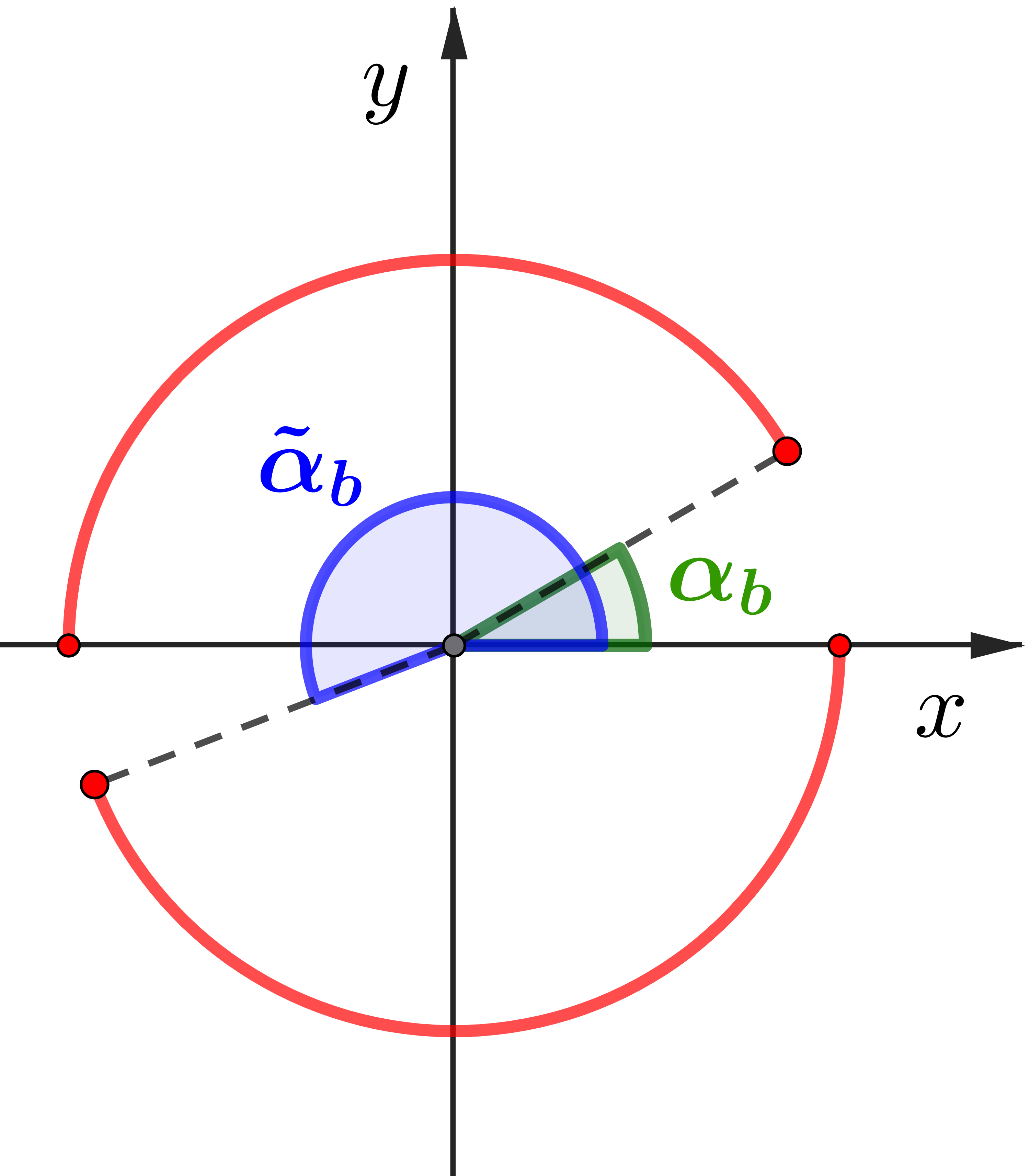} \\
\hline
\end{tabular}
\caption{Angles $\alpha_b$ and $\tilde{\alpha}_b$ depending on the parameters. The red arcs $\mathcal M$ correspond to the minimal Martin boundary.}
\label{anglesalphab}
\end{figure}

\textbf{Main asymptotics theorems.}
We now state the asymptotic behaviour of the Green's functions, $g^{v_0}(r\cos(\alpha), r\sin(\alpha))$, as $r \to +\infty$ and $\alpha \to \alpha_0$. Figure~\ref{fig:soleil} summarises the different cases covered by each theorem and indicates the associated colours.
In Theorem~\ref{alpha=0} (brown arrows), we consider the case where the angle $\alpha$ is fixed at $\alpha_0 \in \{0, \pi\}$. In Theorem~\ref{thm:1} (red arrows), we focus on directions where $\alpha \to \alpha_0$ with $\alpha_0 \in \inter{\mathcal M}$. In Theorem~\ref{thm:2} (purple arrows), we study the case when $\alpha \to 0$ or $\pi$ but remaining within the red arc zones $\inter{\mathcal M}$. Theorem~\ref{thm3} (blue arrows) treats the asymptotic behavior as $\alpha \to \alpha_b$ or $\alpha \to \tilde{\alpha}_b$. Finally, Theorem~\ref{thm:4} (green arrows) provides asymptotics as $\alpha \to \alpha_0$ with $\alpha_0 \in \left([0,2\pi\textnormal{]}\backslash \mathcal M\right)$, and for $\alpha \to 0$ or $\pi$ while remaining outside $\mathcal M$.

\begin{figure}[t]
    \centering
    \includegraphics[width=0.8\linewidth]{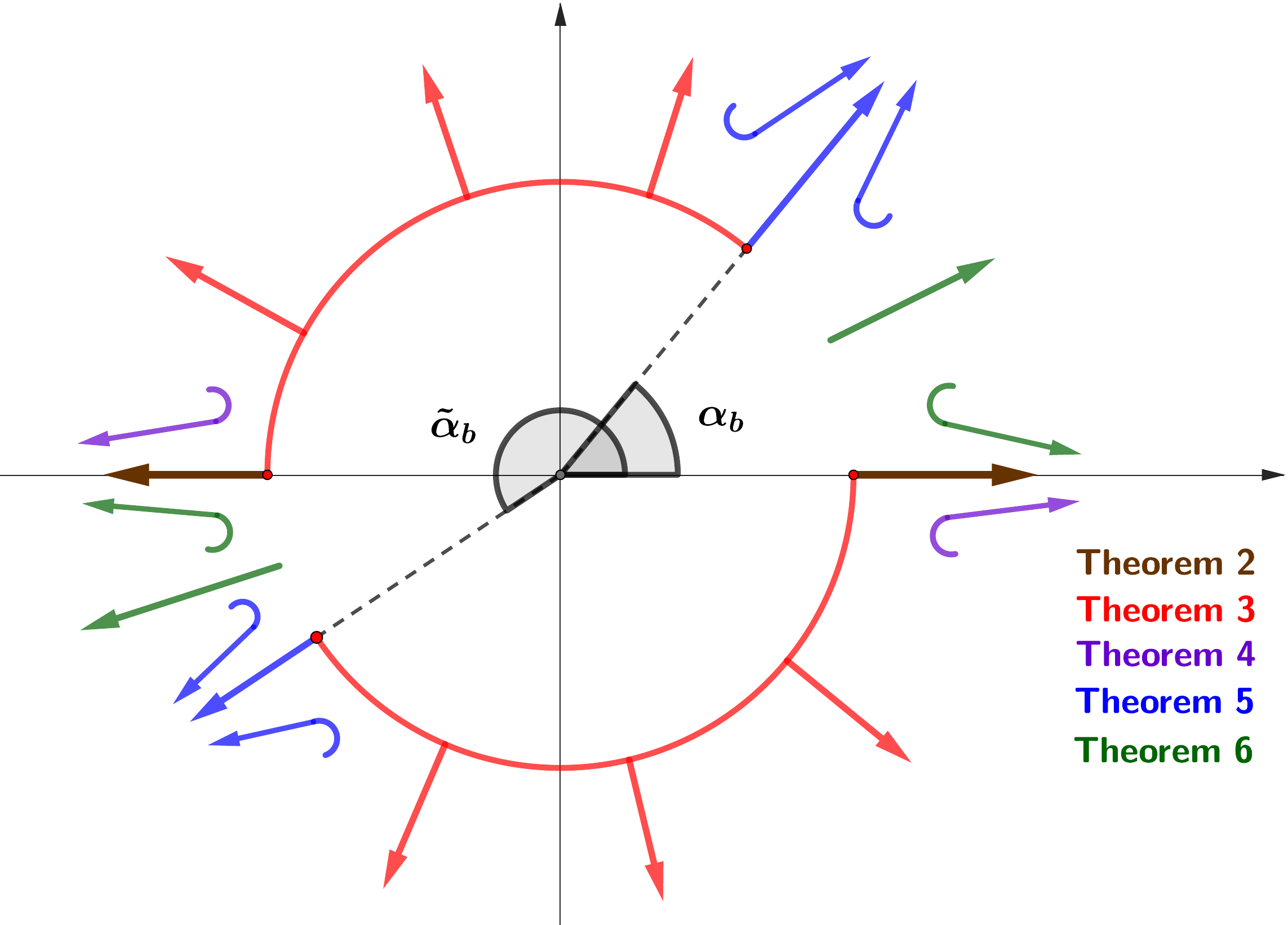}
    \caption{Direction of asymptotics: Theorems~\ref{alpha=0} to \ref{thm:4} (case D). Rounded arrows indicate that the angle $\alpha$ approaches the corresponding direction $\alpha_0$ only from one side, i.e., either $\alpha > \alpha_0$ or $\alpha < \alpha_0$, depending on the arrow. For example, the green rounded arrow on the right means $\alpha \to 0$ with $\alpha > 0$.}
    \label{fig:soleil}
\end{figure}

\begin{figure}[t]
    \centering
    \includegraphics[width=0.6\linewidth]{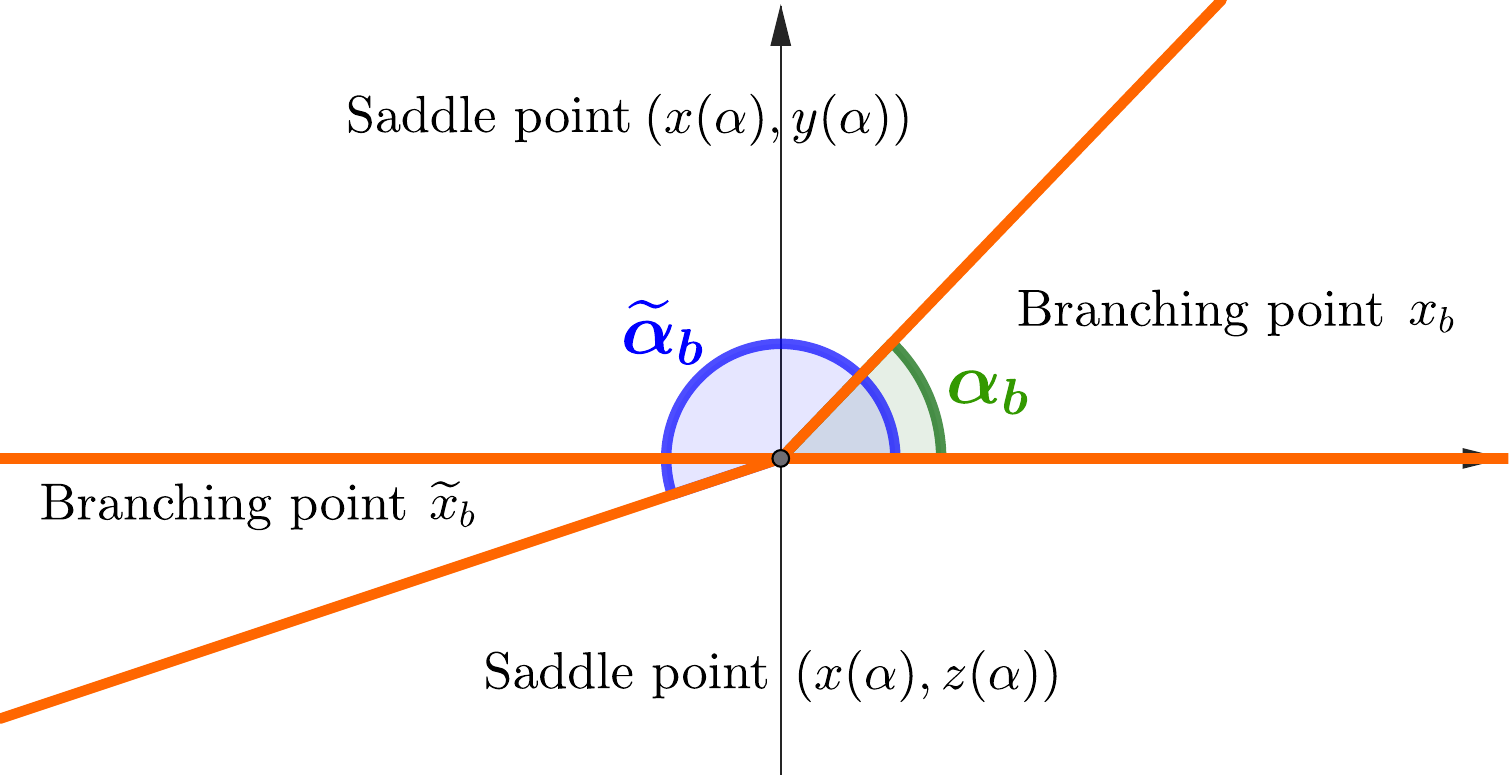}
    \caption{Type of asymptotics according to directions: saddle-point type or branching point.}
    \label{fig:typepasymptotique}
\end{figure}

The first theorem follows from a Tauberian theorem applied to Laplace transforms, based on the explicit expression~\eqref{laplace explicite}. The proof is given in Section~\ref{sec:alpha=0}.

\begin{theo}[Asymptotics, case $\alpha = 0, \pi$]\label{alpha=0}
Let $v_0 \in \R^2$. The following asymptotics hold:
\begin{equation}\label{as:alpha=0}
g^{v_0}(r, 0) \underset{r\to+\infty}{\sim} C_0 f_0(v_0)\frac{e^{-r\M}}{r^{3/2}}, \quad\quad g^{v_0}(-r, 0) \underset{r\to+\infty}{\sim} C_\pi f_\pi(v_0)\frac{e^{r\m}}{r^{3/2}},
\end{equation}
where $C_0, C_\pi > 0$ and $f_0(v_0), f_\pi(v_0) > 0$ are defined as follows.
The constant $C_0$ is given by:
\begin{equation} \label{eq:C_0}
C_0 =
\begin{cases}
    -\dfrac{\sqrt{\det(\Sigma^+)(\M - x_{min}^+)}}{\sqrt{\pi} \Sigma_{22}^+ (\Sigma_{22}^+ + \Sigma_{22}^-) \gamma(\M, Y^-(\M), Z^+(\M))} > 0 \quad &\text{if } \M = x_{max}^+ < x_{max}^- \\[2ex]
    -\dfrac{\sqrt{\det(\Sigma^-)(\M - x_{min}^-)}}{\sqrt{\pi} \Sigma_{22}^- (\Sigma_{22}^+ + \Sigma_{22}^-) \gamma(\M, Y^-(\M), Z^+(\M))} > 0 \quad &\text{if } \M = x_{max}^- < x_{max}^+.
\end{cases}
\end{equation}

Moreover, $f_0(v_0)$ is defined as follows:
\begin{itemize}
    \item If $\M = x_{max}^+ < x_{max}^-$ (cases B and C of Figure~\ref{anglesalphab}), then
    \begin{equation} \label{f(v_0)col}
    f_0(a_0, b_0) =
    \begin{cases}
    \left(b_0 - \dfrac{\Sigma_{22}^+}{(\Sigma_{22}^+ + \Sigma_{22}^-) \gamma(\M, Y^-(\M), Z^+(\M))}\right)
    e^{\M a_0 + b_0 Y^\pm(\M)} &\text{if } b_0 \geq 0 \\[2ex]
    \dfrac{-\Sigma_{22}^+}{(\Sigma_{22}^+ + \Sigma_{22}^-) \gamma(\M, Y^-(\M), Z^+(\M))} 
    e^{\M a_0 + b_0 Z^+(\M)} &\text{if } b_0 < 0
    \end{cases}
    \end{equation}

    \item If $\M = x_{max}^- < x_{max}^+$ (cases A and D of Figure~\ref{anglesalphab}), then
    \begin{equation} \label{f(v_0)branch}
    f_0(a_0, b_0) =
    \begin{cases}
    \dfrac{-\Sigma_{22}^-}{(\Sigma_{22}^+ + \Sigma_{22}^-) \gamma(\M, Y^-(\M), Z^+(\M))}
    e^{\M a_0 + b_0 Y^-(\M)} &\text{if } b_0 \geq 0 \\[2ex]
    \left(-b_0 - \dfrac{\Sigma_{22}^-}{(\Sigma_{22}^+ + \Sigma_{22}^- ) \gamma(\M, Y^-(\M), Z^+(\M))}\right)
    e^{\M a_0 + b_0 Z^\pm(\M)} &\text{if } b_0 < 0
    \end{cases}
    \end{equation}
\end{itemize}

\nr{The expressions for $C_\pi$ and $f_\pi(v_0)$ are symmetric}\nb{The expressions for $C_\pi$ and $f_\pi(a_0,b_0)$ follow from those for $C_0$ and $f_0(a_0,b_0)$ by reflection across the vertical axis: replace $\Sigma_{12}^\pm$, $\mu_1^\pm$, and $a_0$ by their negatives, leaving all other parameters unchanged}.
\end{theo}

The following theorem gives the asymptotics in the most typical case when $\alpha \in \inter{\mathcal M}$. It is proved at the end of Section~\ref{sub:saddle} using the saddle point method \nr{on a Riemann surface }after extending the Laplace transforms.

\begin{theo}[Asymptotics for $\alpha \to \alpha_0$, $\alpha_0 \in \inter{\mathcal M}$]\label{thm:1}
Let $v_0 \in \R^2$. Then, if $\alpha_0 \in \inter{\mathcal{M}}\cap (0, \pi)$, we have
\begin{equation}\label{cas_1}
g^{v_0}(r\cos(\alpha), r\sin(\alpha))
\underset{r\to\infty \atop\alpha\to\alpha_0}{\sim}
C^+(\alpha_0)h_{\alpha_0}(v_0)  \frac{e^{-r(\cos(\alpha)x(\alpha) + \sin(\alpha)y(\alpha))}}{\sqrt{r}}
\end{equation}
and if $\alpha_0 \in \inter{\mathcal{M}}\cap(\pi, 2\pi)$, the following asymptotics hold
\begin{equation}\label{cas_1'}
g^{v_0}(r\cos(\alpha), r\sin(\alpha))
\underset{r\to\infty \atop\alpha\to\alpha_0}{\sim}
C^-(\alpha_0)h_{\alpha_0}(v_0)  \frac{e^{-r(\cos(\alpha)x(\alpha) + \sin(\alpha)z(\alpha))}}{\sqrt{r}}
\end{equation}
where $C^\pm(\alpha)$ and $h_{\alpha_0}(v_0)$ are defined as follows. The coefficient $C^\pm(\alpha)$ is given by:
\begin{align}\label{C(alpha)}
C^\varepsilon(\alpha) &= \frac{1}
{ \sqrt{2\pi  \left( \Sigma^\varepsilon_{11}\sin^2(\alpha) - 2\Sigma^\varepsilon_{12} \sin(\alpha)\cos(\alpha) + \Sigma^\varepsilon_{22} \cos^2(\alpha) \right)} }
\sqrt{ \frac{ \sin(\alpha) }{ \partial_y \gamma_\varepsilon(x(\alpha), y(\alpha)) } }, \quad \varepsilon \in \{+, -\}.
\end{align}

Furthermore, $h_\alpha(v_0)$ is defined as follows:
\begin{itemize}
    \item If $\alpha \in \mathcal{M}\cap(0, \pi),$ then
    \begin{equation} \label{Harm1}
    h_\alpha(v_0) =
    \begin{cases}
    e^{a_0x(\alpha) + b_0y(\alpha)} - \dfrac{\gamma(x(\alpha), y(\alpha), Z^+(x(\alpha)))}{\gamma(x(\alpha), Y^-(x(\alpha)), Z^+(x(\alpha)))} e^{a_0x(\alpha) + b_0Y^-(x(\alpha))} & \text{if } b_0 \geq 0 \\[2ex]
    \left(1 - \dfrac{\gamma(x(\alpha), y(\alpha), Z^+(x(\alpha)))}{\gamma(x(\alpha), Y^-(x(\alpha)), Z^+(x(\alpha)))}\right) e^{a_0x(\alpha) + b_0Z^+(x(\alpha))} & \text{if } b_0 < 0
    \end{cases}
    \end{equation}

    \item If $\alpha \in {\mathcal{M}}\cap(\pi, 2\pi)$, then
    \begin{equation} \label{Harm1'}
    h_\alpha(v_0) =
    \begin{cases}
    \left(1 - \dfrac{\gamma(x(\alpha), Y^-(x(\alpha)), z(\alpha))}{\gamma(x(\alpha), Y^-(x(\alpha)), Z^+(x(\alpha)))}\right) e^{a_0x(\alpha) + b_0Y^-(x(\alpha))} & \text{if } b_0 \geq 0 \\[2ex]
    e^{a_0x(\alpha) + b_0z(\alpha)} - \dfrac{\gamma(x(\alpha), Y^-(x(\alpha)), z(\alpha))}{\gamma(x(\alpha), Y^-(x(\alpha)), Z^+(x(\alpha)))} e^{a_0x(\alpha) + b_0Z^+(x(\alpha))} & \text{if } b_0 < 0
    \end{cases}
    \end{equation}
\end{itemize}
\end{theo}

The functions $h_\alpha$ are in fact harmonic functions, see Theorem~\ref{thm5} later.

The next theorem states the asymptotics of Green's functions as $\alpha \to 0, \pi$ and $\alpha\in\inter{\mathcal M}$.
It establishes the missing connection between Theorem~\ref{alpha=0} and Theorem~\ref{thm:1}. To clarify this point further,
let us note that in  \eqref{cas_1}
the constant term $C(\alpha_0)h_{\alpha_0}(z)\to 0$  as $\alpha_0 \to 0$. 
Hence, in the case $\alpha_0=0$, we need to expand the asymptotics \eqref{cas_1}
up to the second term. The saddle-point method yields: 

    \[
        g^{v_0}(r\cos(\alpha), r\sin(\alpha)) \underset{r \to \infty \atop \alpha \to 0,\, \alpha \geq 0}{\sim}
        \left(C^+(\alpha)\, h_\alpha(v_0) + \frac{C_0 f_0(v_0)}{r} \right)
        \frac{e^{-r(\cos(\alpha)\, x(\alpha) + \sin(\alpha)\, y(\alpha))}}{\sqrt{r}},
    \]
    where $$C^+(\alpha)\, h_\alpha(v_0) \underset{\alpha \to 0}{\sim} \kappa C_0 f_0(v_0)\, \alpha$$  for some constant $\kappa > 0$ (see \eqref{eq:memefonctionharm}).
The first term corresponds to Theorem~\ref{thm:1}, while the second corresponds to Theorem~\ref{alpha=0}. As $r \to \infty$ and $\alpha \to 0$, we observe a competition between these two contributions. It is this competition that is established in the following theorem, proved at the end of Section~\ref{sub:saddle}.

\begin{theo}[Asymptotics as $\alpha \to 0, \pi$ along $\mathcal M$]\label{thm:2}
Let $v_0 \in \R^2$ and let $\alpha_0 = 0$.
\begin{itemize}
    \item Assume that $x_{max}^+ < x_{max}^-$ (cases B and C of Figure~\ref{anglesalphab}). Then, for $\alpha \geq 0$ we have
\begin{equation}\label{cas_2}
g^{v_0}(r\cos(\alpha), r\sin(\alpha))
\underset{r\to\infty \atop\alpha\to0, \alpha \geq 0}{\sim}
 C_0f_0(v_0)\left(\kappa\alpha + \frac{1}{r} \right) \frac{e^{-r(\cos(\alpha)x(\alpha) + \sin(\alpha)y(\alpha))}}{\sqrt{r}}.
\end{equation}
where $\kappa > 0$ is an explicit constant (cf. \eqref{eq:kappa} below), $C_0$ is given by \eqref{eq:C_0}, and $f_0(v_0)$ is given by \eqref{f(v_0)col}. Moreover,
\[
\left.\partial_\alpha \left(C(\alpha)h_{\alpha}(v_0)\right)\right|_{\alpha = 0^+} = \kappa C_0f_0(v_0).
\]
    
\item Assume now that $x_{max}^+ > x_{max}^-$ (cases A and D of Figure~\ref{anglesalphab}). Then, for $\alpha \leq 0$,
\begin{equation}\label{cas_2'}
g^{v_0}(r\cos(\alpha), r\sin(\alpha))
\underset{r\to\infty \atop\alpha\to0, \alpha \leq 0}{\sim}
 C_0f_0(v_0)\left(\kappa|\alpha| + \frac{1}{r} \right) \frac{e^{-r(\cos(\alpha)x(\alpha) + \sin(\alpha)z(\alpha))}}{\sqrt{r}},
\end{equation}
where $f_0(v_0)$ is given by \eqref{f(v_0)branch}. Moreover,
\[
\left.\partial_\alpha \left(C(\alpha)h_{\alpha}(v_0)\right)\right|_{\alpha = 0^-} = \kappa C_0f_0(v_0).
\]
\end{itemize}

The \nb{analogous} result holds for $\alpha_0 = \pi$.
\end{theo}

Now, let us consider the asymptotics as $\alpha \to \alpha_b$ (and $\tilde{\alpha}_b$). If $\alpha$ remains within $\mathcal M$, then the asymptotics are given by those of Theorem~\ref{thm:1}. However, this is no longer the case in the parts outside $\mathcal M$. The asymptotic behaviour arises from the competition between the saddle point $(x(\alpha), y(\alpha))$ and the branching point $(x(\alpha_b), y(\alpha_b)) $ and thus depends on the speed of convergence of $\alpha \to \alpha_b$ relative to $r \to \infty$.
The proof can be found at the end of Section~\ref{sub:alphab}.

\begin{theo}[Asymptotics for $\alpha \to \alpha_b, \tilde{\alpha}_b$]\label{thm3}
Let $v_0 \in \R^2$ and let $\alpha_0 = \alpha_b$. Assume $x_{max}^+ > x_{max}^-$, i.e., $\alpha_b \in (0,\pi)$ (cases A and D of Figure~\ref{anglesalphab}). Define
\begin{equation}\label{eq:K+2}
    K_+^2 = \frac{\Sigma_{22}^+\big(y(\alpha_b) - Y^-(x(\alpha_b))\big)}{4|\sin(\alpha_b)| \left( \Sigma^+_{11}\sin^2(\alpha_b) - 2\Sigma^+_{12} \sin (\alpha_b) \cos (\alpha_b) + \Sigma^+_{22} \cos^2(\alpha_b) \right)}.
\end{equation}
Then:
\begin{itemize}
    \item[$(i)$] If $\alpha > \alpha_b$, then \eqref{cas_1} holds.
    \item[$(ii)$] If $\alpha < \alpha_b$ and $r(\alpha_b - \alpha)^2$ is bounded, then \eqref{cas_1} holds.
    \item[$(iii)$] If $\alpha < \alpha_b$, $r(\alpha - \alpha_b)^2 \to +\infty$, and $\frac{e^{K_+^2 r (\alpha_b - \alpha)^2}}{r(\alpha_b - \alpha)^{3/2}} \to 0$,
    then \eqref{cas_1} holds.
    \item[$(iv)$] If $\alpha < \alpha_b$, $r(\alpha - \alpha_b)^2 \to +\infty$, and
    $\frac{e^{K_+^2 r (\alpha_b - \alpha)^2}}{r(\alpha_b - \alpha)^{3/2}} \to +\infty,$
    then
    \begin{equation}\label{cas_3.2}
    g^{v_0}(r\cos(\alpha), r\sin(\alpha))
    \underset{r\to\infty \atop \alpha\to\alpha_b}{\sim}  C_{br} f_0(v_0) \frac{e^{-r(\cos(\alpha)x(\alpha_b) + \sin(\alpha)y(\alpha_b))}}{r^{3/2} (\alpha_b - \alpha)^{3/2}},
    \end{equation}
    where $f_0(v_0) > 0$ is given by \eqref{f(v_0)branch} and $C_{br} = |\sin(\alpha_b)|^{3/2} C_0$ (see \eqref{eq:C_0}).
    \item[$(v)$] If $\alpha < \alpha_b$ (resp. $\alpha > \alpha_b$), $r(\alpha - \alpha_b)^2 \to +\infty$, and
    \[
    \frac{e^{K_+^2 r (\alpha_b - \alpha)^2}}{r(\alpha_b - \alpha)^{3/2}} \to c > 0,
    \]
    then
    \begin{equation}\label{cas_3.3}
    g^{v_0}(r\cos(\alpha), r\sin(\alpha))
    \underset{r\to\infty \atop \alpha\to\alpha_b}{\sim} \left( C(\alpha_b) h_{\alpha_b}(v_0) + c\, C_{br} f_0(v_0) \right) \frac{e^{-r(\cos(\alpha)x(\alpha) + \sin(\alpha)y(\alpha))}}{\sqrt{r}},
    \end{equation}
    where $h_{\alpha_b}(v_0)$ is given by \eqref{Harm1}.
\end{itemize}

Now assume that $x_{max}^+ < x_{max}^-$ (cases B and C of Figure~\ref{anglesalphab}). Then, the analogous cases $(i)$ to $(v)$ hold. More precisely, let
\[
K_-^2 = \frac{\Sigma_{22}^- \big(z(\alpha_b) - Z^+(x(\alpha_b))\big)}{4|\sin(\alpha_b)| \left( \Sigma^-_{11} \sin^2(\alpha_b) - 2 \Sigma^-_{12} \sin(\alpha_b) \cos(\alpha_b) + \Sigma^-_{22} \cos^2(\alpha_b) \right)}.
\]
\begin{itemize}
    \item[$(i)$-$(iii)$] In cases $(i)$ to $(iii)$, replace $K_+^2$ by $K_-^2$ and the conditions $\alpha < \alpha_b$ (resp. $\alpha > \alpha_b$) by $\alpha > \alpha_b$ (resp. $\alpha < \alpha_b$). Then, \eqref{cas_1'} holds in these cases.
    \item[$(iv)$-$(v)$] In cases $(iv)$ to $(v)$, replace $K_+^2$ by $K_-^2$ and the condition $\alpha < \alpha_b$ by $\alpha > \alpha_b$. Then, the corresponding asymptotics hold replacing $y(\alpha)$ by $z(\alpha)$. Namely, in case $(iv)$,
    \begin{equation}\label{cas_3.2'}
    g^{v_0}(r\cos(\alpha), r\sin(\alpha))
    \underset{r\to\infty \atop \alpha\to\alpha_b}{\sim} C_{br} f_0(v_0) \frac{e^{-r(\cos(\alpha) x(\alpha_b) + \sin(\alpha) z(\alpha_b))}}{r^{3/2} |\alpha_b - \alpha|^{3/2}},
    \end{equation}
    where $f_0(v_0) > 0$ is given by \eqref{f(v_0)col}. In case $(v)$,
    \begin{equation}
    g^{v_0}(r\cos(\alpha), r\sin(\alpha))
    \underset{r\to\infty \atop \alpha\to\alpha_b}{\sim} \left( C(\alpha_b) h_{\alpha_b}(v_0) + c\, C_{br} f_0(v_0) \right) \frac{e^{-r(\cos(\alpha) x(\alpha) + \sin(\alpha) z(\alpha))}}{\sqrt{r}},
    \end{equation}
    where $h_{\alpha_b}(v_0)$ is given by \eqref{Harm1'}.
\end{itemize}

The \nb{analogous} result holds for $\alpha_0 = \tilde{\alpha}_b$.
\end{theo}

The direction $\alpha_b$ gives rise to two harmonic functions. The first of these is $h_{\alpha_b}(v_0)$, which is defined in \eqref{Harm1} when $\alpha=\alpha_b$. 
This function can also be viewed as the limit of the functions $h_{\alpha}(v_0)$ obtained in Theorem~\ref{thm:1} as $\alpha \to\alpha_b$. The second function is $f_0(v_0)$, given by \eqref{f(v_0)col}, and corresponds to the direction $\alpha_0=0$.
When $\alpha \to \alpha_b$ and $r\to \infty$ at the rate described in case (v), linear combinations of these two functions determine the asymptotic constants as in \eqref{cas_3.3}.

The last asymptotic theorem deals with the remaining cases where $\alpha_0 \in [0,2\pi\textnormal{]}\backslash \mathcal M$ or $\alpha \to 0,\pi$ as $\alpha \notin \mathcal M$. The asymptotics are then determined by the branching points. The proof is given at the end of Section~\ref{sub:apresalphab}.

\begin{theo}[Asymptotics for $\alpha \to \alpha_0,$ $\alpha_0 \in [0,2\pi\textnormal{]}\backslash \mathcal M$, or $\alpha \to 0, \pi$ with $\alpha \in [0,2\pi\textnormal{]}\backslash \mathcal M$]\label{thm:4} 
Assume $x^+_{max} > x^-_{max}$ (cases A and D of Figure~\ref{anglesalphab}). Let $v_0 \in \R^2$ and let $0 \leq \alpha_0 < \alpha_b$. Then:
\begin{equation}\label{cas4}
g^{v_0}(r\cos(\alpha), r\sin(\alpha))
\underset{r\to\infty \atop\alpha\to\alpha_0, \alpha \geq 0}{\sim} C'_{br}(\alpha_0) f_0(v_0)\frac{e^{-r(\cos(\alpha)x(\alpha_b) + \sin(\alpha)y(\alpha_b))}}{r^{3/2}
}
\end{equation}
where $C'_{br}(\alpha_0) = \left(\frac{\sin(\alpha_b)}{\sin(\alpha_b - \alpha_0)}\right)^{3/2}C_0$ (see \eqref{eq:C_0}) and where $f_0(v_0)$ is given by \eqref{f(v_0)branch}. 

Assume now $x^+_{max} < x^-_{max}$ (cases B and C of Figure~\ref{anglesalphab}) and let $\alpha_b < \alpha_0 \leq 2\pi$. Then,
\begin{equation}\label{cas4'}
 g^{v_0}(r\cos(\alpha), r\sin(\alpha))
\underset{r\to\infty \atop\alpha\to\alpha_0, \alpha \leq 2\pi}{\sim}  C'_{br}(\alpha_0)f_0(v_0)\frac{e^{-r(\cos(\alpha)x(\alpha_b) + \sin(\alpha)z(\alpha_b))}}{r^{3/2}}
\end{equation}
where $C'_{br}(\alpha_0)$ has the same expression and where $f_0(v_0)$ is given by \eqref{f(v_0)col}.

The symmetric results hold for $\tilde\alpha_b< \alpha_0 \leq\pi$ in the cases A and B and for $ \pi\leq \alpha_0 < \tilde\alpha_b$ in the cases C and D.
\end{theo}

Note that the asymptotics \eqref{cas4} formally match case $(iv)$ of Theorem~\ref{thm3} in the direction $\alpha_b$, since $\sin(\alpha_b - \alpha) \sim \alpha_b - \alpha$ as $\alpha \to \alpha_b.$ Since $C'_{br}(\alpha)|_{\alpha = 0} = C_0$, the asymptotics given by \eqref{cas4} also formally match \eqref{as:alpha=0} for $\alpha = 0.$ Furthermore, we observe that a single harmonic function, $f_0(v_0)$ (arising from the asymptotics at $\alpha = 0$ obtained in Theorem~\ref{alpha=0}) corresponds to all directions described in Theorem~\ref{thm:4}.

\subsubsection{Martin Boundary and positive harmonic functions}\label{subsec:results_Martin}
Before stating the results on the Martin boundary and harmonic functions, let us recall the definition of harmonic function. 

\begin{defi}
A measurable non-negative function $h$ is called harmonic with respect to the Markov process $V$ if for all open set $U$ relatively compact in $\R^2$ and $v_0 \in \R^2$, $$h(v_0) = \E_{v_0}\left[h\left(V_{T_U}\right)\fc_{T_U < +\infty}\right] $$
where $T_U = \inf\{t\geq 0 \mid V_t \in U\}$. 
\end{defi}

The general theory of Martin boundary \cite[Theorem 3]{kunitaWatanabe1965} justifies that the limits of the quotient of Green's functions $\frac{g(v_0,v)}{g(0,v)}$ as $|v| \to +\infty$ -- called Martin functions -- 
are generally \textit{excessive} in $v_0$ (see \cite[Section 3]{kunitaWatanabe1965}), which is weaker than harmonic. 
In Theorems~\ref{alpha=0} to \ref{thm:4}, we obtained all Martin functions. In particular, we obtained: 
\begin{itemize}
    \item For each direction $\alpha_0\in\inter{\mathcal M}$, one function $h_{\alpha_0}(v_0)$ (see Theorem~\ref{thm:1});
    \item For the directions $0$ and $\pi$, two functions $f_0(v_0)$ and $f_\pi(v_0)$ (see Theorems~\ref{alpha=0} and~\ref{thm:2});
    \item For $\alpha_0 = \alpha_b$, all the convex combinations of $h_{\alpha_b}(v_0)$ and $f_0(v_0)$, see \eqref{cas_3.3} (Theorem~\ref{thm3});
    \item Similarly, for $\alpha_0 = \tilde\alpha_b$, with $h_{\tilde\alpha_b}(v_0)$ and $f_\pi(v_0)$ (see Theorem~\ref{thm3});
    \item For each direction $\alpha_0 \in [0,2\pi]\backslash \mathcal M$, the functions $f_0(v_0)$ and $f_\pi(v_0)$ (see Theorem~\ref{thm:4}).
\end{itemize}
We prove in Section~\ref{sec:5} that those functions are harmonic.


\begin{rem}[Partial differential equation]
The partial differential equation associated with the harmonicity of a function $h$ with respect to our process is given by 
\begin{equation}
    \begin{cases}
(\frac{1}{2}\nabla \cdot \Sigma^+ \nabla + \mu^+ \cdot \nabla)\, h(v) = 0 & \text{for } v \in \mathbb{R} \times (0, +\infty), \\
(\frac{1}{2}\nabla \cdot \Sigma^- \nabla + \mu^- \cdot \nabla)\, h(v) = 0 & \text{for } v \in \mathbb{R} \times (-\infty, 0), \\
(q_1,1+q_2) \nabla h(a,0^+)= (-q_1,1-q_2) \nabla h(a,0^-) & \text{for } a \in \mathbb{R},
\\
h(a,0^+)= h(a,0^-) \text{ and } (1,0)\cdot\nabla h(a,0^+)=(1,0)\cdot\nabla h(a,0^-) & \text{for } a \in \mathbb{R}
.
    \end{cases}
    \label{eq:pdeharm}
\end{equation}

It is straightforward to verify that the harmonic functions $f_0, f_\pi,$ and $h_\alpha, \alpha \in \mathcal M\backslash\{0,\pi\}$ obtained in our main results satisfy this equation.
\end{rem}
Then, the structures of the Martin boundary and the minimal Martin boundary follow: the proof is detailed in Section~\ref{sec:5}.
Let us just notice the fact that the non-minimality of the full Martin boundary comes from the Green's functions asymptotics 
\eqref{cas_3.3} in Theorem~\ref{thm3}.

\begin{theo}[Martin Boundary, minimal Martin Boundary and positive harmonic functions]\label{thm5}
All of the functions $h_\alpha, \alpha \in \mathcal M\backslash\{0,\pi\}$, $f_0$ and $f_\pi$ are minimal harmonic functions. Furthermore, the Martin boundary $\Gamma$ of the process is given by 
\begin{equation}\label{front_Martin}
\Gamma = \left[\{e^{i\alpha}\}_{\alpha \in \mathcal M} \cup \{ue^{i\alpha}\}_{u \in [1,2],\, \alpha = \alpha_b, \tilde\alpha_b}\right]/\mathcal{R} \;\sim\; \mathbb{S}^1
\end{equation}
where $\mathcal{R}$ is the equivalence relation defined by $e^{i.0} \,\mathcal{R}\, 2e^{i\alpha_b}$, $e^{i\pi} \,\mathcal{R}\, 2e^{i\tilde\alpha_b}$
(see Figure~\ref{tableauGamma}).

The minimal Martin boundary, denoted by $\Gamma_{min}$, is given by
\begin{equation}
\Gamma_{min} = \{e^{i\alpha}\}_{\alpha \in \mathcal M}
\end{equation}
(see Figure~\ref{anglesalphab}). Therefore, according to Martin boundary theory, for each positive harmonic function $h$ (with respect to the process), there exists a unique Radon measure $\nu$ on $\mathcal M$ such that : 
\begin{equation}\label{eq:representation}
    \forall v \in \R^2, \quad h(v) = \int_{\mathcal M} \frac{h_\alpha(v)}{h_\alpha(0)} \, d\nu(\alpha).
\end{equation}
where we have used the notation $\frac{h_0(v)}{h_0(0)} = \frac{f_0(v)}{f_0(0)}$ and $\frac{h_\pi(v)}{h_\pi(0)} = \frac{f_\pi(v)}{f_\pi(0)}$ since $$\frac{h_\alpha(v)}{h_\alpha(0)} \underset{\alpha \to 0\atop \alpha \in \mathcal M}{\longrightarrow} \frac{f_0(v)}{f_0(0)}\quad \textnormal{and}\quad \frac{h_\alpha(v)}{h_\alpha(0)} \underset{\alpha \to \pi \atop \alpha \in \mathcal M}{\longrightarrow} \frac{f_\pi(v)}{f_\pi(0)}.$$
\end{theo}

\begin{figure}[t]
\begin{tabular}{|c|c|}
\hline
Case A & Case B\\
\hline
$x_{max}^+ > x_{max}^-$& $x_{max}^+ < x_{max}^-$\\
$x_{min}^+ <\ x_{min}^-$  & $x_{min}^+ <\ x_{min}^-$  \\
\hline
\includegraphics[scale = 0.5]{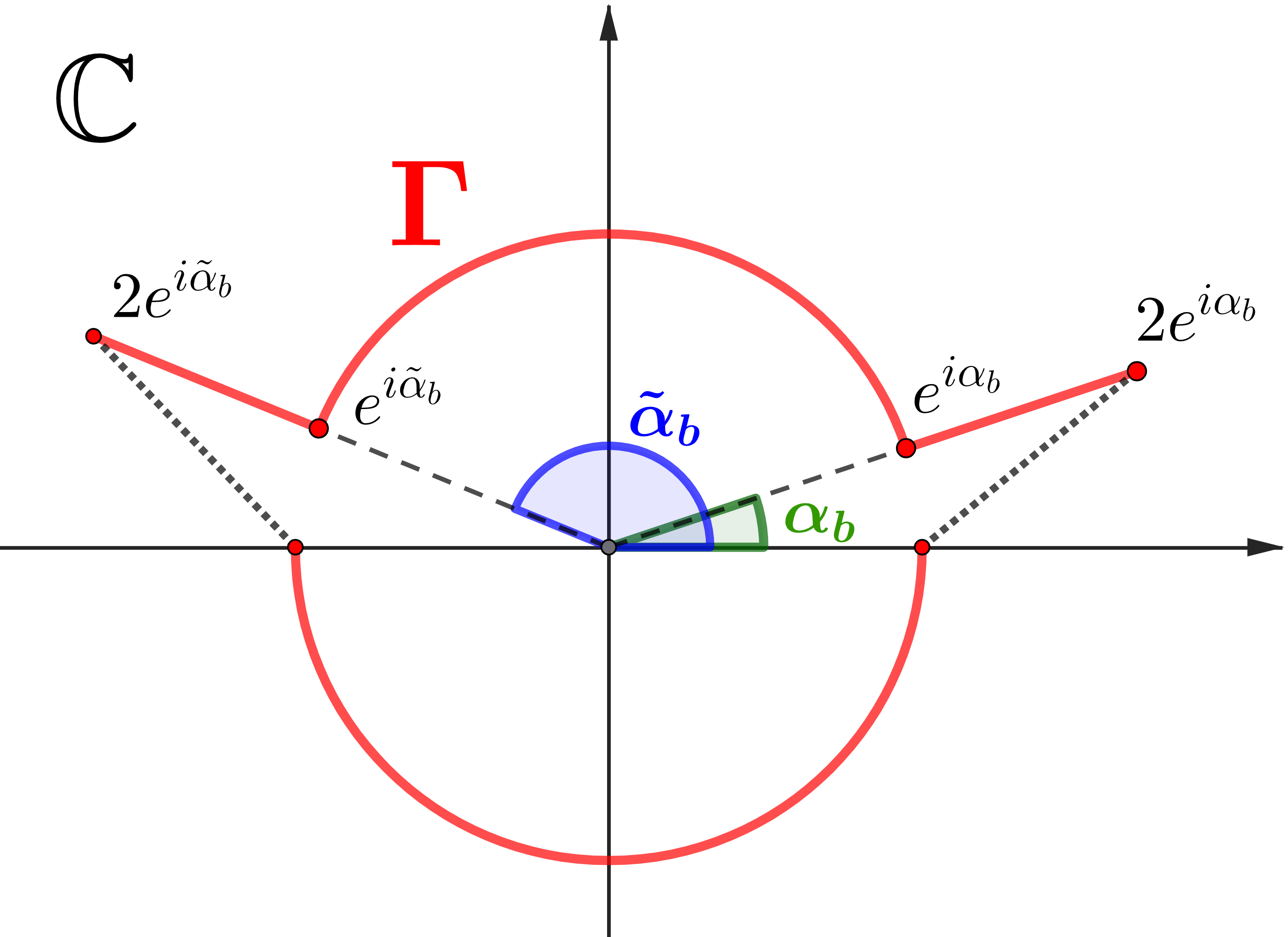}&\includegraphics[scale = 0.5]{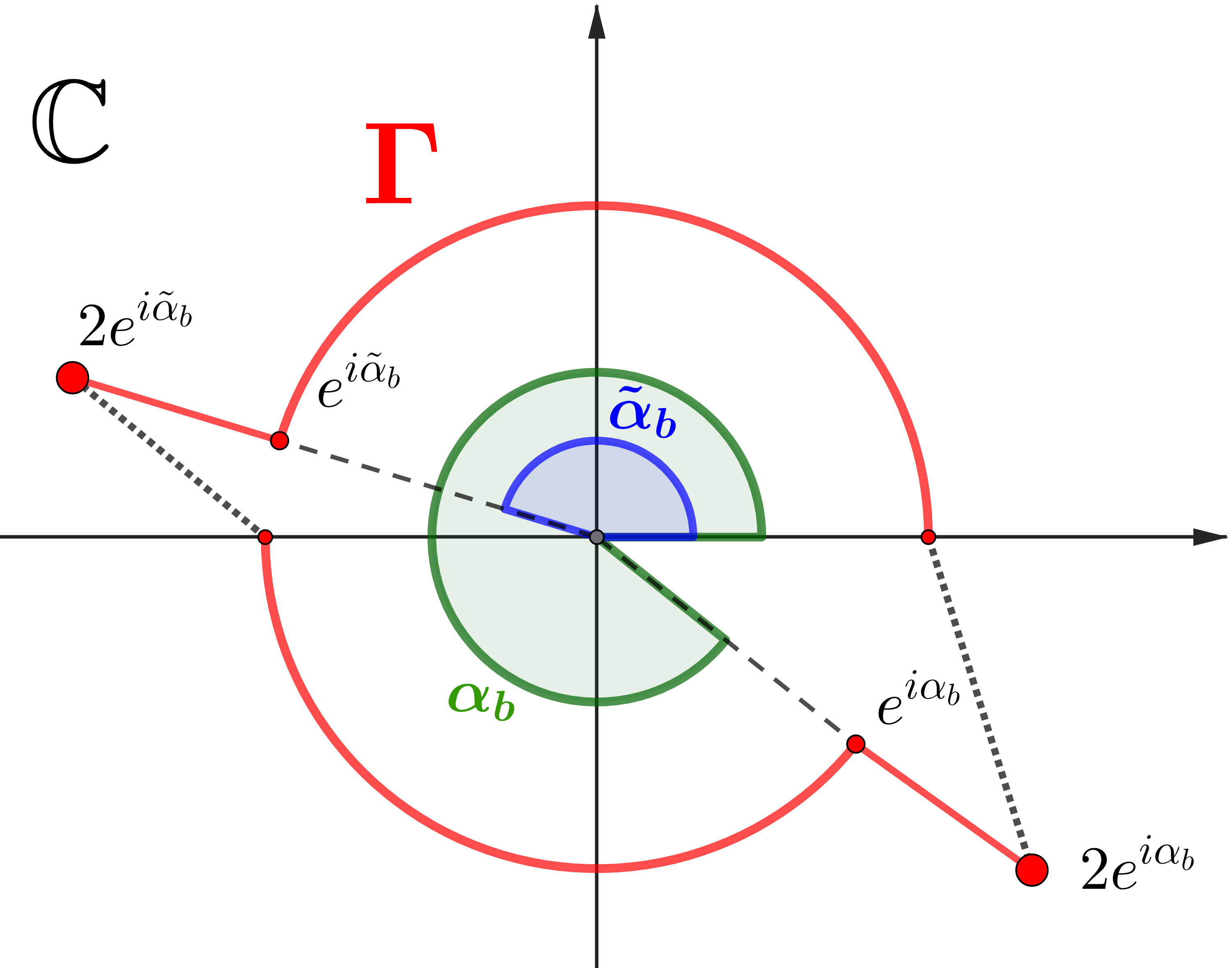}  \\
\hline
Case C & Case D\\
\hline
$x_{max}^+ < x_{max}^-$ & $x_{max}^+ > x_{max}^-$ \\
$x_{min}^+ >\ x_{min}^-$& $x_{min}^+ >\ x_{min}^-$\\
\hline
\includegraphics[scale = 0.5]{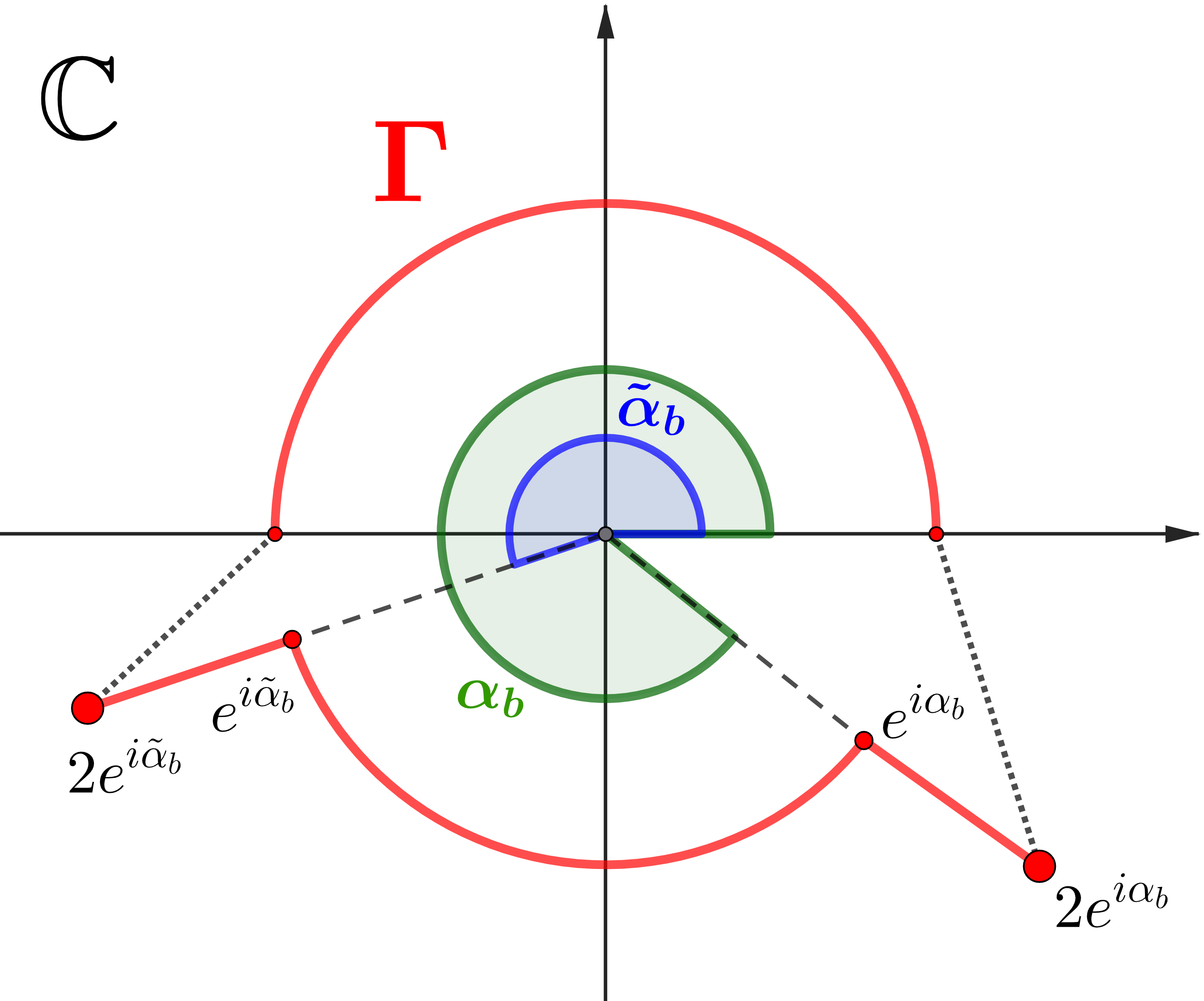}&\includegraphics[scale = 0.5]{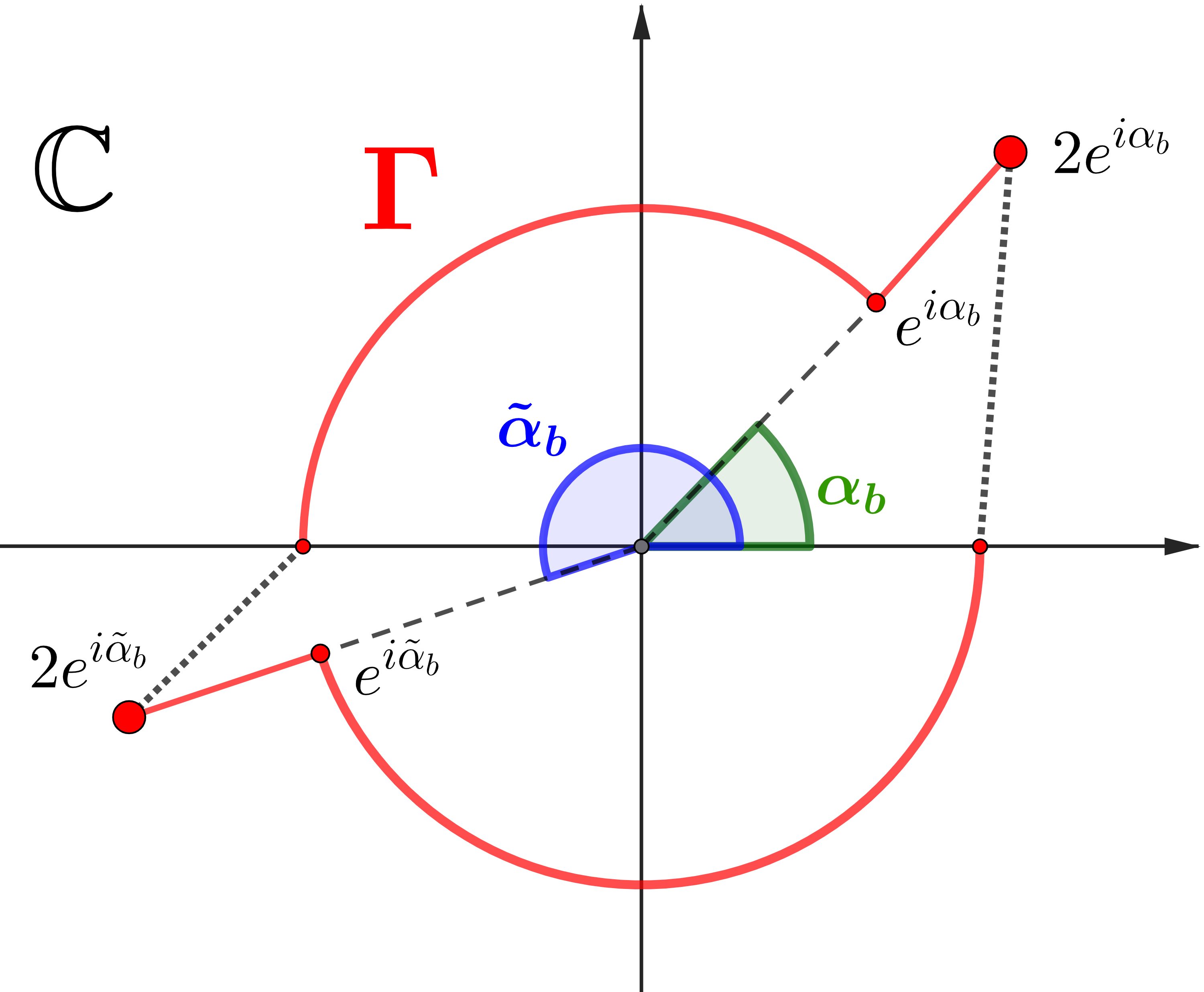}\\
\hline
\end{tabular}
\caption{Representation of the Martin boundary $\Gamma$ depending on the parameters.}\label{tableauGamma}
\end{figure}


Using the preceding representation theorem, we can derive the well-known one-dimensional result concerning the probabilities of escaping into the upper or lower half-plane.
By assumption~\eqref{hyp:drift} on the drift (see Figure~\ref{parametres}), the second component $B$ of the process naturally satisfies $B_t \underset{t\to+\infty}{\longrightarrow} \pm\infty$. By standard arguments, the mappings 
\[
v_0 \longmapsto \P_{v_0}\bigl(B_t \underset{t\to+\infty}{\longrightarrow} +\infty\bigr) \quad \text{and} \quad v_0 \longmapsto \P_{v_0}\bigl(B_t \underset{t\to+\infty}{\longrightarrow} -\infty\bigr)
\]
are harmonic functions. 
Looking at the bounded non negative harmonic functions (which reduces to convex combinations of $h_{\alpha_{\mu^+}}$ and $h_{\alpha_{\mu^-}}$), we deduce in Section~\ref{sec:5} the following classical one-dimensional result.

\begin{corollary}[Upper and lower escape probabilities]\label{cor:escape+-infty}
Let $(a_0, b_0) \in \R^2$. Then: $$\P_{(a_0,b_0)}\bigl(B_t \underset{t\to+\infty}{\longrightarrow} +\infty\bigr) = h_{\alpha_{\mu^+}}(b_0),\quad \textnormal{and}\quad \P_{(a_0,b_0)}\bigl(B_t \underset{t\to+\infty}{\longrightarrow} -\infty\bigr) = h_{\alpha_{\mu^-}}(b_0).$$
In particular,
\begin{align}
\P_{(a_0,b_0)}\bigl(B_t \underset{t\to+\infty}{\longrightarrow} +\infty\bigr) &= 1 - \P_{(a_0,b_0)}\bigl(B_t \underset{t\to+\infty}{\longrightarrow} -\infty\bigr) \\
& = \begin{cases}
1 + \frac{\mu_2^-}{\mu_2^+ - \mu_2^-} e^{-2 b_0 \mu^+_2 / \Sigma_{22}^+} & \text{if } b_0 \geq 0, \\
\frac{\mu_2^+}{\mu_2^+ - \mu_2^-} e^{-2 b_0 \mu_2^- / \Sigma_{22}^-} & \text{if } b_0 < 0.
\end{cases}
\end{align}
\end{corollary}

\section{Process and  functional equation}\label{sec:process_functionalequation}

In Section~\ref{sub:existence}, we prove Proposition~\ref{prop:existence_deuxnoyaux}. Then, we present some preliminary lemmas in Section~\ref{sub:preliminary}, which are necessary to establish the functional equation~\eqref{eq fonctionnelle}. A short intuitive proof of the functional equation is given in Section~\ref{sub:formalproof}, followed by a detailed one in Section~\ref{sub:detailedproof}. As a first reading, the reader may focus on the short proof.

\subsection{Construction of the process}\label{sub:existence}

In this section, we prove Proposition~\ref{prop:existence_deuxnoyaux}. The proof is adapted from~\cite[Proposition~1]{Lejay2011}, where the covariance matrix is assumed to be diagonal. We first establish the existence and pathwise uniqueness of the solution to equation~\eqref{EDS}, and then identify the corresponding generator.

\subsubsection*{Existence and uniqueness of \eqref{EDS}}

First, note that \eqref{EDS} can be written as the following SDE:
\begin{equation}\label{EDSS}
\begin{cases} 
 dA_t &= \sqrt{\Sigma_{11}(B_t) - \frac{\Sigma_{12}^2(B_t)}{\Sigma_{22}(B_t)}}\,dW^1_t + \frac{\Sigma_{12}(B_t)}{\sqrt{\Sigma_{22}(B_t)}}\,dW^2_t + \mu_1(B_t)\,dt + \frac{\Sigma_{12}^+ - \Sigma_{12}^-}{\Sigma_{22}^+ + \Sigma_{22}^-} \,dL^0_t(B), \quad (1)\\[1ex]
 dB_t &= \sqrt{\Sigma_{22}(B_t)}\,dW^2_t + \mu_2(B_t)\,dt + \frac{\Sigma_{22}^+ - \Sigma_{22}^-}{\Sigma_{22}^+ + \Sigma_{22}^-} \,dL^0_t(B), \quad\quad\quad\quad\quad\quad\quad\quad\quad\quad\quad\quad (2)
\end{cases}
\end{equation}

Here, $\big((W^1_t, W^2_t)\big)_{t \geq 0}$ is a standard two-dimensional Brownian motion, and $L^0(B)$ denotes the symmetric local time of $B$ at $0$.
Indeed, the matrix
\[
\sigma(B_t) = \begin{pmatrix}
\sqrt{\Sigma_{11}(B_t) - \frac{\Sigma_{12}^2(B_t)}{\Sigma_{22}(B_t)}} & \frac{\Sigma_{12}(B_t)}{\sqrt{\Sigma_{22}(B_t)}}\\
0 & \sqrt{\Sigma_{22}(B_t)}
\end{pmatrix}
\]
satisfies $\sigma(B_t)\, \sigma(B_t)^T = \Sigma(B_t)$.

Let $(W^1, W^2)$ be a standard two-dimensional Brownian motion. By~\cite{Legall1984}, there exists a pathwise unique solution $(B_t)_{t \geq 0}$ of equation~$(2)$ starting from $b_0 \in \R$. Then, equation~$(1)$ admits a pathwise unique solution starting from $a_0 \in \R$ given by:
\[
A_t = a_0 + \int_0^t\sqrt{\Sigma_{11}(B_s) - \frac{\Sigma_{12}^2(B_s)}{\Sigma_{22}(B_s)}}\,dW^1_s + \int_0^t\frac{\Sigma_{12}(B_s)}{\sqrt{\Sigma_{22}(B_s)}}\,dW^2_s + \int_0^t\mu_1(B_s)\,ds + \frac{\Sigma_{12}^+ - \Sigma_{12}^-}{\Sigma_{22}^+ + \Sigma_{22}^-} L^0_t(B).
\]

Using a generalisation of the Yamada Watanabe Engelbert theorem (see~\cite[Proposition~2.10]{Kurtz2007}), uniqueness in law also holds for the process starting from $(a_0, b_0)$. By standard arguments, we deduce that the family of probability laws indexed by $(a_0, b_0) \in \R^2$ defines a Markov process.

\subsubsection*{Shape of the generator}

Let us prove that the generator can be written in divergence form. To this end, consider the continuous piecewise affine regularisations of $\Sigma$ and $\mu$ defined by: $\Sigma^n(b) = \Sigma^+$ if $b \geq \frac{1}{n}$, $\Sigma^n(b) = \Sigma^-$ if $b \leq \frac{-1}{n}$, and $\Sigma^n$ affine on $\left[-\frac{1}{n}, \frac{1}{n}\right]$ (similarly for $\mu^n$).

Let $(W^1, W^2)$ be a two-dimensional Brownian motion. By standard arguments, the SDE
\[
\begin{cases} 
 dA^n_t &= \sqrt{\Sigma^n_{11}(B^n_t) - \frac{{\Sigma_{12}^n}^2(B^n_t)}{\Sigma^n_{22}(B^n_t)}}\,dW^1_t + \frac{\Sigma^n_{12}(B^n_t)}{\sqrt{\Sigma^n_{22}(B^n_t)}}\,dW^2_t + \left(\mu^n_1(B^n_t) + \frac{1}{2}\frac{d}{dy}\Sigma^n_{12}(B^n_t)\right)dt, \\[1ex]
 dB^n_t &= \sqrt{\Sigma^n_{22}(B^n_t)}\,dW^2_t + \left(\mu^n_2(B^n_t) + \frac{1}{2} \frac{d}{dy}\Sigma^n_{22}(B^n_t)\right)dt
\end{cases}
\]
defines a Markov process with generator $\mathcal{L}_n = \frac{1}{2} \nabla \cdot (\Sigma^n(a) \nabla) + \mu^n(b) \cdot \nabla$.
By~\cite{Stroock1988}, this family of processes converges in law to the strong Markov process with generator given by~\eqref{divergence-form}. Moreover, using arguments similar to those in~\cite{Lejay2011}, one can show that $(A^n_t, B^n_t)$ converges in the $L^1$ norm to $(A_t, B_t)$ for all $t \geq 0$.
By identifying the limits, we conclude that~\eqref{divergence-form} is indeed the generator of $(A, B)$.

\subsection{Preliminary lemmas}\label{sub:preliminary}

To derive the desired functional equation, we first establish a few preliminary lemmas. We begin by stating Aronson-type estimates for the transition density, along with some elementary consequences. These inequalities hold since the process admits a generator in divergence form.

\begin{lemma}[Some estimates due to Aronson]
There exists a constant $M > 0$ such that, for all $t > 0$ and $v_0 \in \R^2$,
\begin{equation}\label{Aronson}
\frac{1}{Mt} \, e^{-M|v - v_0|^2 / t - Mt} \leq p_t^{v_0}(v) \leq \frac{M}{t} \, e^{-|v - v_0|^2 / Mt + Mt}.
\end{equation}
Let $g_t^{v_0}(v)$ be defined by
\[
g_t^{v_0}(v) := \int_0^t p_s^{v_0}(v) \, ds.
\]
Then,
\begin{equation}\label{utile!}
g_t^{v_0}(v) \leq \frac{M^2 \, e^{ -\frac{|v - v_0|^2}{Mt} + Mt }}{|v - v_0|^2} \, t.
\end{equation}
Finally, for all $v_0 = (a_0, b_0) \in \R^2$ and $x\nb{, y} \in \R$,
\begin{equation}\label{utile!2}
\E_{(a_0, b_0)}\left[e^{x A_t \nb{+ yB_t}}\right] \leq \nb{\pi M^2} e^{Mt+\frac{x^2 \nb{ + y^2}}{4}M t + a_0 x \nb{+ b_0y}}.
\end{equation}
\end{lemma}

\begin{proof}
Inequalities \eqref{Aronson} are proved in~\cite[Theorem II.3.8]{Stroock1988}. \nb{Inequations \eqref{utile!} and \eqref{utile!2} are then a straightforward consequence of \eqref{Aronson}. To obtain \eqref{utile!2}, note that \begin{align*}
\E_{(a_0, b_0)}\left[e^{x A_t + yB_t}\right]& \leq  \frac{Me^{Mt}}{t}\int_{a\in\R}\int_{b\in\R} e^{ax \nb{+ by}} e^{ - \frac1{Mt} (a-a_0)^2 - \frac1{Mt}(b-b_0)^2}dbda\\
&\leq \frac{Me^{Mt}}{t} e^{a_0x \nb{+b_0y}}\int_{a\in\R} e^{-\frac1{Mt} a^2 + ax} da \int_{b\in\R} e^{-\frac1{Mt} b^2 + by}db.
\end{align*}
Evaluating the Gaussian integrals gives the desired inequality.}
\end{proof}

We now relate the occupation measure on the axis $\{y = 0\}$, defined in~\eqref{eq:H(z0,A)}, to the values of the Green's function $g^{v_0}$ restricted to $\{y = 0\}$. This connection is made precise in the following lemma.

\begin{lemma}[Link between $H(v_0,\cdot)$ and $g^{v_0}(a,0)$]\label{Tlocal_densite}
Let $v_0 = (a_0, b_0) \in \R^2$ with $b_0 \neq 0$. Let $f: \R \to [0, +\infty)$ be a continuous nonnegative function. Then:
\begin{equation}\label{aaaaa}
\E_{v_0}\left[\int_0^{+\infty}  f(A_s) \, dL^0_s(B)\right] = \left(\frac{\Sigma_{22}^- + \Sigma_{22}^+}{2}\right) \int_\R g^{v_0}(a, 0) f(a) \, da.
\end{equation}
In particular, combining this with\nr{the asymptotic behavior as $t \to +\infty$ and} the definition~\eqref{def:phi}, we obtain the identity
\begin{equation}
\phi^{v_0}(x) = \left(\frac{\Sigma_{22}^- + \Sigma_{22}^+}{2}\right) \int_\R g^{v_0}(a, 0) \, e^{xa} \, da.
\end{equation}
\end{lemma}

\begin{proof}
    The proof reduces to showing that for all $t \geq 0$ and for every continuous, compactly supported, nonnegative function $f$, we have:
    \begin{equation}\label{abababab}
        \E_{v_0}\left[\int_0^t  f(A_s) \, dL^0_s(B)\right] = \left(\frac{\Sigma_{22}^- + \Sigma_{22}^+}{2}\right)\int_\R g^{v_0}_t(a,0)f(a) \, da.
    \end{equation}

    Let $f$ be such a function and let $\epsilon > 0$. Then, 
\begin{align}
    \E\left[\frac{1}{2\epsilon}\int_0^t f(A_s)\fc_{B_s \in [-\epsilon, \epsilon]} \, d\langle B \rangle_s\right] 
    &= \int_0^t \int_\R  f(a) \cdot \frac{1}{2\epsilon} \left( \Sigma_{22}^- \int_{[-\epsilon,0]} p_s^{v_0}(a,b) \, db + \Sigma_{22}^+ \int_{[0,\epsilon]} p_s^{v_0}(a,b) \, db \right) da \, ds \label{rhs_eq_techniquelemma_localtime}\\
    &\nb{=\int_0^t  \int_\R f(a)\frac{1}{2} \left( \Sigma_{22}^- \int_{[-1,0]} p_s^{v_0}( a,\epsilon b) \, db + \Sigma_{22}^+ \int_{[0,1]} p_s^{v_0}(a,\epsilon b) \, db \right) da \, ds.} 
\end{align}
\nb{The transition density $p_s^{v_0}$ is continuous on the $x-$axis. Then, for all $a \in \R$, we obtain: $$\frac12\Sigma_{22}^- \int_{[-1,0]} p_s^{v_0}( a,\epsilon b) \, db + \frac12\Sigma_{22}^+ \int_{[0,1]} p_s^{v_0}(a,\epsilon b) \, db \underset{\epsilon \to 0}{\longrightarrow} \frac{\Sigma_{22}^- + \Sigma_{22}^+}{2} p_s^{v_0}(a,0).$$}      Using the boundedness of $f$ and the Aronson estimates, \nb{the following inequality holds for $b \in [-1,1]$, $a \in \R$ and $s \in [0,t]$: $$|f(a) p_s^{v_0}( a,\epsilon b)| \leq ||f||_\infty\frac{M}{s}e^{-\frac{(a-a_0)^2}{Ms}+Ms}.$$ 
This last expression is integrable in $(s,a,b) \in [0,t]\times \R \times [0,1]$. Then, by the dominated convergence theorem,} the quantity \eqref{rhs_eq_techniquelemma_localtime} converges to
    \[
    \frac{\Sigma_{22}^- + \Sigma_{22}^+}{2} \int_0^t \int_\R f(a) p_s^{v_0}(a,0) \, da \, ds = \frac{\Sigma_{22}^- + \Sigma_{22}^+}{2} \int_\R g_t^{v_0}(a,0) f(a) \, da.
    \]
    as $\epsilon \to 0$. Moreover, we recall the almost sure identity:
    \begin{equation}
        \forall t \geq 0, \quad L_t^0(B) = \lim_{\epsilon \to 0} \int_0^t \frac{1}{2\epsilon} \fc_{B_s \in [-\epsilon, \epsilon]} \, d\langle B \rangle_s.
    \end{equation}

    If $s \mapsto H_s$ is a measurable step function with compact support, then:
    \begin{equation}\label{bbbbb}
        \frac{1}{2\epsilon} \int_0^t H_s \fc_{B_s \in [-\epsilon, \epsilon]} \, d\langle B \rangle_s 
        \overset{L^1,\ a.s.}{\underset{\epsilon \to 0}{\longrightarrow}} \int_0^t H_s \, dL^0_s(B).
    \end{equation}

    Since $s \mapsto f(A_s)$ is the uniform limit of step functions, the convergence in \eqref{bbbbb} remains valid when $H_s$ is replaced with $f(A_s)$.
\end{proof}

\nb{To state the functional equation \eqref{eq fonctionnelle}, we first show that the corresponding Laplace transforms converge. To do this, we will refer to the one-dimensional skewed bang-bang process, see Appendix~\ref{annexe:dim1} 
where the Green's functions are obtained explicitly and are shown to be asymptotically equivalent to some positive constants.
\begin{lemma}[Convergence of Laplace transforms]\label{lem:convergencePHI}
For $x, y, z \in \C$ with $Re(x) = 0$, $Re(y) < 0$ and $Re(z) > 0$, the Laplace transforms $\phi^{v_0}(x), \phi^{v_0}_+(x,y)$ and $\phi^{v_0}_-(x,z)$ are absolutely convergent.
\end{lemma}
\begin{proof} 
For the proof, we use the fact that the second coordinate $(B_t)_{t\geq0}$ of the process $(V_t)_{t\geq0}$ is a drifted one-dimensional skewed bang-bang Brownian motion and apply the results of Appendix~\ref{annexe:dim1}, see also \cite[Section 1.2.2.2]{Petit2026-fc}. The convergence of $\phi^{v_0}(0)$ reduces to $\E[L^0_\infty(B)] < +\infty$, which is proven in Appendix~\ref{annexe:dim1}. Hence, $\phi^{v_0}(x)$ is absolutely convergent.
We now consider $\phi^{v_0}_+(x,y)$. If $g_B^{b_0}$ denotes the Green's function of the process $(B_t)_{t\geq0}$, then by the  translation invariance along the $x$-axis of the two-dimensional process, we have:
$$\int_{u\in\R} g^{v_0}(\nb{a,b})d\nb{a} = g^{b_0}_B(\nb{b}),$$ 
where $g^{b_0}_B(\nb{b})$ is equivalent to a positive constant when $v \to +\infty$ as shown in Appendix~\ref{annexe:dim1}.
For $x = 0$ and real $y < 0$, the quantity
\begin{align}
\phi_+^{v_0}(\nb{0},y) &= \int_{(u,v)\in \R \times \R_+} e^{x\nb{a}+y\nb{b}}g^{v_0}(\nb{a},\nb{b})d\nb{a}d\nb{b} = \int_{\nb{b}\in\R_+}e^{y\nb{b}} g^{b_0}_B(\nb{b})d\nb{b}
\end{align}
is thus finite. The Laplace transform $\phi_+^{v_0}(x,y)$ is then absolutely convergent for $x, y \in \C$ satisfying $Re(x) = 0$ and $Re(y) < 0$. The proof for $\phi^{v_0}_-(x,z)$ is symmetrical.
\end{proof}}

\subsection{Short version of the proof of the functional equation}\label{sub:formalproof}

We now present a brief argument to establish Proposition~\ref{prop:eqfonc}. The idea of the proof is to apply Itô's formula to
$$f(a,b) = e^{ax + by}\fc_{b\geq0} + e^{ax + bz}\fc_{b<0}.$$
Note that $f \notin \mathcal{C}^2(\R^2)$. In this section, we assume that all local martingales are true martingales, ignore convergence issues, and apply the standard Itô formula despite the lack of smoothness of $f$.
We compute:
\begin{equation}
\nabla f(a,b) := \left(
    \begin{array}{ll}
      x f(a,b) \\
      ye^{ax + by}\fc_{b\geq0} + ze^{ax + bz}\fc_{b<0}
    \end{array}
\right)
\end{equation}
and
$$\nabla^2 f(a,b) = \begin{pmatrix}
x^2f(a,b)  & xye^{ax + by}\fc_{b\geq0} + xze^{ax + bz}\fc_{b<0}\\
xye^{ax + by}\fc_{b\geq0} + xze^{ax + bz}\fc_{b<0} & y^2e^{ax + by}\fc_{b\geq0} + z^2e^{ax + bz}\fc_{b<0}  + (y-z)e^{ax}\delta_{b=0}
\end{pmatrix}.$$
Formally, Itô's formula yields
\begin{equation}\label{itoto}
f(V_t) = f(V_0) + \int_0^t \nabla f(V_s) \, dV_s + \frac{1}{2}\sum_{i=1}^2\sum_{j=1}^2\int_0^t \partial_i\partial_j f(V_s) \, d\langle V^i, V^j\rangle_s
\end{equation}
(with notation $V^1 = A$ and $V^2 = B$).
Since $dV_t = \sigma(B_t)dW_t + \mu(B_t)dt + q\,dL_t^0(B)$, we get
$$\E\left[ \int_0^t \nabla f(V_s) \, dV_s\right] = (x\mu_1^+ + y\mu_2^+)  \E\left[\int_0^t \fc_{B_s\geq0} e^{xA_s + yB_s} \, ds\right] + (x\mu_1^-  +  z\mu_2^-) \E\left[\int_0^t \fc_{B_s<0}  e^{xA_s + zB_s} \, ds\right]$$
$$ + \left(q_1x + \frac{q_2(y+z)}{2}\right) \E\left[\int_0^t  e^{xA_s} \, dL_s^0(B)\right].$$

Now, let us consider the second derivative terms. Similarly, we have:
$$\E\left[ \frac{1}{2}\sum_{i=1}^2\sum_{j=1}^2\int_0^t \partial_i\partial_j f(V_s) \, d\langle V^i, V^j\rangle_s\right] = \frac{1}{2}(x^2\Sigma_{11}^+ + 2x y \Sigma_{12}^+ + y^2\Sigma_{22}^+)  \E\left[\int_0^t \fc_{B_s\geq0} e^{xA_s + yB_s} \, ds\right]$$
$$ +  \frac{1}{2}(x^2\Sigma_{11}^- + 2x z \Sigma_{12}^- + z^2\Sigma_{22}^-) \E\left[\int_0^t \fc_{B_s<0}  e^{xA_s + zB_s} \, ds\right]$$
$$+ \left(\frac{y-z}{2}\right) \E\left[\int_0^t  e^{xA_s }\delta_{B_s=0} \, ds\right].$$

The last term can be interpreted as $\left(\frac{y-z}{2}\right)\E\left[\int_0^t  e^{xA_s} \, dL^0_t(B)\right]$. Then, by rearranging the terms in the expected value of \eqref{itoto}, we obtain
\begin{equation}
\gamma_-(x,z)\E\left[\int_0^t \fc_{B_s<0}  e^{xA_s + zB_s} \, ds\right] + \gamma_+(x,y)\E\left[\int_0^t \fc_{B_s>0}  e^{xA_s + yB_s} \, ds\right] + \gamma(x,y,z)\E\left[\int_0^t  e^{xA_s} \, dL_s^0(B)\right]
\end{equation}
$$= \E[f(V_t)] - \left(e^{xa_0 + yb_0}\fc_{b_0\geq 0} + e^{xa_0 + zb_0}\fc_{b_0<0}\right).$$

\nr{Assuming that $x$, $y$, and $z$ are chosen such that $\E[f(V_t)] \to 0$ as $t \to +\infty$ and that the Laplace transforms converge, we obtain the desired equation.}
\nb{By the choice of $x, y, z$, the preceding expected values converge to the corresponding Laplace transforms (see Lemma~\ref{lem:convergencePHI}) and $\E[f(V_t)] \to 0$ as $t\to+\infty$.}

\subsection{Detailed version of the proof of the functional equation}\label{sub:detailedproof}
In this section, we provide a detailed justification for the reasoning presented in Section~\ref{sub:formalproof} when $b_0 \neq 0$ (see Remark~\ref{rem:b0neq0} for the case $b_0 = 0$). Note that an Itô–Tanaka formula exists in dimension~$2$ for non-smooth functions \cite{Feng2007}. However, it is not straightforward to use, so instead we opted for an approximation using smooth functions.
We then approximate the preceding function $f$ by smooth functions $f_\epsilon$ defined by
\begin{equation}\label{fe}
f_\epsilon(a,b) = e^{ax + bz} + \psi(b/\epsilon)(e^{ax + by} - e^{ax + bz}),
\end{equation}
where $\psi$ is an increasing $C^\infty$ function satisfying
$\psi(u) := \left\{
    \begin{array}{lll}
      0 \quad &\textnormal{if}\quad u \leq -1, \\
      1 \quad &\textnormal{if}\quad u \geq 1, \\
      \frac{1}{2} \quad &\textnormal{if}\quad u = 0.
    \end{array}
\right.$
Note that $f_\epsilon(a,b) = e^{ax + by}$ if $b \geq \epsilon$ and $f_\epsilon(a,b) = e^{ax + bz}$ if $b \leq -\epsilon$. The Itô formula can now be applied to the $C^\infty$ function $f_\epsilon$: for all $t \geq 0$, we have
\begin{equation}\label{ito}
f_\epsilon(V_t) = f_\epsilon(V_0) + \int_0^t \nabla f_\epsilon(V_s) \, dV_s + \frac{1}{2} \sum_{i=1}^2 \sum_{j=1}^2 \int_0^t \partial_i\partial_j f_\epsilon(V_s) \, d\langle V^i, V^j\rangle_s.
\end{equation}

\textbf{Step 1 : Study of integrals against $dV_s$.} Note that
\begin{equation}\label{fe'}
\nabla f_\epsilon(a,b) := \left(
    \begin{array}{ll}
      x f_\epsilon(a,b) \\
      ze^{ax + bz} + \frac{1}{\epsilon}\psi'(b/\epsilon)(e^{ax + by} - e^{ax + bz}) + \psi(b/\epsilon)(ye^{ax + by} - ze^{ax + bz})
    \end{array}
\right).
\end{equation}
Then,
\begin{equation}\label{chiant1}
\int_0^t \nabla f_\epsilon(V_s)\fc_{B_s\geq0} \, dV_s  = x \int_0^t \fc_{B_s\geq0} f_\epsilon(V_s) \, dA_s
\end{equation}
\[
+ \int_0^t \fc_{B_s\geq0} \left( ze^{xA_s + zB_s} + \frac{1}{\epsilon}\psi'(B_s/\epsilon)(e^{xA_s + yB_s} - e^{xA_s + zB_s}) + \psi(B_s/\epsilon)(ye^{xA_s + yB_s} - ze^{xA_s + zB_s}) \right) dB_s,
\]
where $\fc_{B_s\geq0} \, dA_s = (\Sigma_{11}^+ dW^1_s + \Sigma_{12}^+ dW^2_s + \mu_1^+ ds) + q_1 \, dL_t^0(B)$ and $\fc_{B_s\geq0} dB_s = (\Sigma_{12}^+dW^1_s + \Sigma_{22}^+dW^2_s + \mu_2^+ds) + q_2 dL_t^0(B)$. Then, by Lemma~\ref{utile!2}, the Brownian integrals are martingales (not just local martingales), so their expectations vanish. Therefore,

\[\E_{v_0}\left[\int_0^t \nabla f_\epsilon(V_s)\fc_{B_s\geq0} \, dV_s\right] = x\mu_1^+ \E_{v_0}\left[\int_0^t \fc_{B_s\geq0} f_\epsilon(V_s) \, ds\right] + \,y\mu_2^+ \E_{v_0}\left[\int_0^t \fc_{B_s\geq\epsilon} e^{xA_s + yB_s} \, ds\right]\]
\begin{equation}\label{chiant2}
+ \,\mu_2^+ \E_{v_0}\Big[\int_0^t \fc_{0\leq B_s<\epsilon} \Big( ze^{xA_s + zB_s} + \frac{1}{\epsilon}\psi'(B_s/\epsilon)(e^{xA_s + yB_s} - e^{xA_s + zB_s}) 
\end{equation}
\[
\, + \psi(B_s/\epsilon)(ye^{xA_s + yB_s} - ze^{xA_s + zB_s}) \Big) ds\Big] + \left(q_1x + \frac{q_2(y+z)}{2}\right) \E_{v_0}\left[\int_0^t e^{xA_s} \, dL_t^0(B)\right].
\]

\nb{Furthermore, we have $$\fc_{0\leq B_s<\epsilon}\left|  ze^{xA_s + zB_s} + \psi(B_s/\epsilon)(ye^{xA_s + yB_s} - ze^{xA_s + zB_s})\right| \leq \left|  ze^{xA_s + zB_s}\right| +  \left|ye^{xA_s + yB_s}\right|  +  \left|ze^{xA_s + zB_s}\right|.$$
The right hand side of this last expression is $\P_{(a_0,b_0)}-$integrable by inequality~\eqref{utile!2}. Then, by the dominated convergence theorem, we obtain:}
\nr{By the dominated convergence theorem and inequality~\eqref{utile!2},}
\[
\E_{v_0}\left[\int_0^t \fc_{0\leq B_s<\epsilon} \left( ze^{xA_s + zB_s} + \psi(B_s/\epsilon)(ye^{xA_s + yB_s} - ze^{xA_s + zB_s}) \right) ds\right] \underset{\epsilon\to 0}{\longrightarrow} 0.
\]
\nb{Applying the dominate convergence theorem similarly, we have:}
\[
\E_{v_0}\left[\int_0^t \fc_{B_s\geq\epsilon} e^{xA_s + yB_s} \, ds\right] \underset{\epsilon\to 0}{\longrightarrow} \E_{v_0}\left[\int_0^t \fc_{B_s\geq0} e^{xA_s + yB_s} \, ds\right].
\]

Moreover, by the definition of $g_t$, we have:
\[
\E_{v_0}\left[\int_0^t \fc_{0\leq B_s<\epsilon} \frac{1}{\epsilon}\psi'(B_s/\epsilon)(e^{xA_s + yB_s} - e^{xA_s + zB_s}) \, ds\right] = \int_{\R} \int_0^\epsilon \frac{1}{\epsilon}\psi'(b/\epsilon)(e^{ax + by} - e^{ax + bz})g_t(a, b) \, db da.
\]

Since $b_0 \neq 0$, estimates~\eqref{utile!} yield the following limit for all $a \in \R$ as $\epsilon \to 0$:
\begin{align}
\int_0^\epsilon \frac{1}{\epsilon} \psi'(b/\epsilon)(e^{ax + by} - e^{ax + bz})g_t(a, b) \, db
&= \int_0^1 \psi'(u)(e^{ax + \epsilon u y} - e^{ax + \epsilon u z})g_t(a, \epsilon u) \, du \\
&\underset{\epsilon\to 0}{\longrightarrow} (e^{ax} - e^{ax})g_t(a, 0) = 0. \label{eq:usesb0neq0}
\end{align}

(Note that $\frac{1}{\epsilon} \psi'(\cdot/\epsilon) \underset{\epsilon \to 0}{\longrightarrow} \delta_0$ in the sense of distributions.) Integrating in $a \in \R$ and using~\eqref{utile!} again, we get
\begin{equation}\label{g(a,0)}
\E_{v_0}\left[\int_0^t \fc_{0\leq B_s<\epsilon} \frac{1}{\epsilon}\psi'(B_s/\epsilon)(e^{xA_s + yB_s} - e^{xA_s + zB_s}) \, ds\right] \underset{\epsilon \to 0}{\longrightarrow} 0.
\end{equation}

By a similar analysis on $\int_0^t \nabla f_\epsilon(V_s)\fc_{B_s<0} \, dV_s$, the following limit holds:
\begin{equation}\label{eq:nabla}
\E_{v_0}\left[\int_0^t \nabla f_\epsilon(V_s) \, dV_s\right] \underset{\epsilon \to 0}{\longrightarrow} (x\mu_1^+ + y\mu_2^+) \E_{v_0}\left[\int_0^t \fc_{B_s \geq 0} e^{xA_s + yB_s} \, ds\right] \,
\end{equation}
\[
+\, (x\mu_1^- + z\mu_2^-) \E_{v_0}\left[\int_0^t \fc_{B_s < 0} e^{xA_s + zB_s} \, ds\right] \,+\, \left(q_1x + \frac{q_2(y + z)}{2}\right) \E_{v_0}\left[\int_0^t e^{xA_s} \, dL_t^0(B)\right].
\]

\textbf{Step 2: Study of integrals against $d\langle V^i, V^j\rangle_s$.}

Since $\partial_1f_\e(a,b) = xf_\e(a,b)$, we can similarly show that 
\begin{equation}
\E_{v_0}\left[\int_0^t \partial_1\partial_1  f_\epsilon(V_s)d\langle A, A\rangle_s\right] \underset{\e\to 0}{\longrightarrow} \Sigma_{11}^+x^2\E_{v_0}\left[\int_0^t \fc_{B_s\geq0} e^{xA_s + yB_s} ds \right] + \Sigma_{11}^-x^2\E_{v_0}\left[\int_0^t \fc_{B_s<0} e^{xA_s + zB_s} ds \right]
\end{equation}
and 
\begin{equation}
\E_{v_0}\left[\int_0^t \partial_1\partial_2  f_\epsilon(V_s)d\langle A, B\rangle_s\right] \underset{\e\to 0}{\longrightarrow} \Sigma_{12}^+xy\E_{v_0}\left[\int_0^t \fc_{B_s\geq0} e^{xA_s + yB_s} ds \right] + \Sigma_{12}^-xz\E_{v_0}\left[\int_0^t \fc_{B_s<0} e^{xA_s + zB_s} ds \right].
\end{equation}

Now, let us study the last term $i = j = 2$. We have $d\langle B, B\rangle_t = \Sigma_{22}^+\fc_{B_t \geq 0}dt + \Sigma_{22}^-\fc_{B_t < 0}dt$, and
\begin{align*}
\partial_2\partial_2  f_\epsilon(a,b) &=  z^2e^{ax + bz} + \frac{1}{\epsilon^2}\psi''(b/\epsilon)(e^{ax + by} - e^{ax + bz}) \\
&\quad + \frac{2}{\epsilon}\psi'(b/\epsilon)(ye^{ax + by} - ze^{ax + bz}) + \psi(b/\epsilon)(y^2e^{ax + by} - z^2e^{ax + bz}).
\end{align*}

By the dominated convergence theorem and inequality \eqref{utile!2}, we obtain similarly:

\[\E_{v_0}\left[\int_0^t \left(z^2e^{A_sx + B_sz}+\psi(B_s/\epsilon)(y^2e^{A_sx + bB_s} - z^2e^{A_sx + B_sz})\right) d\langle B, B\rangle_s\right] \]
\[\underset{\e\to 0}{\longrightarrow} \Sigma_{22}^+y^2\E_{v_0}\left[\int_0^t \fc_{B_s\geq0} e^{xA_s + yB_s}ds\right] \, +\, \Sigma_{22}^-z^2\E_{v_0}\left[\int_0^t \fc_{B_s<0} e^{xA_s + zB_s} ds\right] \].

By the same arguments as for \eqref{g(a,0)}, we have
\begin{align}\label{eq:usesb0neq02}
\E_{v_0}\left[\int_0^t \fc_{-\e \leq B_s<\e} \frac{2}{\epsilon}\psi'(B_s/\epsilon)(ye^{xA_s + yB_s} - ze^{xA_s + zB_s}) d\langle B, B\rangle_s\right] 
\underset{\e\to 0}{\longrightarrow} (\Sigma_{22}^- + \Sigma_{22}^+)(y-z)\int_{\R}e^{ax}g_t(a, 0)da.
\end{align}

Additionally, 
\begin{align*}
\E_{v_0}&\left[\int_0^t \fc_{0\leq B_s<\e} (1/\epsilon^2)\psi''(B_s/\epsilon)(e^{xA_s + yB_s} - e^{xA_s + zB_s})ds\right] \\
&=\int_{\R} \int_0^\epsilon(1/\epsilon^2)\psi''(b/\epsilon) (e^{xa + yb} - e^{xa + zb})g_t(a,b) dbda \\
&=\int_{\R} \int_{0}^1\psi''(b) \frac{e^{ax +b\epsilon y} - e^{ax + b\epsilon z}}{\epsilon} 
g_t(a, \epsilon b)dbda \\
 &\underset{\e\to 0}{\longrightarrow}(y-z)\int_{\R}e^{ax}g_t(a,0)da \int_{0}^1b\psi''(b)db\\
 &=(y-z)(\psi(0) - \psi(1))\int_{\R}e^{ax}g_t(a,0)da = \frac{(z-y)}{2}\int_{\R}e^{ax}g_t(a,0)da
\end{align*}

Here, the limit is once again justified by the estimates in \eqref{utile!}. Combining this with Lemma~\ref{Tlocal_densite}, we have shown that
\begin{equation}\label{eq:delta}
\frac{1}{2}\sum_{i=1}^2\sum_{j=1}^2\E_{v_0}\left[\int_0^t \partial_i\partial_j f_\epsilon(V_s) d\langle V^i, V^j\rangle_s \right]
\underset{\e\to 0}{\longrightarrow} \frac{1}{2}(\Sigma_{11}^+x^2 + 2\Sigma_{12}^+xy + \Sigma_{22}^+y^2)\E_{v_0}\left[\int_0^t \fc_{B_s\geq0} e^{xA_s + yB_s} ds \right]
\end{equation}
\[
+ \frac{1}{2}(\Sigma_{11}^-x^2 + 2\Sigma_{12}^-xy + \Sigma_{22}^-y^2)\E\left[\int_0^t \fc_{B_s<0} e^{xA_s + zB_s} ds \right] + (y-z)\E_{v_0}\left[\int_0^t  e^{xA_s} dL^0_s(B)\right].
\]

\textbf{Step 3: Sum up and take the limit as $t \to +\infty$.}

Combining \eqref{ito}, \eqref{eq:nabla}, and \eqref{eq:delta}, we obtain:
\begin{equation}\label{tfixe}
\E_{v_0}[f(V_t)] = \gamma_-(x,z)\E_{v_0}\left[\int_0^t \fc_{B_s<0} e^{xA_s + zB_s} ds \right] + \gamma_+(x,y)\E_{v_0}\left[\int_0^t \fc_{B_s>0} e^{xA_s + yB_s} ds \right] 
\end{equation}
\[
+ \left(q_1x + \frac{y(1 + q_2) + z(q_2 - 1)}{2}\right)\E_{v_0}\left[\int_0^t  e^{xA_t} dL^0_s(B)\right] + e^{xa_0 + yb_0\fc_{b_0>0} + zb_0\fc_{b_0<0}}.
\]
\nb{By the dominated convergence theorem, we have $\E_{v_0}[f(V_t)] \to 0$ (since $f(V_t)$ is bounded by $1$). Finally, by Lemma~\ref{lem:convergencePHI}, all the expected values on the right-hand side converge to $\phi_-^{v_0}(x,z)$, $\phi_+^{v_0}(x,y)$, and $\phi^{v_0}(x)$, respectively. This yields the desired functional equation.}
\nr{Moreover, note that $f(V_t) \leq e^{xA_t}$ since $y < 0$ and $z > 0$. By Lemma~\ref{transience}, $\E_{v_0}[f(V_t)] \underset{t\to+\infty}{\longrightarrow} 0$ if $-\eta <x<0$.
Note that $x < 0$, $y < 0$, $z > 0$ can be taken arbitrarily close to $0$ such that the conditions $\gamma_+(x,y) < 0$, $\gamma_-(x,z) < 0$, $y-z < 0$, and $q_1x + \frac{y(1 + q_2) + z(q_2 - 1)}{2} < 0$ are satisfied. For such $x, y, z$, the quantity
\begin{align*}
&\gamma_-(x,z)\E_{v_0}\left[\int_0^t \fc_{B_s<0} e^{xA_s + zB_s} ds \right] + \gamma_+(x,y)\E_{v_0}\left[\int_0^t \fc_{B_s>0} e^{xA_s + yB_s} ds \right] \\
&\quad + \left(q_1x + \frac{y(1 + q_2) + z(q_2 - 1)}{2}\right)\E_{v_0}\left[\int_0^t  e^{xA_t} dL^0_s(B)\right]
\end{align*}
converges to
\[
\gamma_-(x,z)\phi^{v_0}_-(x,z) + \gamma_+(x,y)\phi^{v_0}_+(x,y) + \left(q_1x + \frac{y(1 + q_2) + z(q_2 - 1)}{2}\right)\phi^{v_0}(x) \in [-\infty, 0)
\]
as $t \to +\infty$ (if some terms diverge, they are equal to $-\infty$ so the sum remains well defined). By \eqref{tfixe}, the previous sum is finite when $x \in (-\eta, 0)$, so $\phi^{v_0}_-(x,z)$, $\phi^{v_0}_+(x,y)$, and $\phi^{v_0}(x)$ are finite for $x \in (-\eta, 0)$, $y < 0$ and $z > 0$ arbitrarily close to $0$.
Now that we know the Laplace transforms converge for such $x, y, z$, the asymptotics of \eqref{tfixe} as $t\to+\infty$ yield the desired functional equation.}

\begin{rem}[Choice of $b_0 \neq 0$]\label{rem:b0neq0}
The choice $b_0 \neq 0$ for the initial point $(a_0, b_0)$ is made for essentially technical reasons. Indeed, it ensures that $g_t$ remains finite on the strip $\R \times [-\epsilon, \epsilon]$ for $\epsilon > 0$ sufficiently small; this avoids unnecessary difficulties when taking the limit $\epsilon \to 0$ in \eqref{eq:usesb0neq0} and \eqref{eq:usesb0neq02}. The functional equation also holds for $b_0 = 0$, which can be proved by taking the limit of \eqref{eq fonctionnelle} as $v_0 \to (a_0, 0)$.
\end{rem}

\section{Proofs of Theorems~\ref{prop:laplace explicitee} and \ref{alpha=0}: Laplace transforms and asymptotics on the axis}\label{sec:alpha=0}
Since the proof of Theorem~\ref{alpha=0} is quite independent of the proofs of the other theorems, it has been placed in a separate section.
Let us recall the expressions for $Y^\pm(x)$ and $Z^\pm(x)$ given in \eqref{Ypm} and \eqref{Zpm}.

\nr{\begin{lemma}
There exists $\e > 0$ such that for all $u \in [-\e, 0]$ and $v \in \R$, 
\[
\Re\bigl(Y^-(u + iv)\bigr) < 0 \quad \text{(resp.} \quad \Re\bigl(Z^+(u + iv)\bigr) > 0\text{).}
\]
\end{lemma}}
\nb{\begin{lemma}
For all $s \in \R$, 
\[
\Re\bigl(Y^-(is)\bigr) < 0 \quad \text{and} \quad \Re\bigl(Z^+(is)\bigr) > 0.
\]
\label{lemmesansnom?}
\end{lemma}}

\begin{proof}
By elementary considerations (see \cite[Lemma 3.1]{Franceschi_2024}, correcting an \nr{error}\nb{constant} without any consequence), \nb{for all $u,s \in \R$,} we have
\begin{equation}\label{Re}
\Re(Y^{\pm}(u+is))=
 \frac{1}{\Sigma^+_{22 }}\Big(-\Sigma^+_{12} u -\mu^+_2 
\end{equation}
\[\pm \frac{\sqrt{\det(\Sigma^+)}}{\sqrt{2}} 
   \sqrt{  (u-x^+_{min}) (x^+_{max}-u)+s^2 +| (u+is - x^+_{min})(x^+_{max}-u-is ) |  } \Big).\]
\nr{Then, $\Re(Y^{-}(u+is)) \leq \Re(Y^{-}(u))$ for all $s \in \R$. The inequality $Y^{-}(0)= -\frac{2\mu^+_2}{\Sigma^+_{22}}<0$ yields the first inequality.}\nb{Then, $\Re(Y^{-}(is)) \leq Y^{-}(0) = -\frac{2\mu^+_2}{\Sigma^+_{22}}<0$ for all $s \in \R$.} The second inequality is obtained symmetrically since $\mu_2^- < 0$. 
\end{proof}

We are now able to prove Theorem~\ref{prop:laplace explicitee}.

\begin{proof}[Proof of Theorem~\ref{prop:laplace explicitee}]
\nr{For $x \in (-\e, 0)$, we have $Y^-(x) < 0$ and $Z^+(x) > 0$. We can then substitute $(x, Y^-(x), Z^+(x))$ into the functional equation \eqref{eq fonctionnelle} to derive the equality \eqref{laplace explicitee} for $x \in (-\e, 0)$. Since $\phi$ is a Laplace transform, equation \eqref{laplace explicitee} holds in the entire domain of convergence of $\phi$ by holomorphicity. Since the expression is meromorphic on $\C \setminus \bigl((-\infty,\m]\cup [\M, +\infty)\bigr)$, it remains only to show that the equation $\gamma(x, Y^-(x), Z^+(x)) = 0$ has no solution on $\C \setminus \bigl((-\infty,\m)\cup (\M, +\infty)\bigr)$.}

\nb{By Lemma~\ref{lemmesansnom?} and Proposition~\ref{prop:eqfonc}, we can then substitute $(x, Y^-(x), Z^+(x))$ into the functional equation \eqref{eq fonctionnelle}. This yields equality~\eqref{laplace explicitee} for $x \in i\R$.} Using expressions \eqref{eq:q}, \eqref{Ypm} and \eqref{Zpm}, we can write $\gamma(x, Y^-(x), Z^+(x))$ as:
\begin{align}\label{calcul_de_gamma}
\gamma(x, Y^-(x), Z^+(x)) &= \frac{1}{\Sigma_{22}^+ + \Sigma_{22}^-}\big((\Sigma_{12}^+ - \Sigma_{12}^-) x + \Sigma_{22}^+ Y^-(x) - \Sigma_{22}^- Z^+(x)\big) \\
&= \frac{1}{\Sigma_{22}^+ + \Sigma_{22}^-} \Big(-\mu_2^+ + \mu_2^- -  \sqrt{(\Sigma^+_{12}{}^2 - \Sigma^+_{11} \Sigma^+_{22}) x^2 + 2 (\mu^+_2 \Sigma^+_{12} - \mu^+_1 \Sigma^+_{22}) x + (\mu^+_2)^2} \nonumber \\
&\quad - \sqrt{(\Sigma^-_{12}{}^2 - \Sigma^-_{11} \Sigma^-_{22}) x^2 + 2 (\mu^-_2 \Sigma^-_{12} - \mu^-_1 \Sigma^-_{22}) x + (\mu^-_2)^2} \Big). \nonumber
\end{align}
Furthermore, by Assumption~\eqref{hyp:drift}, $-\mu_2^+ + \mu_2^- < 0$ and by the same arguments as for \eqref{Re},
\[
\Re\left(\sqrt{(\Sigma^\pm_{12}{}^2 - \Sigma^\pm_{11} \Sigma^\pm_{22}) x^2 + 2 (\mu^\pm_2 \Sigma^\pm_{12} - \mu^\pm_1 \Sigma^\pm_{22}) x + (\mu^\pm_2)^2}\right) \geq 0.
\]
Considering the real part, we see that $\gamma(x, Y^-(x), Z^+(x))$ does not vanish on $\C \setminus \bigl((-\infty,\m]\cup [\M, +\infty)\bigr)$\nb{. The right-hand side of \eqref{laplace explicitee} is then a holomorphic function on the strip $\{x \in \C \mid \widetilde x_b < Re (x) < x_b\}$. Let us denote it by $\zeta(x)$ and show that it is the Laplace transform of a function $f$ on the strip. We  note that $\sup_{\widetilde x_b <u< x_b} |\zeta(u+is)|=O_{|s|\to+\infty} \left(\frac{1}{s}\right)$  which yields
$$  \int_{ \widetilde x_b <u< x_b} \int_{s \in {\bf R}} |\zeta(u+is)|^2 ds du <+ \infty.   $$
Then, by the Paley-Wiener Theorem on vertical strips (see \cite[Theorem 2.1]{Harper2010-jp}), $\zeta $ is the Laplace transform of a function $f$ on the strip. In particular, if $Re(x)=0$, it is also the Fourier transform of the same function. But by Lemma~\ref{Tlocal_densite} and \eqref{laplace explicitee} which was proven for $Re(x) = 0$, it is also the Fourier transform of $\frac{\Sigma_{22}^+ + \Sigma_{22}^-}{2} g^{v_0}(\cdot, 0)$. Then by the uniqueness of Fourier transform, $f=\frac{\Sigma_{22}^+ + \Sigma_{22}^-}{2} g^{v_0}(\cdot, 0)$. It follows that  $\phi^{v_0}$, defined as the Laplace transform of $\frac{\Sigma_{22}^+ + \Sigma_{22}^-}{2} g^{v_0}(\cdot, 0)$,  does converge  and equals the right-hand side of \eqref{laplace explicitee} on the desired strip.} \nr{, and the first result follows.}

\nb{For the last statement, we substitute the explicit expression \eqref{laplace explicitee} of $\phi^{v_0}$ just obtained into \eqref{eq fonctionnelle}. For $x \in i\R$, $z = Z^+(x)$ (with $Re(Z^+(x)) > 0$ by Lemma~\ref{lemmesansnom?}), and $y < 0$, we derive:} 

\nr{Equation \eqref{eq fonctionnelle}, combined with the explicit expression of $\phi^{v_0}$ (see Theorem~\ref{prop:laplace explicitee}), provides an explicit formula for $\phi^{v_0}_+(x,y)$. Indeed, by substituting $z = Z^+(x) > 0$ (with $x < 0$ small enough) and $y < 0$ into \eqref{eq fonctionnelle}, we obtain}
\begin{equation}\label{laplace explicite1}
\phi^{v_0}_+(x,y) = \frac{-\gamma(x,y,Z^+(x))\phi^{v_0}(x) - e^{xa_0 + yb_0 \fc_{b_0 \geq 0} + Z^+(x) b_0 \fc_{b_0 < 0}}}{\gamma_+(x,y)}.
\end{equation}

\end{proof}

We now state a final preparatory lemma prior to the proof of Theorem~\ref{alpha=0}.

\begin{lemma}[Behaviour of $\phi^{v_0}$ at $\M$]\label{lem:DLphi}
Let $C_0$ and $f_0(v_0)$ be defined as in Theorem~\ref{alpha=0}. Then,
    $$\phi^{v_0}(x) = \phi^{v_0}(\M) - 2\sqrt{\pi}\left(\frac{\Sigma_{22}^+ + \Sigma_{22}^-}{2}\right)C_0f_0(v_0)\sqrt{\M-x} + o(\sqrt{\M-x}) $$
    as $x \to x_b$, where $f_0(v_0)$ is defined in Theorem~\ref{alpha=0}.
\end{lemma}

\begin{proof}
    Without loss of generality, we assume in this proof that ${\M} = x_{max}^+ < x_{max}^-$. We define $C_y$ as $C_y=\frac{\sqrt{\det(\Sigma^+)(\M - x_{min}^+)}}{\Sigma_{22}^+}$ so that $Y^-(x) = Y^-(\M) - C_y\sqrt{\M-x} + o(\sqrt{\M-x}).$ 

Now, let us perform straightforward computations using the explicit expression of $\phi^{v_0}$ given by \eqref{laplace explicitee}. We have:
\begin{align*}
\frac{1}{\gamma(x,Y^-(x), Z^+(x))} &=\frac{1}{\gamma(\M,Y^-(\M), Z^+(\M)) - C_y\left(\frac{1+q_2}{2}\right)\sqrt{\M-x}(1 + o(1))}\\
 &= \frac{1}{\gamma(\M,Y^-(\M), Z^+(\M))}\times \frac{1}{1 - \frac{C_y\left(\frac{1+q_2}{2}\right)}{\gamma(\M,Y^-(\M), Z^+(\M))}\sqrt{\M-x}(1 + o(1))} \\
&= \frac{1}{\gamma(\M,Y^-(\M), Z^+(\M))} \left(1+ \frac{C_y\left(\frac{1+q_2}{2}\right)}{\gamma(\M,Y^-(\M), Z^+(\M))}\sqrt{\M-x}(1 + o(1)) \right).
\end{align*}

Furthermore,
\begin{align*}
e^{Y^-(x)b_0\fc_{b_0 > 0}} = e^{(Y^-(\M) -\sqrt{\M - x}( C_y +o(1)))b_0\fc_{b_0 > 0}} 
&= e^{Y^-(\M)b_0\fc_{b_0 > 0}}(1 -\sqrt{\M - x}(C_y+o(1))b_0\fc_{b_0 > 0}).
\end{align*}

Multiplying the two quantities together, we obtain:
\begin{align*}
\phi^{v_0}(x) - \phi^{v_0}&(\M) = \frac{C_y}{\gamma(\M,Y^-(\M), Z^+(\M))}\left(b_0\fc_{b_0 > 0} - \frac{\frac{1+q_2}{2}}{\gamma(\M,Y^-(\M), Z^+(\M))}\right) \\
&\quad \times e^{a_0\M + Y^-(\M)b_0\fc_{b_0 > 0}+Z^+(\M)b_0\fc_{b_0 < 0}} \sqrt{x - \M} + o(\sqrt{x- \M}),
\end{align*}
which is the desired formula.
\end{proof}

Since $\phi^{v_0}$ is the Laplace transform of $\frac{\Sigma_{22}^+ + \Sigma_{22}^-}{2}g^{v_0}(\cdot,0)$ (see Lemma~\ref{Tlocal_densite}), we can now derive the asymptotics of $g^{v_0}(r,0)$ as $r \to +\infty$ (i.e., Theorem~\ref{alpha=0}) using the Tauberian Theorem~\ref{lapinv}.
\begin{proof}[Proof of Theorem~\ref{alpha=0}]
By symmetry, we only consider the asymptotics as $r \to +\infty$. First, suppose that $a_0 = 0$. Note that the first and third conditions of Theorem~\ref{lapinv} are satisfied by Theorem~\ref{prop:laplace explicitee} and Lemma~\ref{lem:DLphi}. Let us verify the second condition of Theorem~\ref{lapinv}.

From \eqref{Re}, there exist $\delta > 0$ and constants $M, c > 0$ such that $\Re(Y^-(u+is)) \leq M - cs$ for all $u, s \in \R$ satisfying $|u| \leq \frac{1}{\tan(\delta)}|s|$. Then, $e^{b_0\Re(Y^-(x))\fc_{b_0>0}}$ is bounded for $x \in G^+_\delta({\M})$. The same reasoning applies to $e^{b_0\Re(Z^+(x))\fc_{b_0<0}}$.

Furthermore, from \eqref{Re}, possibly after choosing $\delta$ close enough to $\pi/2$, we have
$$|\gamma(x, Y^-(x), Z^+(x))| \underset{|x| \to +\infty \atop x \in G^+_\delta({\M})}{\longrightarrow} +\infty.$$
We therefore obtain the second condition of Theorem~\ref{lapinv} from the explicit expression of $\phi^{v_0}$ given in \eqref{laplace explicite}. Then, by Theorem~\ref{lapinv}, the following asymptotic holds:
\begin{equation}\label{as:x0=0}
    g^{(0,b_0)}(r,0) \underset{r \to +\infty}{\sim} \frac{-2\sqrt{\pi}}{\Gamma(-1/2)}C_0f_0(v_0)\frac{e^{-r\M}}{r^{3/2}}  = C_0f_0(v_0)\frac{e^{-r\M}}{r^{3/2}}.
\end{equation}

Now suppose that $a_0 \neq 0$. Note that if $Z \sim \P_{(0,b_0)}$, then $Z + (a_0,0) \sim \P_{(a_0,b_0)}$. This follows from the pathwise uniqueness of the SDE \eqref{EDS}. Therefore, $g^{(a_0, b_0)}(x,0) = g^{(0, b_0)}(x - a_0,0)$, which yields \eqref{as:x0=0} using the identity $e^{a_0\M}f_0(0,b_0) = f_0(a_0,b_0)$ (see \eqref{f(v_0)col} and \eqref{f(v_0)branch}).
\end{proof}

\section{Proofs of Theorems~\ref{thm:1} to \ref{thm:4}: asymptotics along all directions}\label{sec:proofstheorems}
First, in Section~\ref{sub:lapinv}, we establish a representation of the Green's functions $g^{v_0}$ as simple integrals using an inverse Laplace theorem and some complex analysis. These integrals are of saddle-point type. Therefore, in Section~\ref{sub:pathsaddlepoint}, we define paths of steepest descent in a neighbourhood of the saddle points $x(\alpha)$ corresponding to different directions $\alpha$. 
We should now shift the integration contours up to the saddle point, taking into account the singularities of the integrand. This procedure depends on the position of the saddle points $x(\alpha)$ relative to the branching points $\m$ and $\M$ of the integrand. 
In Section~\ref{sub:saddle}, we prove Theorems~\ref{thm:1} and~\ref{thm:2}. In Section~\ref{sub:alphab}, we prove Theorem~\ref{thm3}. Finally, we prove Theorem~\ref{thm:4} in Section~\ref{sub:apresalphab}. In each section, we define a different path deformation for the integral representation of $g^{v_0}$. For each case, we identify the dominant contribution of the integral and determine which parts are negligible.

Without loss of generality, our proofs only consider angles $\alpha \in [0,\pi]$, since the remaining cases are symmetric.

\subsection{Laplace inverse and reduction to simple integrals}\label{sub:lapinv}

\nr{Using the two-dimensional Laplace inversion theorem, we can express $g^{v_0}(a,b)$ in terms of its Laplace transform (see \cite[Theorems 24.3 and 24.4]{doetsch_introduction_1974} and \cite{brychkov_multidimensional_1992}). } 

\nb{Using the Fourier-inverse theorem (or, equivalently, Laplace inversion theorem, see \cite[Theorems 24.3 and 24.4]{doetsch_introduction_1974} and \cite{brychkov_multidimensional_1992}), we can now express $g^{v_0}(a,b)$ in terms of a double integral. As in \cite[Lemma 4.1]{Franceschi_2024}, this expression can be written as a sum of simple integrals using complex analysis. Before stating the corresponding integral representations, note that the kernel $\gamma_+(x,y)$ given by \eqref{kernels} defines branches $X^+(y)$ and $X^-(y)$ such that $\gamma_+(X^\pm(y), y) = 0$ (as in Section~\ref{subsec:defprocess2} for the definition of $Y^\pm(x)$). These branches have branching points at $y_{min}$ and $y_{max}$  defined by
\begin{equation*}
y_{min} = \frac{\mu_1^+\Sigma^+_{12} - \mu^+_2\Sigma^+_{11} - \sqrt{D_2}}{\det(\Sigma^+)}, \quad y_{max} = \frac{\mu^+_1\Sigma^+_{12} - \mu^+_2\Sigma^+_{11} + \sqrt{D_2}}{\det(\Sigma^+)},
\end{equation*}
where $D_2 = (\mu^+_1\Sigma^+_{12} - \mu^+_2\Sigma^+_{11})^2 + {\mu^+_1}^2\det(\Sigma^+)$.}

\nr{\begin{lemma}[Inverse Laplace theorem\label{lem:lapinv}]
Let $v_0 \in \R^2$ and $(a,b) \in \R^2$ such that $b > 0$ and $(a,b) \neq v_0$. Then, for $\e_1 > 0, \e_2 > 0$ sufficiently small,
\begin{equation}\label{double}
g^{v_0}(a,b) = \frac{1}{(2i\pi)^2} \int_{-\e_1-i\infty}^{-\e_1+i\infty}\int_{-\e_2-i\infty}^{-\e_2+i\infty} \phi^{v_0}_+(x,y)e^{-ax-by}dydx
\end{equation}
where the convergence is understood in the sense of principal value.
\end{lemma}
\begin{proof}
This follows from the two-dimensional Laplace inversion theorem (see \cite[Theorems 24.3 and 24.4]{doetsch_introduction_1974} and \cite{brychkov_multidimensional_1992}), provided that the double integral converges in the sense of principal value. This convergence is ensured by the explicit expression \eqref{laplace explicite} and by standard computations based on integration by parts.
\end{proof}}

\nr{By similar considerations as in Section~\ref{subsec:defprocess2}, we define branches $X^\pm(y)$ by $\gamma_+(X^\pm(y), y) = 0$. These branches have branching points at $y_{min}$ and $y_{max}$, defined by
\begin{equation*}
y_{min} = \frac{\mu_1^+\Sigma^+_{12} - \mu^+_2\Sigma^+_{11} - \sqrt{D_2}}{\det(\Sigma^+)}, \quad y_{max} = \frac{\mu^+_1\Sigma^+_{12} + \mu^+_2\Sigma^+_{11} - \sqrt{D_2}}{\det(\Sigma^+)},
\end{equation*}
where $D_2 = (\mu^+_1\Sigma^+_{12} - \mu^+_2\Sigma^+_{11})^2 + {\mu^+_1}^2\det(\Sigma^+)$.}

\begin{lemma}[\nr{Reduction to simple integrals}\nb{Integral representation of Green's functions}]\label{double_simple}
Let $v_0 \in \R^2$ and $(a,b) \in \R^2$ such that $b > 0$ and $(a,b) \neq v_0$. Then, \nr{for $\e > 0$ small enough, }$g^{v_0}(a,b)$ can be written as $g^{v_0}(a,b) = I_1^{v_0}(a,b) + I_2^{v_0}(a,b)$, where
\begin{equation}\label{I1}
I_1^{v_0}(a,b) = \frac{1}{2i\pi}\int_{\nr{-\e}- i\infty}^{\nr{-\e} + i\infty}
\phi^{v_0}(x)\frac{\gamma(x, Y^+(x), Z^+(x))}{\partial_y\gamma_+(x, Y^+(x))} e^{-ax - bY^+(x)}
\, dx
\end{equation}

and
\begin{equation}\label{I2}
I_2^{v_0}(a,b) = \frac{1}{2i\pi}\int_{\nr{-\e} - i\infty}^{\nr{-\e} + i\infty}
\frac{e^{(a_0 - a)x} \left(
e^{(b_0 - b)Y^+(x)}\fc_{b_0\geq 0} 
+ e^{b_0Z^+(x) - bY^+(x)}\fc_{b_0 < 0}
\right)}{\partial_y\gamma_+(x, Y^+(x))} \, dx \quad \textnormal{if } b > b_0.
\end{equation}


Furthermore, \nb{ for $\epsilon >0$,} the following expressions hold for $I_2^{v_0}(a,b)$ if $b_0 > 0$:
\begin{equation}\label{I2''}
I_2^{v_0}(a,b) = \frac{1}{2i\pi}\int_{-\e - i\infty}^{-\e + i\infty}
\frac{e^{(a_0 - a)X^+(y) + (b_0 - b)y}}{\partial_x\gamma_+(X^+(y), y)} \, dy \quad \textnormal{if } a > a_0
\end{equation}
\begin{equation}\label{I2'''}
I_2^{v_0}(a,b) = \frac{-1}{2i\pi}\int_{-\e - i\infty}^{-\e + i\infty}
\frac{e^{(a_0 - a)X^-(y) + (b_0 - b)y}}{\partial_x\gamma_+(X^-(y), y)} \, dy \quad \textnormal{if } a < a_0.
\end{equation}
\end{lemma}

\nb{The representations \eqref{I2''} and \eqref{I2'''} will be used
to study the asymptotics near the horizontal axis, when
$b/a\to0$ (see Theorem~\ref{thm:2}).
Integrating with respect to $y$ rather than $x$ allows us
to apply the saddle-point method in these limiting directions.}

\begin{proof}
We first invert the Laplace transform in the vertical variable,
using exponential damping to apply Fourier--Plancherel, and evaluate
the resulting integrals by residues. We then invert in the horizontal
variable; identifying the second contribution with the Green function
of a homogeneous Brownian motion also gives the alternative
representations.

Fix $x\in i\mathbb R$ and define
\[
F_x^{v_0}(b)
=
\int_{\mathbb R} e^{ax}g^{v_0}(a,b)\,da,
\qquad b>0.
\]
Since $|e^{ax}|=1$, we have
$
|F_x^{v_0}(b)|
\leq
\int_{\mathbb R}g^{v_0}(a,b)\,da
=
g_B^{b_0}(b),
$
where $g_B^{b_0}$ denotes the Green density of the
one-dimensional process $B$ started at $b_0$.
The explicit formulas in Appendix~\ref{annexe:dim1} show that this density
is bounded on $(0,\infty)$.
Consequently, for every $\varepsilon>0$, the function
\[
b\longmapsto
e^{-\varepsilon b}F_x^{v_0}(b)\mathbf 1_{\{b>0\}}
\]
belongs to $L^1(\mathbb R)\cap L^2(\mathbb R)$.

By Fubini's theorem,
$
\phi_+^{v_0}(x,y)
=
\int_0^\infty e^{by}F_x^{v_0}(b)\,db,$
for $\Re y<0$.
In particular,
\[
\phi_+^{v_0}(x,-\varepsilon+is)
=
\int_0^\infty
e^{isb}e^{-\varepsilon b}F_x^{v_0}(b)\,db.
\]
Thus $s\mapsto\phi_+^{v_0}(x,-\varepsilon+is)$ is the
Fourier transform of the preceding $L^1\cap L^2$ function.

Recall that, for $x\in i\mathbb R$,
$
\Re Y^+(x)\geq Y^+(0)=0$,
$
\Re Y^-(x)\leq Y^-(0)
=
-\frac{2\mu_2^+}{\Sigma_{22}^+}.
$
We henceforth choose
$
0<\varepsilon<\frac{2\mu_2^+}{\Sigma_{22}^+}.
$
Then $\gamma_+(x,y)\neq0$ on the line
$\Re y=-\varepsilon$.
By Theorem~\ref{prop:laplace explicitee},
\begin{equation}
\label{lem42-transform}
\phi_+^{v_0}(x,y)
=
-\phi^{v_0}(x)
\frac{\gamma(x,y,Z^+(x))}{\gamma_+(x,y)}
-
\frac{
e^{a_0x+b_0Z^+(x)\mathbf 1_{\{b_0<0\}}
      +b_0y\mathbf 1_{\{b_0\geq0\}}}
}{
\gamma_+(x,y)
}.
\end{equation}

The Fourier--Plancherel theorem gives
\[
e^{-\varepsilon b}F_x^{v_0}(b)\mathbf 1_{\{b>0\}}
=
\lim_{R\to\infty}
\frac{1}{2\pi}
\int_{-R}^{R}
e^{-isb}\phi_+^{v_0}(x,-\varepsilon+is)\,ds
\]
in $L^2(\mathbb R,db)$.
After multiplication by $e^{\varepsilon b}$ and restriction
to $b>0$, this yields
\begin{equation}
\label{lem42-inversion}
F_x^{v_0}(b)
=
\lim_{R\to\infty}
\frac{1}{2\pi i}
\int_{-\varepsilon-iR}^{-\varepsilon+iR}
e^{-by}\phi_+^{v_0}(x,y)\,dy
\end{equation}
in $L^2_{\mathrm{loc}}((0,\infty),db)$, and almost everywhere
along a subsequence.
We now show that the two terms obtained by substituting
\eqref{lem42-transform} into \eqref{lem42-inversion}
have limits as $R\to\infty$.

For the first term, consider
\[
\frac{-\gamma(x,y,Z^+(x))}{\gamma_+(x,y)}e^{-by}.
\]
Close the vertical segment
$[-\varepsilon-iR,-\varepsilon+iR]$, traversed upwards,
by the right semicircle
\[
\Gamma_R^+
=
\left\{
-\varepsilon+Re^{i\theta}:
-\frac{\pi}{2}\leq\theta\leq\frac{\pi}{2}
\right\},
\]
traversed from its upper endpoint to its lower endpoint. See Figure~\ref{doublesimple1}.
The resulting contour is clockwise.
For sufficiently large $R$, its only possible pole is
$Y^+(x)$.

For fixed $x$, uniformly on this semicircle as $R\to\infty$,
$
\frac{\gamma(x,y,Z^+(x))}{\gamma_+(x,y)}
=
O(R^{-1}).
$
Moreover,
$
|e^{-by}|
=
e^{\varepsilon b-bR\cos\theta}
\leq e^{\varepsilon b}.
$
After parametrisation of the semicircle, the factor $R$
from $dy$ cancels the preceding $O(R^{-1})$ bound.
For $b>0$, the resulting integrand tends to zero for
every $\theta\in(-\pi/2,\pi/2)$ and is bounded uniformly
for sufficiently large $R$.
Dominated convergence therefore shows that the integral
over $\Gamma_R^+$ tends to zero.
The residue theorem gives
\begin{equation}
\label{lem42-first-residue}
\frac{\phi^{v_0}(x)}{2\pi i}
\lim_{R\to\infty}
\int_{-\varepsilon-iR}^{-\varepsilon+iR}
\frac{-\gamma(x,y,Z^+(x))}{\gamma_+(x,y)}
e^{-by}\,dy
=
\phi^{v_0}(x)
\frac{\gamma(x,Y^+(x),Z^+(x))}
     {\partial_y\gamma_+(x,Y^+(x))}
e^{-bY^+(x)}.
\end{equation}

For the second term, the vertical integral is absolutely
convergent, since
$
\frac{1}{\gamma_+(x,-\varepsilon+is)}
=
O(|s|^{-2})
\qquad (|s|\to\infty).
$
If $b>b_0$ (which always holds when $b_0\leq0$), we close
the contour to the right as above.
If $0<b<b_0$, we instead close it to the left, obtaining
a counterclockwise contour whose only possible pole is
$Y^-(x)$.
In both cases the exponential factor decays on the
corresponding semicircle.
When $b=b_0>0$, this exponential factor is constant,
but the semicircle integral still tends to zero,
because its integrand is $O(R^{-2})$ and its length
is $O(R)$.

It follows that the second contribution is
\begin{equation}
\label{lem42-second-residue}
K_x^{v_0}(b)
=
e^{a_0x+b_0Z^+(x)\mathbf 1_{\{b_0<0\}}}
\begin{cases}
\displaystyle
\frac{
e^{(b_0\mathbf 1_{\{b_0\geq0\}}-b)Y^+(x)}
}{
\partial_y\gamma_+(x,Y^+(x))
},
& b\geq b_0,\quad b>0,\\[3mm]
\displaystyle
-\frac{e^{(b_0-b)Y^-(x)}}
{\partial_y\gamma_+(x,Y^-(x))},
& 0<b<b_0.
\end{cases}
\end{equation}
The two expressions agree at $b=b_0>0$, since
$
\partial_y\gamma_+(x,Y^-(x))
=
-\partial_y\gamma_+(x,Y^+(x)).
$
Combining \eqref{lem42-inversion},
\eqref{lem42-first-residue} and
\eqref{lem42-second-residue}, we obtain, for each fixed
$x\in i\mathbb R$ and almost every $b>0$,
\begin{equation}
\label{lem42-partial-transform}
F_x^{v_0}(b)
=
\phi^{v_0}(x)
\frac{\gamma(x,Y^+(x),Z^+(x))}
     {\partial_y\gamma_+(x,Y^+(x))}
e^{-bY^+(x)}
+
K_x^{v_0}(b).
\end{equation}

We next invert with respect to $x$.
The explicit formulas for the branches give
$
\Re Y^\pm(it)\sim \pm C^+|t|,
$
$\Re Z^+(it)\sim C^-|t|$,
when $|t|\to\infty$,
where
$
C^\pm
=
\frac{\sqrt{\det\Sigma^\pm}}{\Sigma_{22}^\pm}>0.
$
Also,
$
|\phi^{v_0}(it)|
\leq \mathbb E_{v_0}[L_\infty^0(B)]<\infty.
$
Consequently, the first term on the right-hand side of
\eqref{lem42-partial-transform} belongs to $L^1(\mathbb R,dt)$
for every $b>0$.
Its inverse Fourier transform is the function $I_1^{v_0}$
in \eqref{I1}; the same estimates show that
$I_1^{v_0}(a,b)$ is continuous for $b>0$.

The function $t\mapsto K_{it}^{v_0}(b)$ also belongs to
$L^1(\mathbb R,dt)$ whenever $b\neq b_0$.
Indeed, if $b_0\geq0$ and $b\neq b_0$, its exponential
factor decays at rate $C^+|b-b_0||t|$.
If $b_0<0$, it decays at rate
$(C^+b+C^-|b_0|)|t|$.
When $b=b_0>0$, this exponential decay is absent and
the term is only $O(|t|^{-1})$; we treat this case
by continuity below.

For every $b>0$, the function $a\mapsto g^{v_0}(a,b)$
belongs to $L^1(\mathbb R)$.
By Fubini's theorem,
\eqref{lem42-partial-transform}, with $x=it$, holds for
almost every $(t,b)$.
Fourier inversion therefore yields
\[
g^{v_0}(a,b)=I_1^{v_0}(a,b)+I_2^{v_0}(a,b)
\]
for almost every $(a,b)$ with $b>0$, where, initially
for $b\neq b_0$,
\[
I_2^{v_0}(a,b)
=
\frac{1}{2\pi}
\int_{\mathbb R}e^{-iat}K_{it}^{v_0}(b)\,dt.
\]
In particular, when $b>b_0$, this gives \eqref{I2}.

If $b_0\leq0$, the preceding exponential bounds are
locally uniform for $b>0$, so $I_2^{v_0}$ is continuous
there. The continuity of $g^{v_0}$ then extends the
identity to every $(a,b)$ with $b>0$.

Suppose now that $b_0>0$.
Let $\widetilde g^{v_0}$ be the Green density of Brownian
motion on $\mathbb R^2$ with covariance matrix $\Sigma^+$
and drift $\mu^+$, started at $v_0$.
Its bilateral Laplace transform satisfies
\begin{equation}
\label{lem42-free-transform}
\widetilde\phi^{v_0}(x,y)
=
-\frac{e^{a_0x+b_0y}}{\gamma_+(x,y)}
\end{equation}
whenever
$x\in i\mathbb R$ and
$-2\mu_2^+/\Sigma_{22}^+<\Re y<0$.
Indeed, in this domain $\Re\gamma_+(x,y)<0$, and
\eqref{lem42-free-transform} follows by integrating
\[
\mathbb E_{v_0}[e^{x\widetilde A_t+y\widetilde B_t}]
=
e^{a_0x+b_0y+t\gamma_+(x,y)}
\]
over $t\geq0$.

Applying the same damped Fourier inversion and residue
calculation to \eqref{lem42-free-transform} gives exactly
\eqref{lem42-second-residue} with $b_0>0$.
Thus
\[
I_2^{v_0}(a,b)=\widetilde g^{v_0}(a,b)
\qquad\text{for }b>0,\ b\neq b_0.
\]
The Gaussian integral representation of
$\widetilde g^{v_0}$ shows that it is continuous away
from $v_0$. We use this function to extend $I_2^{v_0}$
to the points $(a,b_0)$ with $a\neq a_0$.
Since both $I_1^{v_0}+\widetilde g^{v_0}$ and $g^{v_0}$
are continuous away from $v_0$, their almost-everywhere
equality implies
\[
g^{v_0}(a,b)
=
I_1^{v_0}(a,b)+\widetilde g^{v_0}(a,b)
\]
for every $b>0$ and $(a,b)\neq v_0$.

It remains to prove \eqref{I2''} and \eqref{I2'''}.
We continue to assume $b_0>0$ and keep the same
$\varepsilon>0$.
For $x=it$ and $y=-\varepsilon+is$, we have
\[
\Re\gamma_+(it,-\varepsilon+is)
=
\gamma_+(0,-\varepsilon)
-\frac12(t,s)\Sigma^+(t,s)^{\mathsf T}
<0.
\]
In particular, for every $y\in-\varepsilon+i\mathbb R$,
the two roots $X^\pm(y)$ satisfy
\[
\Re X^-(y)<0<\Re X^+(y).
\]
To see this, note that the inequalities hold at
$y=-\varepsilon$, and neither root can cross the
imaginary axis by the preceding strict inequality.

The weighted Green density
$e^{-\varepsilon b}\widetilde g^{v_0}(a,b)$ is integrable
on $\mathbb R^2$.
We may therefore apply the preceding Fourier inversion
argument in the reverse order, first with respect to $x$.
For fixed $y\in-\varepsilon+i\mathbb R$, the inner integral is
\[
-\frac{1}{2\pi i}
\int_{-i\infty}^{i\infty}
\frac{e^{(a_0-a)x}}{\gamma_+(x,y)}\,dx.
\]
It is absolutely convergent.
If $a>a_0$, closing the contour to the right gives
$
\frac{e^{(a_0-a)X^+(y)}}
{\partial_x\gamma_+(X^+(y),y)}.
$
If $a<a_0$, closing it to the left gives
$
-\frac{e^{(a_0-a)X^-(y)}}
{\partial_x\gamma_+(X^-(y),y)}.
$
Inverting with respect to $y$ consequently yields
\eqref{I2''} and \eqref{I2'''}.
Both integrals converge absolutely, since
$
\Re X^\pm(-\varepsilon+is)
=
\pm\frac{\sqrt{\det\Sigma^+}}{\Sigma_{11}^+}|s|
+O(1)$,
when $|s|\to\infty.
$
The exponential bounds are locally uniform on
$a>a_0$ and $a<a_0$, respectively, so the resulting
functions are continuous there.
This also justifies passing from the almost-everywhere
Fourier inversion identities to the pointwise formulas.
\end{proof}

\begin{proof}
\end{proof}

\begin{figure}[t]
\centering
\includegraphics[scale=1.5]{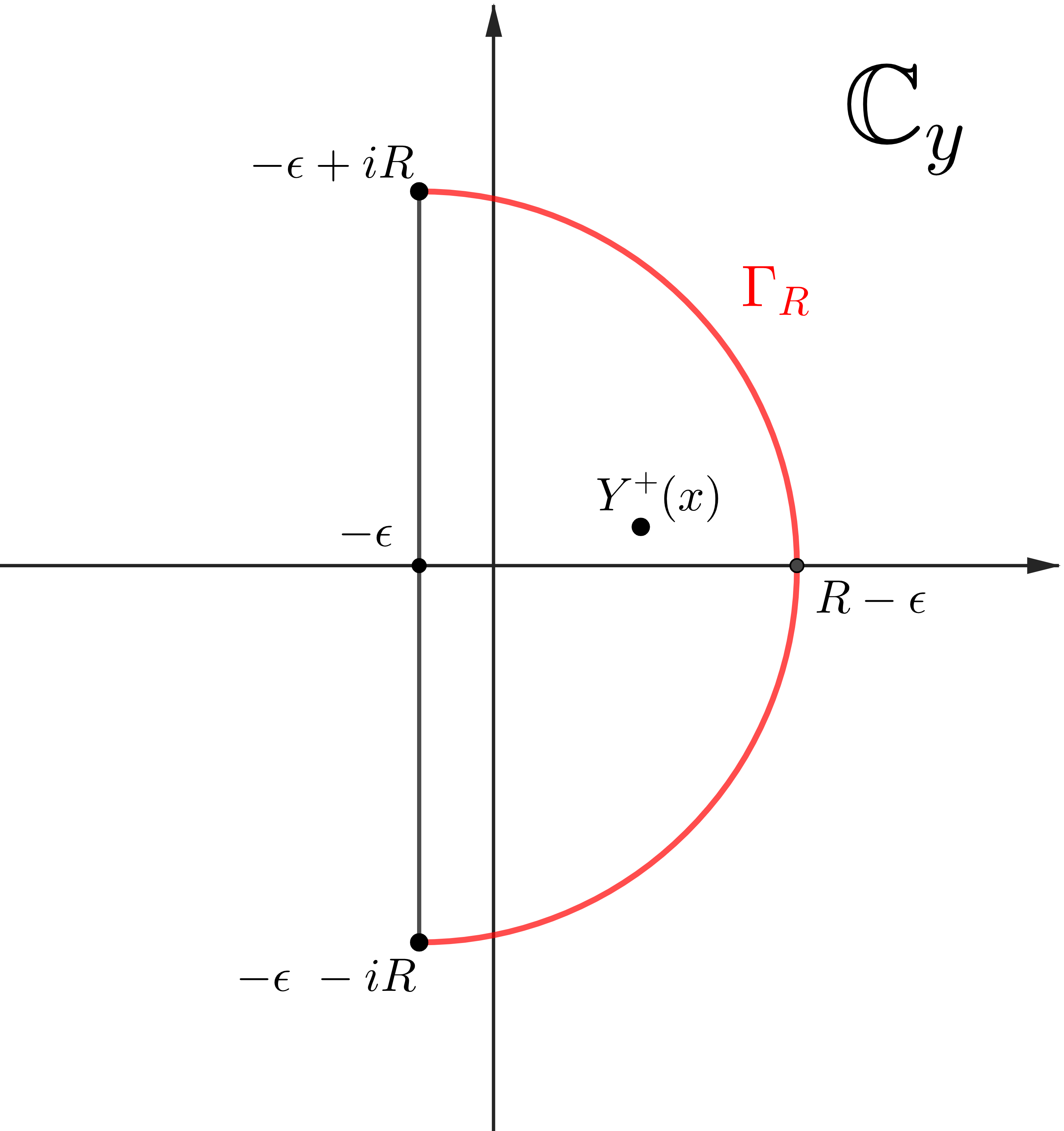}
\caption{Changing path for representations \eqref{I1} and \eqref{I2} of $I_1^{v_0}(a,b)$ and $I_2^{v_0}(a,b)$.}
\label{doublesimple1}
\end{figure}

\subsection{Paths of the saddle point method}\label{sub:pathsaddlepoint}
As the previous integrals are of the saddle-point type, we now introduce the associated paths. We recall the definition of $(x(\alpha), y(\alpha))$ for $\alpha \in [0, \pi]$ given by~\eqref{xalphayalpha}.
Let us define
\begin{equation}\label{fff}
F(x, \alpha) = -\cos(\alpha)\, x - \sin(\alpha)\, Y^+(x) + \cos(\alpha)\, x(\alpha) + \sin(\alpha)\, y(\alpha).
\end{equation}

By construction, for all $\alpha \in (0, \pi)$, we have $F(x(\alpha), \alpha) = 0$ and $F'_x(x(\alpha), \alpha) = 0$. It follows that $(Y^+)'(x(\alpha)) = -\frac{\cos(\alpha)}{\sin(\alpha)}$, and
\[
x'(\alpha)\, F''_x(x(\alpha), \alpha) + F''_{x,\alpha}(x(\alpha), \alpha) = 0.
\]

Then,
\begin{equation}\label{eq_link_cste}
x'(\alpha)\, F''_x(x(\alpha), \alpha) = -\left( \sin(\alpha) - (Y^+)'(x(\alpha)) \cos(\alpha) \right) = \frac{-1}{\sin(\alpha)}.
\end{equation}

From computations carried out in~\cite[Section~5]{Franceschi_2024}, we also have
\begin{equation}
\label{fzfz}
F''_x(x(\alpha), \alpha) = \frac{ \Sigma^+_{11}\sin^2(\alpha) - 2\Sigma^+_{12} \sin(\alpha)\cos(\alpha) + \Sigma^+_{22} \cos^2(\alpha) }{ \partial_y \gamma_+(x(\alpha), y(\alpha)) \sin(\alpha) } > 0, \quad \alpha \in (0, \pi).
\end{equation}

\subsubsection*{Case $\alpha_0 \in (0, \pi)$.}
Let us fix $\alpha_0 \in (0, \pi)$. By the parameterised Morse Lemma~\cite[Lemma~A.1]{Franceschi_2024}, there exists a neighbourhood of $(0, \alpha_0) \in \C \times \R$, given by
\begin{equation}\label{OMEGA}
\Omega(0, \alpha_0) = \left\{ (\omega, \alpha) \in \C \times [0, \pi] \; \middle| \; |\omega| \leq L, \; |\alpha - \alpha_0| \leq \eta \right\}
\end{equation}
for some $L, \eta > 0$, and a function
$x : \Omega(0, \alpha_0) \longrightarrow \C$
of class $\mathcal{C}^\infty$ (and holomorphic in the first variable) such that
\[
x(0, \alpha) = x(\alpha) \quad \text{for all } \alpha \text{ with } |\alpha - \alpha_0| \leq \eta,
\]
and satisfying
\begin{equation}\label{eq:F=t2}
F(x(\omega, \alpha), \alpha) = \omega^2, \quad \forall (\omega, \alpha) \in \Omega(0, \alpha_0).
\end{equation}

Furthermore,
\begin{equation}\label{zop}
x'_\omega(0, \alpha) = \sqrt{ \frac{2}{F''_x(x(\alpha), \alpha)} }.
\end{equation}

Let $0 < \epsilon < K$. We define the contour of steepest descent as
\[
\Gamma_{x,\alpha} = \left\{ x(it, \alpha) \;\middle|\; t \in [-\epsilon, \epsilon] \right\}.
\]
Note that $F(x(it, \alpha), \alpha) = -t^2$. We denote by $x^+_\alpha = x(i\epsilon, \alpha)$ and $x^-_\alpha = x(-i\epsilon, \alpha)$ the endpoints of $\Gamma_{x,\alpha}$.
It can be shown (see~\cite{Franceschi_2024}) that for any $\alpha_0 \in (0, \pi)$, there exist $\eta > 0$ and $\nu > 0$ such that
\begin{equation}\label{nu}
\operatorname{Im}(x^+_\alpha) > \nu, \quad \operatorname{Im}(x^-_\alpha) < -\nu \quad \text{for all } \alpha \text{ with } |\alpha - \alpha_0| < \eta.
\end{equation}

\subsubsection*{Case $\alpha_0 \in \{0, \pi\}$.}

We define
\begin{equation}
    G(y, \alpha) = -\cos(\alpha) X^+(y) - \sin(\alpha) y 
    + \cos(\alpha) x(\alpha) + \sin(\alpha) y(\alpha).
\end{equation}
As before, we have $G(y(\alpha), \alpha) = 0$, $G'_y(y(\alpha), \alpha) = 0$, and $G''_y(y(\alpha), \alpha) > 0$ for $\alpha \in [0, \pi/2) \cup (\pi/2, \pi]$.
Then, if $\alpha_0 \in [0, \pi/2) \cup (\pi/2, \pi]$ (and in particular if $\alpha_0 = 0$ or $\alpha_0 = \pi$), we define symmetrically the function $y(\omega, \alpha)$ corresponding to $G(y, \alpha)$ via the parameterised Morse Lemma.
We then set
\[
\Gamma_{y,\alpha} = \left\{ y(it, \alpha) \; \middle| \; t \in [-\epsilon, \epsilon] \right\},
\]
whose endpoints are $y^+_\alpha = y(i\epsilon, \alpha)$ and $y^-_\alpha = y(-i\epsilon, \alpha)$.
The paths $\Gamma_{x,\alpha}$ and $\Gamma_{y,\alpha}$ are linked by the relations
\begin{equation}
\Gamma_{x,\alpha} = \underrightarrow{\overleftarrow{X^+(\Gamma_{y,\alpha})}}, \quad
\Gamma_{y,\alpha} = \underrightarrow{\overleftarrow{Y^+(\Gamma_{x,\alpha})}}, \quad \alpha \in (0, \pi/2) \cup (\pi/2, \pi).
\label{eq:gammaXarrow}
\end{equation}
where the arrows indicate that the paths are traversed in the reversed direction.


\subsection{Lemmas for red-arc parts}\label{sub:saddle}

Now, let us consider the asymptotics of Green's functions $g^{v_0}$ in the directions corresponding to the red arc parts, excluding $\alpha_b$ and $\tilde\alpha_b$ (see Figure~\ref{anglesalphab}).
\begin{figure}[t]
    \centering
    \includegraphics[scale=0.4]{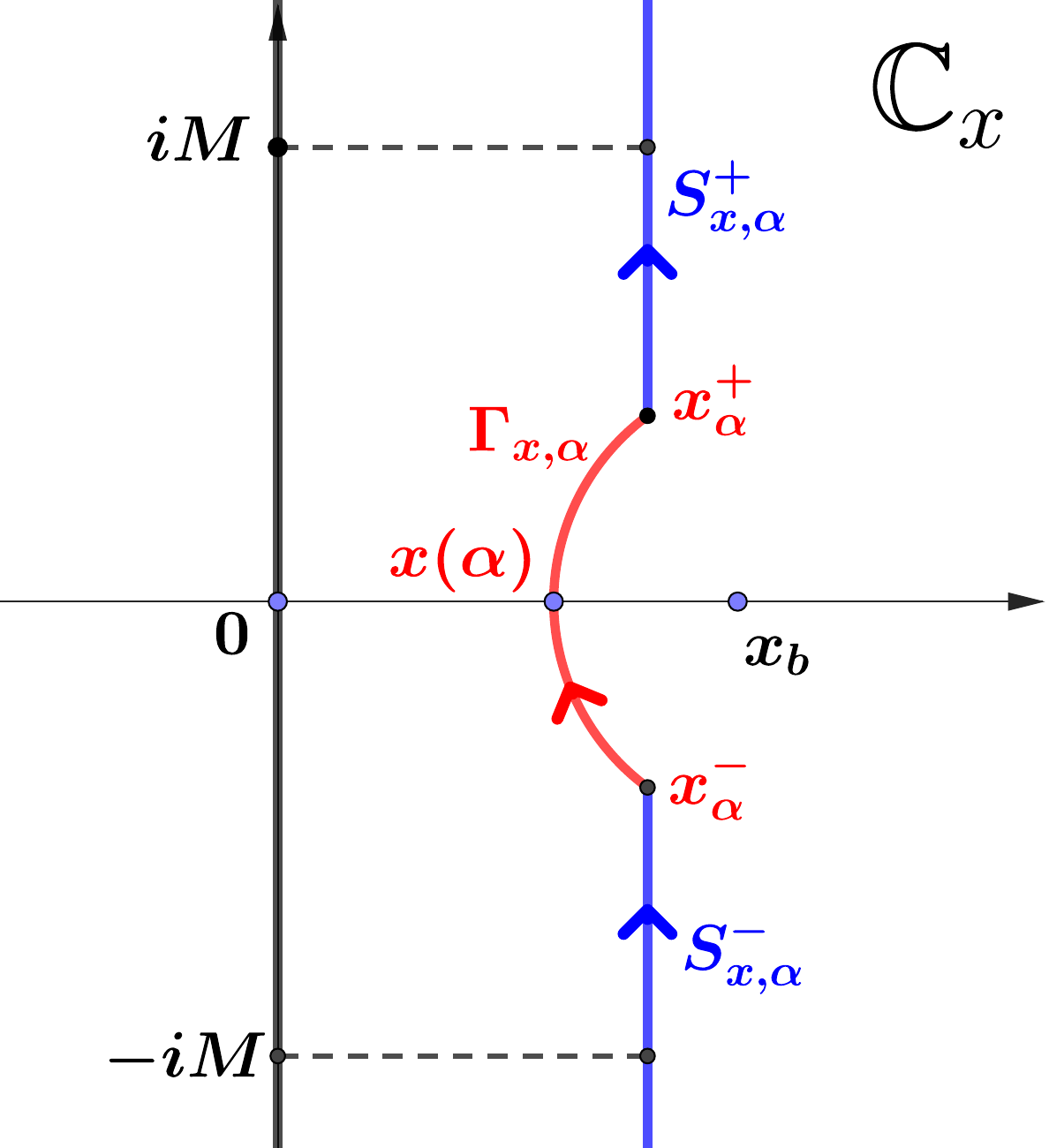}
    \caption{Changing path of Lemma~\ref{lem:chemin}}
    \label{fig:chgtchemincol}
\end{figure}
Recall the notation introduced in \eqref{def:xb}. We define
\[
S^+_{x, \alpha} = \{ x^+_{\alpha} + it \mid t \geq 0 \}, \quad
S^-_{x, \alpha} = \{ x^-_{\alpha} - it \mid t \geq 0 \},
\]
\[
S^+_{y, \alpha} = \{ y^+_{\alpha} + it \mid t \geq 0 \}, \quad
S^-_{y, \alpha} = \{ y^-_{\alpha} - it \mid t \geq 0 \}.
\]

The following lemma describes how to modify the contours of integration in \eqref{I1} and \eqref{I2} so that they pass through the saddle point along the steepest descent path.

\begin{lemma}[Changing path through the saddle point]\label{lem:chemin}
Let $v_0 \in \R^2$, $a \in \R$, $b > 0$ such that $(a,b) \neq v_0$, and let
$\alpha \in (0,\pi)$ be such that ${\m} < x(\alpha) < {\M}$. Then:

\begin{equation}\label{I11}
I_1^{v_0}(a,b) = \frac{1}{2i\pi} \int_{S^-_{x, \alpha} \cup \Gamma_{x,\alpha} \cup S^+_{x, \alpha}} \phi^{v_0}(x) \frac{\gamma(x,Y^+(x),Z^+(x))}{\partial_y\gamma_+(x,Y^+(x))} e^{-a x - b Y^+(x)} \, dx,
\end{equation}
and
\begin{equation}\label{I22}
I_2^{v_0}(a,b) = \frac{1}{2i\pi} \int_{S^-_{x, \alpha} \cup \Gamma_{x,\alpha} \cup S^+_{x, \alpha}} \frac{e^{(a_0 - a)x} \left(
e^{(b_0 - b) Y^+(x)} \fc_{b_0 > 0} + e^{b_0 Z^+(x) - b Y^+(x)} \fc_{b_0 < 0}
\right)}{\partial_y\gamma_+(x,Y^+(x))} \, dx, \quad \text{if } b > b_0.
\end{equation}

Furthermore, if $b_0 > 0$, then
\begin{equation}\label{I22''}
I_2^{v_0}(a,b) = \frac{1}{2i\pi} \int_{S^-_{y, \alpha} \cup \Gamma_{y,\alpha} \cup S^+_{y, \alpha}} \frac{e^{(a_0 - a) X^+(y) + (b_0 - b)y}}{\partial_x\gamma_+(X^+(y),y)} \, dy, \quad \text{if } a > a_0,
\end{equation}
and
\begin{equation}\label{I22'''}
I_2^{v_0}(a,b) = -\frac{1}{2i\pi} \int_{S^-_{y, \alpha} \cup \Gamma_{y,\alpha} \cup S^+_{y, \alpha}} \frac{e^{(a_0 - a) X^-(y) + (b_0 - b)y}}{\partial_x\gamma_+(X^-(y),y)} \, dy, \quad \text{if } a < a_0.
\end{equation}
\end{lemma}

\begin{proof}
We prove \eqref{I11}; the other identities are analogous. The argument is similar to the one given in \cite[Lemma 6.2]{Franceschi_2024}. Using the Residue Theorem (see Figure~\ref{fig:chgtchemincol}) and Lemma~\ref{double_simple}, it suffices to verify that
\begin{equation}\label{bout_neg}
\sup_{u \in [X^+(y_{\max}) - \eta, x^{\max} + \eta]} \left|
\frac{ \phi^{v_0}(u + is) \, \gamma(u + is, Y^+(u + is), Z^+(u + is)) }
     { \partial_y\gamma_+(u + is, Y^+(u + is)) }
\, e^{-a(u + is) - b Y^+(u + is)}
\right| \longrightarrow 0
\end{equation}
as $|s| \to \infty$, for some $\eta > 0$. This decay follows directly from \eqref{laplace explicitee} and the estimate on the real part in \eqref{Re}.
\end{proof}

The following two lemmas provide the asymptotic contribution of integrals along the steepest descent contours.

\begin{lemma}[Asymptotics along $\Gamma_{x,\alpha}$ in the red arc region with $\alpha_0 \neq 0, \pi$]
\label{lem:contribsaddle}
Let $\alpha_0 \in (0,\pi)$ such that ${\m} < x(\alpha_0) < {\M}$. Let $r = \sqrt{a^2 + b^2}$ and define $\alpha = \alpha(a,b) \in (0,\pi)$ such that $\cos(\alpha) = \frac{a}{\sqrt{a^2 + b^2}}, \qquad 
\sin(\alpha) = \frac{b}{\sqrt{a^2 + b^2}}.$
Then, the following asymptotic holds as $r \to +\infty$ and $\alpha \to \alpha_0$:
\begin{equation}\label{as:coll}
\frac{1}{2i\pi} \int_{\Gamma_{x,\alpha}} 
\frac{
\phi^{v_0}(x) \gamma(x, Y^+(x), Z^+(x)) e^{-a x - b Y^+(x)} 
+ e^{(a_0 - a)x} \left(
e^{(b_0 - b) Y^+(x)} \fc_{b_0 > 0}
+ e^{b_0 Z^+(x) - b Y^+(x)} \fc_{b_0 < 0}
\right)
}{
\partial_y \gamma_+(x, Y^+(x))
}
 dx 
\end{equation}
\[
\underset{r \to +\infty \atop \alpha \to \alpha_0}{\sim}
C^+(\alpha_0) h_{\alpha_0}(v_0)
\frac{e^{-r(\cos(\alpha)x(\alpha) + \sin(\alpha)y(\alpha))}}{\sqrt{r}},
\]
where $C^+(\alpha_0)$ and $h_{\alpha_0}(v_0)$ are defined in \eqref{C(alpha)} and \eqref{Harm1}.
\end{lemma}

\begin{proof}
We have $F''_{x}(x(\alpha_0), \alpha_0) > 0$ (see \eqref{fzfz}). Then, by the parameterised saddle point method (see \cite[Lemma 8.1]{Franceschi_2024}), for all $n \geq 0$, the integral \eqref{as:coll} admits the following asymptotic expansion as $\alpha = \alpha(a,b) \to \alpha_0$ and $r = \sqrt{a^2 + b^2} \to +\infty$:
\begin{equation}
\label{contrib_col} 
e^{-r(\cos(\alpha) x(\alpha) + \sin(\alpha) y(\alpha))} \sum_{k=0}^n \frac{c_k(\alpha)}{r^k \sqrt{r}} + o\left(\frac{1}{r^n \sqrt{r}}\right),
\end{equation}
where the coefficients $c_0(\alpha), \ldots, c_n(\alpha)$ are continuous functions of $\alpha$ and satisfy
\[
c_0(\alpha) = \frac{h_{\alpha}(v_0) \, x'_\omega(0, \alpha)}{2 \sqrt{\pi} \, \partial_y \gamma_+(x(\alpha), y(\alpha))} = C(\alpha) h_{\alpha}(v_0)
\]
(see \eqref{zop}). In particular, $c_0(\alpha_0) > 0$ by expression \eqref{Harm1}. Therefore, the dominant asymptotic term is given by the $k=0$ term in \eqref{contrib_col}, which concludes the proof.
\end{proof}

\begin{lemma}[Asymptotics along $\Gamma_{x,\alpha}$ and $\Gamma_{y,\alpha}$ in the red arc region at $\alpha_0 = 0, \pi$] \label{lem:colalpha=0}
Suppose $x^+_{max} < x^-_{max}$ and let $\alpha_0 = 0$. If $r = \sqrt{a^2 + b^2}$ and $\alpha = \alpha(a,b) \in (0,\pi)$ is the angle such that $\cos(\alpha) = \frac{a}{\sqrt{a^2 + b^2}}$ and $\sin(\alpha) = \frac{b}{\sqrt{a^2 + b^2}}$, then the following asymptotics hold:
\begin{equation}\label{as:collalpha=0}
\frac{1}{2i\pi}\int_{\Gamma_{x,\alpha}} \phi^{v_0}(x) \frac{\gamma(x, Y^+(x), Z^+(x))}{\partial_y \gamma_+(x, Y^+(x))} e^{-a x - b Y^+(x)} \, dx 
+ \frac{1}{2i\pi} \int_{\Gamma_{y,\alpha}} \frac{e^{(a_0 - a) X^+(y) + (b_0 - b) y}}{\partial_x \gamma_+(X^+(y), y)} \, dy
\end{equation}
\[
\underset{r \to +\infty, \ \alpha \to 0, \ \alpha \geq 0}{\sim} 
C_0 f_0(v_0) \left( \kappa \alpha + \frac{1}{r} \right) \frac{e^{-r(\cos(\alpha) x(\alpha) + \sin(\alpha) y(\alpha))}}{\sqrt{r}},
\]
where $\kappa > 0$ is an explicit constant (see \eqref{eq:kappa}) and $f_0(v_0)$ is given by \eqref{f(v_0)col}. Furthermore, 
\[
\partial_\alpha \bigl(C(\alpha) h_\alpha(v_0)\bigr) \big|_{\alpha = 0} = \kappa C_0 f_0(v_0).
\]

If $x^+_{min} > x^-_{min}$, the symmetric result holds for $\alpha \to \pi$ with $\alpha \leq \pi$.
\end{lemma}

\begin{proof}
To prove this, we consider the integral representation along the contour $\Gamma_{y,\alpha}$. By performing the change of variables $x = X^+(y)$, we obtain:
\begin{equation}\label{chgtvar}
\int\limits_{\Gamma_{x, \alpha}}\phi^{v_0}(x)\frac{\gamma(x,Y^+(x),Z^+(x))}{\partial_y\gamma_+(x,Y^+(x))}e^{-ax-bY^+(x)}\,dx = \int\limits_{\Gamma_{y, \alpha}}\phi^{v_0}(X^+(y))\frac{\gamma(X^+(y),y,Z^+(X^+(y)))}{\partial_x\gamma_+(X^+(y),y)}e^{-aX^+(y)-by}\,dy.
\end{equation}
Recall that $(x(\alpha),y(\alpha)) \underset{\alpha \to 0}{\longrightarrow} (x^+_{{max}}, Y^\pm(x^+_{{max}}))$. Using \eqref{laplace explicitee}, we have:
\[
\phi^{v_0}(X^+(y)) = -\frac{e^{X^+(y)a_0 + Y^-(X^+(y))b_0\fc_{b_0\geq 0} +  Z^+(X^+(y))b_0\fc_{b_0<0}}}{\gamma(X^+(y),Y^-(X^+(y)), Z^+(X^+(y)))}.
\]
By definition, $X^+(y)$ is holomorphic in a neighborhood of $Y^\pm(x^+_{{max}})$. Moreover, by the Viète identities, the expression
\[
Y^{-}(X^+(y)) = \frac{\Sigma^+_{11} (X^+(y))^2 + 2 \mu^+_1 X^+(y)}{ \Sigma^+_{22} y }
\]
is holomorphic in a neighborhood of $Y^\pm(x^+_{{max}})$. Since we are in the case $x^+_{{max}} < x^-_{{max}}$, the function $Z^+(X^+(y))$ is also holomorphic in a neighborhood of $Y^\pm(x^+_{{max}})$. Thus, the integrand on the right-hand side of \eqref{chgtvar} is holomorphic near the saddle point.

Applying the parameterized saddle-point method to the contour $\Gamma_{y,\alpha}$, we deduce that for any $\alpha_0 \in [0, \pi/2)$, the expression in \eqref{as:collalpha=0} admits the following asymptotic expansion as $\alpha = \alpha(a,b) \to \alpha_0$ and $r = \sqrt{a^2 + b^2} \to \infty$:
\begin{equation}\label{DLpetit}
\left(\tilde c_0(\alpha)+ \frac{\tilde c_1(\alpha)}{r} + o\left(\frac{1}{r}\right)\right)\frac{e^{-r(\cos(\alpha)x(\alpha) + \sin(\alpha)y(\alpha))}}{\sqrt r},
\end{equation}
where the coefficients $\tilde c_k(\alpha)$ are continuous in $\alpha$. By uniqueness of the asymptotic expansion, we necessarily have $c_k(\alpha) = \tilde c_k(\alpha)$ for all $\alpha \in (0, \pi/2]$ and for $k=0,1$. 
Recall that $X^+(y(\alpha)) = x(\alpha)$ and $Y^+(x(\alpha)) = y(\alpha)$ for all $\alpha \in (0,\pi/2)$. In particular, we have:
\begin{equation}\label{tttt}
c_0(\alpha) = C^+(\alpha) \cdot \frac{e^{x(\alpha)a_0 + b_0 y(\alpha)\fc_{b_0>0} + Z^+(x(\alpha))b_0\fc_{b_0<0}}}{\gamma(x(\alpha), Y^-(x(\alpha)), Z^+(x(\alpha)))} 
\end{equation}
\[
\times \left( -\gamma(x(\alpha), y(\alpha), Z^+(x(\alpha))) e^{(Y^-(x(\alpha)) - y(\alpha))b_0\fc_{b_0 \geq 0}} + \gamma(x(\alpha), Y^-(x(\alpha)), Z^+(x(\alpha))) \right).
\]

Furthermore, we observe that
\begin{align*}
y(\alpha) - Y^{-}(X^+(y(\alpha))) &= Y^+(X^+(y(\alpha))) - Y^-(X^+(y(\alpha))) \\
&= \frac{2}{\Sigma^+_{22}} \sqrt{ \det(\Sigma_+) \cdot (x^+_{{max}} - X^+(y(\alpha))) \cdot (X^+(y(\alpha)) - x^+_{{min}}) }.
\end{align*}
We also recall that $X^+(Y^\pm(x^+_{{max}})) = x^+_{{max}}$, and that
\[
(X^+)'(y)\Big|_{y = Y^\pm(x^+_{{max}})} = 0, \quad (X^+ )''(y)\Big|_{y = Y^\pm(x^+_{{max}})} = -\frac{\Sigma^+_{22}}{\partial_x \gamma_+(x^+_{{max}}, Y^\pm(x^+_{{max}}))}.
\]
It follows that
\begin{equation}\label{ttt}
y(\alpha) - Y^-(X^+(y(\alpha))) \sim \sqrt{ \frac{2 \det(\Sigma^+) (x^+_{{max}} - x^+_{{min}})}{ \Sigma^+_{22} \partial_x \gamma_+(x^+_{{max}}, Y^\pm(x^+_{{max}})) } } \cdot \alpha \sim \Pi \alpha,
\end{equation}
where we define
\[
\Pi = \sqrt{ \frac{2 \det(\Sigma^+) (x^+_{{max}} - x^+_{{min}})}{ \Sigma^+_{22} \partial_x \gamma_+(x^+_{{max}}, Y^\pm(x^+_{{max}})) } } = 2 \sqrt{ \frac{ \det(\Sigma^+) }{ \Sigma^+_{11} \Sigma^+_{22} } },
\]
since $\partial_x \gamma_+(x^+_{{max}}, Y^\pm(x^+_{{max}})) = \frac{\Sigma^+_{11}(x^+_{{max}} - x^+_{{min}})}{2}$.

Injecting \eqref{ttt} into \eqref{tttt}, and using the asymptotics
\[
C^+(\alpha) \underset{\alpha \to 0^+}{\longrightarrow} \frac{1}{\sqrt{2\pi \Sigma^+_{22} \partial_x \gamma_+(x^+_{{max}}, Y^\pm(x^+_{{max}}))}} = \frac{1}{\sqrt{\pi \Sigma^+_{11} \Sigma^+_{22}(x^+_{{max}} - x^+_{{min}})}},
\]
we obtain, after simplification,
\begin{equation}\label{eq:memefonctionharm}
C(\alpha) h_\alpha(v_0) = c_0(\alpha) \underset{\alpha \to 0^+}{\sim} \frac{2}{\Sigma^+_{11} \Sigma^+_{22}} \sqrt{ \frac{ \det(\Sigma^+) }{ \pi(x^+_{{max}} - x^+_{{min}}) } } f_0(v_0) \alpha = \kappa C_0 f_0(v_0) \alpha,
\end{equation}
with
\begin{equation}\label{eq:kappa}
\kappa = \frac{ -(\Sigma^+_{22} + \Sigma^-_{22}) \gamma(x^+_{{max}}, Y^\pm(x^+_{{max}}), Z^+(x^+_{{max}})) }{ \Sigma^+_{11} (x^+_{{max}} - x^+_{{min}}) }.
\end{equation}

Moreover, evaluating \eqref{DLpetit} at $\alpha = 0$ gives
\[
g(r, 0) \underset{r \to \infty}{\sim} c_1(0) \cdot \frac{e^{-r\M}}{r^{3/2}}.
\]
Since $c_1(\alpha) = \tilde c_1(\alpha)$ is continuous at $\alpha = 0^+$, 
we have $c_1(\alpha) \underset{\alpha \to 0}{\longrightarrow} C_0 f(v_0) > 0$, we conclude that the full expansion \eqref{as:collalpha=0} holds as claimed.
\end{proof}

The next lemma shows that in representations \eqref{I11}, \eqref{I22}, \eqref{I22''} and \eqref{I22'''}, the integrals along $S^\pm_{x,\alpha}$ and $S^\pm_{y,\alpha}$ are negligible compared to those along $\Gamma_{x,\alpha}$ and $\Gamma_{y,\alpha}$.

\begin{lemma}[Negligibility of the integrals along $S_{x,\alpha}^\pm$ and $S_{y,\alpha}^\pm$] 
\label{lem:neglig1}
Let $v_0 \in \R^2$. For $(a,b) \in \R \times (0,+\infty)$, let $\alpha(a,b) \in (0,\pi)$ be the angle such that $\cos (\alpha) = \frac{ a }{ \sqrt{a^2+b^2}}$ and $\sin(\alpha) = \frac{ b }{\sqrt{a^2+ b^2}}$.
\begin{itemize}
\item Suppose $\alpha_0 \in (0, \pi)$. Then, for $\eta > 0$ small enough, there exist $r_0 > 0$ and $D > 0$ such that for all $(a,b)$ satisfying $\sqrt{a^2+b^2} > r_0$ and $|\alpha(a,b) - \alpha_0| < \eta$, we have:
\begin{equation}
    \label{truc}
\left| \int\limits_{S^{+}_{x, \alpha}} \phi^{v_0}(x)\frac{\gamma(x,Y^+(x),Z^+(x))}{\partial_y\gamma_+(x,Y^+(x))} e^{-ax - bY^+(x)} \, dx \right| \leq \frac{D}{b} e^{-a x(\alpha) - b y(\alpha) - \epsilon^2 \sqrt{a^2+b^2}}.
\end{equation}
If $b > b_0$, then the following estimate also holds:
\end{itemize}
\begin{equation}
\label{truct}
\left| \int\limits_{S^{+}_{x, \alpha}}  
\frac{e^{(a_0-a)x} \left(
e^{(b_0-b)Y^+(x)}\fc_{b_0>0} 
+ e^{b_0Z^+(x) - bY^+(x)}\fc_{b_0 < 0}
\right)}{\partial_y\gamma_+(x,Y^+(x))} \, dx \right| \leq  
  \frac{D}{b - b_0} e^{-a x(\alpha) - b y(\alpha) - \epsilon^2 \sqrt{a^2 + (b - b_0)^2}}.
\end{equation}  
\begin{itemize}
\item Suppose ${\M} = x^{+}_{\max}$ and let $\alpha_0 = 0$. Then, for $\eta > 0$ small enough, there exist $r_0 > 0$ and $D > 0$ such that for all $(a,b)$ satisfying $\sqrt{a^2 + b^2} > r_0$ and $0 < \alpha(a,b) \leq \eta$, we have:
\begin{equation}
    \label{truc''}
\left| \int\limits_{S^{+}_{x, \alpha}} \phi^{v_0}(x) \frac{\gamma(x,Y^+(x),Z^+(x))}{\partial_y\gamma_+(x,Y^+(x))} e^{-ax - bY^+(x)} \, dx \right| \leq D e^{-a x(\alpha) - b y(\alpha) - \epsilon^2 \sqrt{a^2 + b^2}},
\end{equation}

\begin{equation}
    \label{truc1}
\left| \int\limits_{S^{+}_{y, \alpha}}
\frac{e^{(a_0 - a)X^+(y) + (b_0 - b)y}}{\partial_x\gamma_+(X^+(y),y)} \, dy \right| \leq \frac{D}{a} e^{-a x(\alpha) - b y(\alpha) - \epsilon^2 \sqrt{a^2 + b^2}}.
\end{equation}

\item Symmetric estimates hold for $\alpha_0 = \pi$ when ${\m} = x^{+}_{\min}$.
\end{itemize}
The same estimates hold for $S^{-}_{x, \alpha}$ and $S^{-}_{y, \alpha}$.
\end{lemma}

\begin{proof}
First, let us consider the case $0 < \alpha_0 < \pi$. Using the notation \eqref{fff}, inequality \eqref{truc} can be written as
\begin{equation}
    \label{eq:lefths}
\left| \int\limits_{s>0}   
\frac{\phi^{v_0}(x_\alpha^+ +is) \gamma(x_\alpha^+ + is, Y^+(x_\alpha^+ + is), Z^-(x_\alpha^+ + is))}{\partial_y\gamma(x_\alpha^+ + is, Y^+(x_\alpha^+ + is))} e^{\sqrt{a^2 + b^2} (F(x_\alpha^+ + is,\alpha) - F(x_\alpha^+, \alpha))} \, ds \right| \leq \frac{D}{b}.
\end{equation}
Furthermore,
\begin{equation}
    \label{bb}
\left| \exp\big( \sqrt{a^2 + b^2} (F(x_\alpha^+ + is, \alpha) - F(x_\alpha^+, \alpha)) \big) \right| = \exp \big( -b \big( \Re Y^+(x_\alpha^+ + is) - \Re Y^+(x_\alpha^+) \big) \big).
\end{equation}
By \eqref{Ypm}, the quantity $\Re Y^+(x_\alpha^+ + is) - \Re Y^+(x_\alpha^+)$ vanishes at $s = 0$ and grows linearly as $s \to +\infty$. Therefore, there exist constants $c > 0$ and $\eta > 0$ such that for all $\alpha$ satisfying $|\alpha - \alpha_0| \leq \eta$ and all $s \geq 0$,
\begin{equation}
    \label{cv}
\Re Y^+(x_\alpha^+ + is) - \Re Y^+(x_\alpha^+) \geq c s.
\end{equation}

By \eqref{Ypm} and \eqref{Zpm}, the quantity $\gamma(u + is, Y^\pm(u + is), Z^+(u + is))$ grows linearly as $s \to +\infty$, uniformly in $u \in [-\m - \e, \M + \e]$ for a given $\e > 0$. Similarly, $\partial_y \gamma_+(u + is, Y^+(u + is))$ tends to $0$ as $s \to +\infty$, also uniformly in $u \in [-\m - \e, \M + \e]$. Then, by \eqref{laplace explicitee}, we have
\begin{equation}
    \label{zlm2}
\sup_{s \geq 0, |\alpha - \alpha_0| \leq \eta} \left| \frac{ \phi^{v_0}(x_\alpha^+ + is) \gamma(x_\alpha^+ + is, Y^+(x_\alpha^+ + is), Z^+(x_\alpha^+ + is))}{\partial_y \gamma_+(x_\alpha^+ + is, Y^+(x_\alpha^+ + is))} \right| < +\infty.
\end{equation}
Then, for all $(a, b)$ such that $|\alpha(a,b) - \alpha_0| \leq \eta$, the left-hand side of \eqref{eq:lefths} is bounded by
$$
C \int_0^{\infty} e^{-bcs} \, ds = \frac{C}{cb},
$$
where $C > 0$ is a constant independent of $(a, b)$. This proves \eqref{truc}. Inequalities \eqref{truct} and \eqref{truc1} follow by analogous symmetric arguments.

Now, let us consider the case $\alpha_0 = 0$ and prove \eqref{truc''}. This inequality can be written as \eqref{eq:lefths}, but with the right-hand side $\frac{D}{b}$ replaced by $D$. The exponential term
\[
\exp\left( - \sqrt{a^2 + b^2} \left( F(x_\alpha^+ + is, \alpha) - F(x_\alpha^+, \alpha) \right) \right)
\]
can be bounded by $1$. By similar arguments, there exists $\eta > 0$ such that
\begin{equation}
C' := \sup_{s \geq 0, \, 0 < \alpha \leq \eta} \left| \frac{\gamma(x_\alpha^+ + is, Y^+(x_\alpha^+ + is), Z^+(x_\alpha^+ + is))}{\partial_y \gamma_+(x_\alpha^+ + is, Y^+(x_\alpha^+ + is))} \right| < \infty.
\end{equation}
If $b_0 \neq 0$, then by \eqref{Re} and \eqref{laplace explicitee}, we have
\begin{equation}
C'' := \sup_{0 < \alpha \leq \eta} \int_0^{+\infty} \left| \phi^{s_0}(x_\alpha^+ + is) \right| \, ds < +\infty,
\end{equation}
which gives the desired result with $D = C' C''$. If $b_0 = 0$, an integration by parts argument justifies the same estimate, using the explicit expression \eqref{laplace explicitee} (see \cite[Lemma 7.3]{Franceschi_2024} for similar estimates).
\end{proof}

We now have the tools to prove Theorems~\ref{thm:1} and~\ref{thm:2}.

\begin{proof}[Proof of Theorem~\ref{thm:1}]
Let $\alpha_0 \in (0, \pi)$ be such that $\m < x(\alpha_0) < \M$, and let $v_0 \in \R^2$. By Lemmas~\ref{double_simple} and~\ref{lem:chemin}, the function $g^{v_0}(a,b)$ can be written as $g^{v_0}(a,b) = I_1^{v_0}(a,b) + I_2^{v_0}(a,b)$, where $I_1^{v_0}(a,b)$ and $I_2^{v_0}(a,b)$ are given by \eqref{I11} and \eqref{I22}. 
Using the asymptotic expansion given in \eqref{as:coll} from Lemma~\ref{lem:contribsaddle}, together with the bounds \eqref{truc} and \eqref{truct} from Lemma~\ref{lem:neglig1}, we see that the contributions of the integrals along the segments $S^{\pm}_{x,\alpha}$ are negligible compared to those along $\Gamma_{x,\alpha}$. Therefore, the asymptotics of $g^{v_0}(a,b)$ are given by \eqref{as:coll}.
\end{proof}

\begin{proof}[Proof of Theorem~\ref{thm:2}]
Without loss of generality, we consider the case $\alpha_0 = 0$. We use the representations of $I_1^{v_0}(a,b)$ and $I_2^{v_0}(a,b)$ given by \eqref{I11} and \eqref{I22''}. 
Then, applying the asymptotic expansion \eqref{as:collalpha=0} from Lemma~\ref{lem:colalpha=0} and the estimates \eqref{truc''} and \eqref{truc1} from Lemma~\ref{lem:neglig1}, we conclude the proof in the same way as for Theorem~\ref{thm:1}.
\end{proof}

\subsection{Lemmas for directions $\alpha_b$ and $\tilde\alpha_b$}\label{sub:alphab}

Without loss of generality, we consider the asymptotic behaviour of $g^{v_0}$ in the direction $\alpha_b$. Since the integrands in \eqref{I1} and \eqref{I2} are no longer holomorphic at points $x > \M$, we must introduce a different deformation of the integration path; see Figure~\ref{colchiant}. Note that the integrands of \eqref{I1} and \eqref{I2} admit limits at the top and bottom edges of the branch cut $[\M, +\infty)$. To define the new contour, we introduce the following notation.

\textbf{Notation.} 
Let $M > \M$, and let $f$ be a continuous function defined in a neighbourhood of $[\M, M]$ within $\{x \in \C : \Im(x) \geq 0\}$. We define
\[
\int_{\Omega_{+,M}} f(x)\, dx := \lim_{\eta \to 0} \int_{\Omega_{+,M}^\eta} f(x)\, dx,
\]
where $\Omega_{+,M}^\eta$ denotes the straight segment connecting $M + i\eta$ to $\M + i\eta$, oriented as in Figure~\ref{defint}. Integration along $\Omega_{-,M}$ is defined analogously, following the orientation shown in Figure~\ref{defint}.

If $\alpha < \alpha_b$ (i.e., $x(\alpha) > \M$), we set $\Omega_{\pm, \alpha} = \Omega_{\pm, x(\alpha)}$. If $\alpha > \alpha_b$ (i.e., $x(\alpha) < \M$), we adopt the convention $\Omega_{+,\alpha} = \Omega_{-,\alpha} = \{x(\alpha)\}$, i.e., the constant path reduced to the point $x(\alpha)$. We also define
\[
\Gamma^-_{x, \alpha} = \{x(it, \alpha) \mid t \in [-\e, 0)\}, \qquad \Gamma^+_{x, \alpha} = \{x(it, \alpha) \mid t \in (0, \e]\}.
\]

We can thus consider integrals of the relevant quantities along $\Omega_{\pm,\alpha}$. Let us now define the function $\Phi$ as
\begin{align}
\label{Phi}
\Phi(x) &= \phi^{v_0}(x)\frac{\gamma(x, Y^+(x), Z^+(x))}{\partial_y\gamma_+(x, Y^+(x))} + \frac{e^{a_0x + b_0 Y^+(x)\fc_{b_0>0} + b_0 Z^+(x)\fc_{b_0<0}}}{\partial_y\gamma_+(x, Y^+(x))} \\
\label{DLL}
&=\frac{e^{a_0 x + Z^+(x) b_0 \fc_{b_0 < 0}}}{\partial_y\gamma_+(x, Y^+(x))} \left(-\frac{\gamma(x, Y^+(x), Z^+(x))}{\gamma(x, Y^-(x), Z^+(x))} e^{b_0 Y^-(x) \fc_{b_0 > 0}} + e^{b_0 Y^+(x) \fc_{b_0 > 0}} \right).
\end{align}
Observe that $\Phi$ is discontinuous along the cut $[\M, +\infty)$. The following lemma provides the appropriate contour deformation when $\alpha_0 = \alpha_b$.

\begin{lemma}[Change of contour for direction $\alpha_b$] \label{découpage}
Assume that $x_{max}^+ < x_{max}^-$, i.e., $\alpha_b \in (0,\pi)$. Let $v_0 \in \R^2$, $a \in \R$, and $b > b_0$ such that $(a, b) \neq v_0$.
Then, for $0 < \alpha < \alpha_b$ (i.e., $x_{max}^+ < x(\alpha)$), we have
\begin{equation}\label{I111}
I_1^{v_0}(a,b) + I_2^{v_0}(a,b) = \frac{1}{2i\pi} \int_{S^-_{x, \alpha} \cup \Gamma^-_{x,\alpha} \cup \Omega_{-,\alpha} \cup \Omega_{+,\alpha} \cup \Gamma^+_{x,\alpha} \cup S^+_{x, \alpha}} \Phi(x) e^{-a x - b Y^+(x)}\, dx.
\end{equation}
\end{lemma}

\begin{proof}
We apply the Residue theorem to the representations \eqref{I1} and \eqref{I2} of $I_1^{v_0}(a,b)$ and $I_2^{v_0}(a,b)$, using the contour described in Figure~\ref{colchiant}. The asymptotic as $M \to +\infty$ is justified by \eqref{bout_neg}, and the limit as $\eta \to 0$ follows from the dominated convergence theorem. This yields the desired identity \eqref{I111}.
\end{proof}

\begin{figure}[t]
\centering
\begin{subfigure}[b]{0.45\linewidth}
  \centering
  \includegraphics[width=\linewidth]{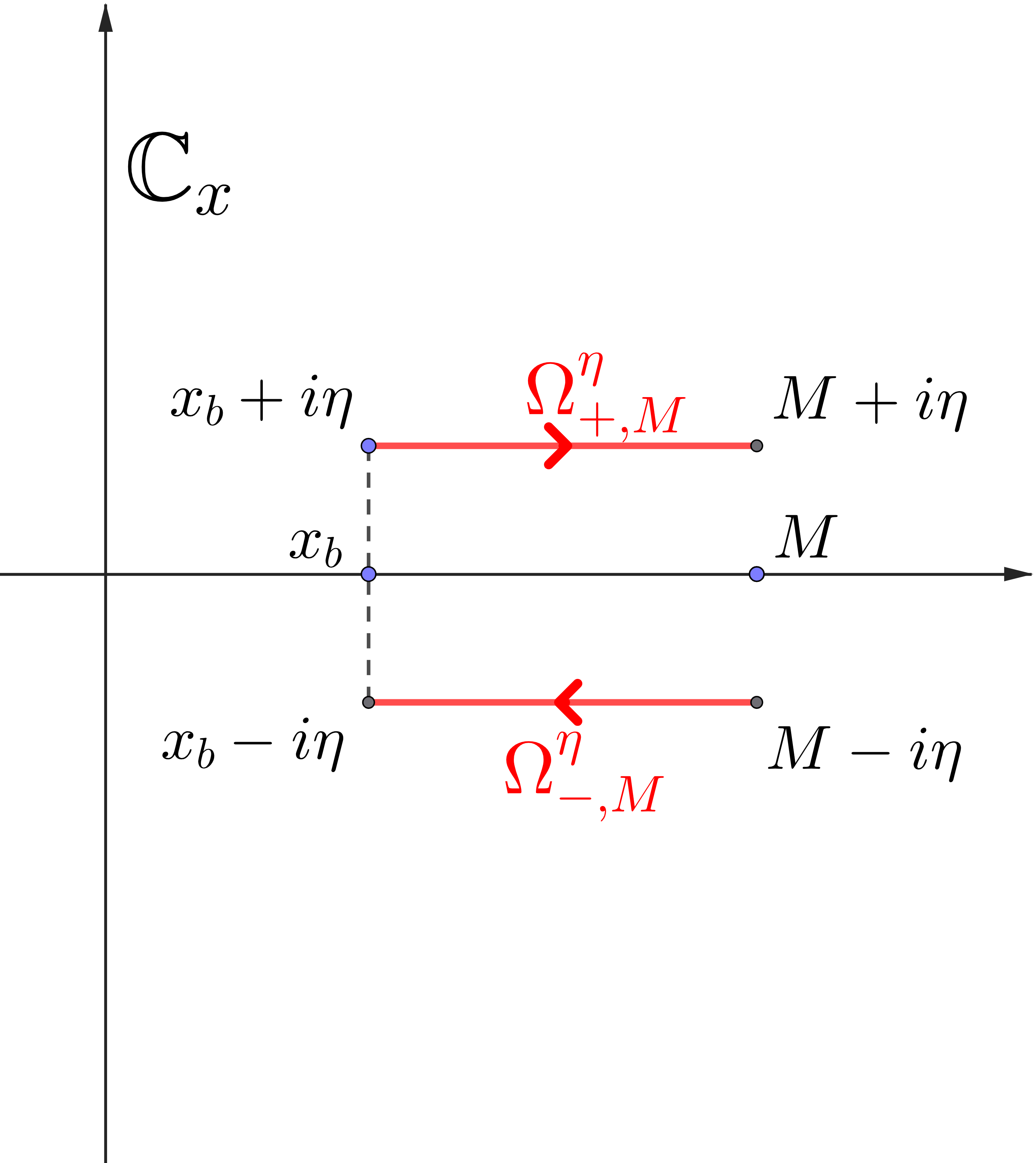}
  \caption{Paths $\Omega^\eta_{+,M}$ and $\Omega^\eta_{-,M}$.}
  \label{defint}
\end{subfigure}
\hfill
\begin{subfigure}[b]{0.45\linewidth}
  \centering
  \includegraphics[width=\linewidth]{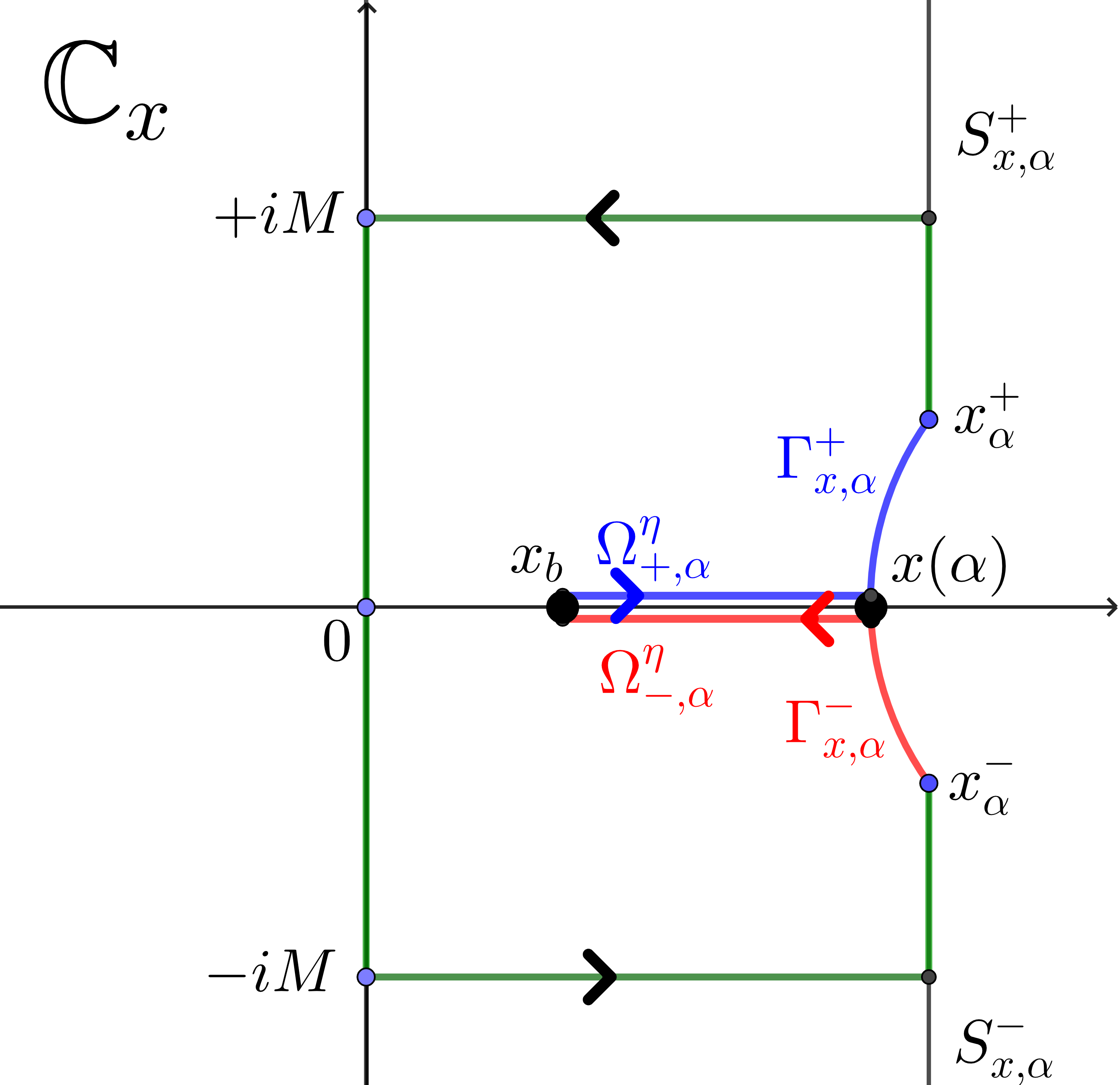}
  \caption{Contour for the Residue Theorem. Paths $\Omega_{+,\alpha}$ and $\Omega_{-,\alpha}$ coincide in $\C_x$.}
  \label{colchiant}
\end{subfigure}
\caption{Illustration of the integration paths $\Omega^\eta_{+,\alpha}$ and the contour deformation for direction $\alpha_b$.}
\end{figure}

For the asymptotic behaviour of Green's functions as $\alpha \to \alpha_b$, both contours $\Omega_{\pm,\alpha}$ and $\Gamma^{\pm}_{x,\alpha}$ contribute to the leading term. Since the function $\Phi$ is not holomorphic at ${\M}$, the standard saddle-point expansion does not apply. However, we present an adapted version of the method.

Recall the notation \eqref{OMEGA}. From the expressions \eqref{Ypm} and \eqref{laplace explicitee}, there exists a neighbourhood $V$ of ${\M}$, containing the disc $\mathbb{D}(0,L)$ for some $L > 0$, and a holomorphic function $\psi$ on $V$ such that
\[
\Phi(x) = \psi\left(\sqrt{{\M} - x}\right), \quad \text{for all } x \in V \setminus [{\M}, +\infty).
\]
By separating the even and odd parts in the power series expansion of $\psi$ around ${\M}$, we may write  
\begin{equation} \label{Phi1Phi2}  
\Phi(x) = \Phi_1(x) + \sqrt{{\M} - x} \, \Phi_2(x)
\end{equation}  
for $x \in V \setminus [{\M}, +\infty)$, where $\Phi_1$ and $\Phi_2$ are holomorphic in a neighbourhood of ${\M}$.
This decomposition allows us to establish the following lemma.

\begin{lemma}[Contribution of $\Gamma^{\pm}_{x,\alpha}$ in direction $\alpha_b$]\label{lem:colalphab}
Let $v_0 \in \R^2$. Suppose $x_{\max}^+ < x_{\max}^-$. Then, the following asymptotic expansion holds:
\begin{equation}\label{jujuju}
\frac{1}{2i\pi}\int_{\Gamma^-_{x,\alpha}\cup \Gamma^+_{x,\alpha}}\Phi(x)\,e^{-r(\cos(\alpha)x+\sin(\alpha)Y^+(x))}\, dx \underset{r\to+\infty\atop \alpha\to\alpha_b}{\sim} C^+(\alpha_b)\,h_{\alpha_b}(v_0)\,\frac{e^{-r(\cos(\alpha)x(\alpha) + \sin(\alpha)y(\alpha))}}{\sqrt r}
\end{equation}
where $C^+(\alpha_b)$ and $h_{\alpha_b}(v_0)$ are given by \eqref{C(alpha)} and \eqref{Harm1}.
\end{lemma}

\begin{proof}
Applying the saddle-point method to the holomorphic function $\Phi_1$, as developed in \cite[Lemma 8.1]{Franceschi_2024}, we obtain:
\begin{equation}\label{jujujuju}
\frac{1}{2i\pi}\int_{\Gamma_{x,\alpha}}\Phi_1(x)\,e^{-r(\cos(\alpha)x+\sin(\alpha)Y^+(x))}\, dx \underset{r\to+\infty\atop \alpha\to\alpha_b}{\sim} C^+(\alpha_b)\,h_{\alpha_b}(v_0)\,\frac{e^{-r(\cos(\alpha)x(\alpha) + \sin(\alpha)y(\alpha))}}{\sqrt r},
\end{equation}
since $\Phi(\M) = \Phi_1(\M)$.
Moreover, for all $(t,\alpha) \in \Omega(0,\alpha_b)$, we have
\[
\left|\sqrt{\M - x(it, \alpha)}\right| \leq \sqrt{|x(0,\alpha_b) - x(0, \alpha)| + |x(0,\alpha) - x(it, \alpha)|} \leq C\left(\sqrt{|\alpha_b - \alpha|} + \sqrt{|t|}\right),
\]
for some constant $C > 0$ independent of $t$ and $\alpha$. Since $\Phi_2$ is holomorphic, it is bounded on every compact sets. Therefore, there exist constants $M, M' > 0$ such that, for all $\alpha$ with $|\alpha - \alpha_b| < \eta$, we have:
\[\left|\int_{\Gamma^+_{x,\alpha}}\sqrt{\M-x}\,\Phi_2(x)\,e^{-r(\cos(\alpha)x+\sin(\alpha)Y^+(x))} dx\right|\]
\begin{align}
&= e^{-r(\cos(\alpha)x(\alpha) + \sin(\alpha)y(\alpha))}\left|\int_{0}^{\e} \sqrt{\M - x(it, \alpha)}\,\Phi_2(x(it, \alpha))\,e^{-rt^2} i x'_\omega(it,\alpha)\,dt\right| \notag\\
&\leq M\,e^{-r(\cos(\alpha)x(\alpha) + \sin(\alpha)y(\alpha))} \int_{0}^{\e} \left(\sqrt{|\alpha_b - \alpha|} + \sqrt{t}\right)e^{-rt^2}dt \label{ineq_M_indepz0} \\
&\underset{r\to+\infty\atop \alpha\to\alpha_b}{\sim} M'\,e^{-r(\cos(\alpha)x(\alpha) + \sin(\alpha)y(\alpha))} \left(\frac{\sqrt{|\alpha_b - \alpha|}}{\sqrt{r}} + \frac{1}{r}\right),\notag
\end{align}
which is negligible compared to \eqref{jujujuju} as $r \to +\infty$ and $\alpha \to \alpha_b$. The same estimate holds for the integral over $\Gamma^-_{x,\alpha}$. This completes the proof of \eqref{jujuju}.
\end{proof}

We split the proof of the contribution of the integrals along $\Omega_{\pm,\alpha}$ into two lemmas.

\begin{lemma}[Value of $\Phi_2$ at $\M$] \label{tres_chiant}
Recall the notation from \eqref{Phi1Phi2}. Suppose $x_{max}^- < x_{max}^+.$ Then:
$$\Phi_2(x_b) =  - 2\frac{\sqrt{\det(\Sigma^-)(\M - x_{min}^-)}}{\Sigma_{22}^-(\Sigma_{22}^+ + \Sigma_{22}^-)}f_0(v_0)$$
where $f_0(v_0)$ is given by \eqref{f(v_0)branch}.
\end{lemma}

\begin{proof}
The proof reduces to showing that
\begin{equation}\label{eq:DLbienrelou}
\Phi(x) =\Phi(\M) - 2\frac{\sqrt{\det(\Sigma^-)(\M - x_{min}^-)}}{\Sigma_{22}^-(\Sigma_{22}^+ + \Sigma_{22}^-)}f_0(v_0)\sqrt{\M - x}+ o(\sqrt{\M - x})
\end{equation}
as $x \to \M$. To prove the above formula, we use the explicit expressions of $\Phi$ and $\phi^{v_0}$ given by \eqref{Phi} and \eqref{laplace explicitee}. Let us denote $C_z=\frac{\sqrt{\det(\Sigma^-)(\M - x_{min}^-)}}{\Sigma_{22}^-}$ so that
$Z^+(x) = Z^+(\M) + C_z\sqrt{\M - x} + o(\sqrt{\M - x}).$
We have:
\begin{align*}
\frac{\gamma(x,Y^+(x), Z^+(x))}{\gamma(x,Y^-(x), Z^+(x))} &= 1 +\frac{\left(\frac{1 + q_2}{2}\right)(Y^+(x)-Y^-(x))}{\gamma(x,Y^-(x), Z^+(x))} \\
&= 1 +\left(\frac{1 + q_2}{2}\right)\frac{Y^+(x)-Y^-(x)}{\gamma(\M,Y^-(\M), Z^+(\M)) + \frac{q_2 - 1}{2}C_z\sqrt{\M - x} + o(\sqrt{\M - x})} \\
&= 1 +\left(\frac{1 + q_2}{2}\right)\frac{Y^+(\M) - Y^-(\M)}{\gamma(\M,Y^-(\M), Z^+(\M))}\left(1 - \frac{\frac{q_2 - 1}{2}C_z\sqrt{\M - x}(1 + o(1))}{\gamma(\M,Y^-(\M), Z^+(\M))} \right).
\end{align*}

Furthermore,
\begin{align*}
e^{Z^+(x)b_0\fc_{b_0 < 0}} &= e^{(Z^+(\M) +\sqrt{\M - x}( C_z + o(1)))b_0\fc_{b_0 < 0}} \\
&= e^{Z^+(\M)b_0\fc_{b_0 < 0}}(1 +\sqrt{\M - x}(C_z + o(1))b_0\fc_{b_0 < 0}).
\end{align*}

Then, if $b_0 < 0$, \eqref{DLL} can be written as
\begin{align*}
\partial_y\gamma_+(x,Y^+(x))\Phi(x) &= e^{\M a_0 + Z^+(\M)b_0}\left(1 + C_z\sqrt{\M - x}(b_0 + o(1))\right) \\
&\quad \times -\left(\frac{1 + q_2}{2}\right)\frac{Y^+(\M) - Y^-(\M)}{\gamma(\M,Y^-(\M), Z^+(\M))}\left(1 - \frac{ \frac{q_2 - 1}{2}C_z\sqrt{\M - x}(1 + o(1))}{\gamma(\M,Y^-(\M), Z^+(\M))} \right)
\end{align*}

which yields \eqref{eq:DLbienrelou}, using the identity
$\partial_y\gamma_+(\M,Y^+(\M)) = \frac{\Sigma_{22}^+}{2}(Y^+(\M) - Y^-(\M)).$
If $b_0 > 0$, the computations are analogous.
\end{proof}

Before setting the asymptotics along $\Omega_{\pm,\alpha}$, we need to introduce some notation.
We recall that $x(0, \alpha_b) = \M$ and that $x'_\omega(0,\alpha_b) = \sqrt{\frac{2}{F''_x(x(\alpha_b), \alpha)}} \neq 0$ by \eqref{fzfz}. Then, by the implicit function theorem, there exists a $C^\infty$ function, denoted by $\tau$, and a neighbourhood $W$ of $(0, \alpha_b)$ in $\R_\omega \times \R_\alpha$ (contained in $\Omega(0,\alpha_b)$, see~\eqref{OMEGA}) such that
$$\forall (\omega, \alpha) \in W, \quad [x(\omega, \alpha) = \M] \Longleftrightarrow [\omega = \tau(\alpha)].$$
In particular, $x(\tau(\alpha), \alpha) = \M$, and $\tau'(\alpha_b) = \frac{-x'(\alpha_b)}{x'_\omega(0,\alpha_b)} > 0.$
Then, for $\alpha < \alpha_b$, we have $\tau(\alpha) < 0$ and
\[
\tau(\alpha) = \frac{-x'(\alpha_b)}{x'_\omega(0,\alpha_b)}(\alpha - \alpha_b)(1 + o_{\alpha\to\alpha_b}(1)).
\]

Note that, by \eqref{eq_link_cste} and \eqref{zop}, we have $|K_+| = \frac{-x'(\alpha_b)}{x'_\omega(0,\alpha_b)}$
with the notation \eqref{eq:K+2}. Since, by \eqref{eq:F=t2}, $F(\M, \alpha) = F(x(\tau(\alpha), \alpha), \alpha) = \tau^2(\alpha)$, we deduce from \eqref{fff} the following identity:
\begin{equation}\label{eq:xalpha_tau}
    \cos(\alpha)x(\alpha_b) + \sin(\alpha)y(\alpha_b) = \cos(\alpha)x(\alpha) + \sin(\alpha)y(\alpha) + \tau^2(\alpha).
\end{equation}

\begin{lemma}[Contribution of $\Omega_{\pm,\alpha}$ in direction $\alpha_b$] \label{lem:contribomega} Suppose that $x_{max}^+ < x_{max}^-$.
Then, \begin{itemize}
    \item If $r(\alpha - \alpha_b)^2$ remains bounded, then the asymptotic contribution of the integrals along $\Omega_{\pm,\alpha}$ in \eqref{I111} is negligible compared to \eqref{jujuju} as $r\to+\infty$ and $\alpha \to \alpha_b$.
    \item If $r(\alpha - \alpha_b)^2 \to +\infty$ and $\alpha \leq \alpha_b$, the following asymptotics hold:
\begin{equation}\label{eq:contribomega}
\frac{1}{2i\pi}\int_{\Omega_{\pm,\alpha}}\Phi(x)e^{-r(\cos(\alpha)x+\sin(\alpha)Y^+(x))}
 dx
\end{equation}
$$  \underset{r\to+\infty\atop \alpha\to\alpha_b, \alpha \leq \alpha_b}{\sim}\frac{\sin^{3/2}(\alpha_b)\sqrt{\det(\Sigma^-)(\M - x_{min}^-)}}{\sqrt \pi \Sigma_{22}^-(\Sigma_{22}^+ + \Sigma_{22}^-)} f_0(v_0)e^{-r(\cos(\alpha)x(\alpha) + \sin(\alpha)y(\alpha))} \frac{e^{r\tau(\alpha)^2}}{(r(\alpha_b - \alpha))^{3/2}}.$$
\end{itemize}
\end{lemma}

\begin{proof}
First, since $\Phi_1$ is holomorphic, we have $$\frac{1}{2i\pi}\int_{\Omega_{x,\alpha}^-\cup\Omega^+_{x,\alpha}}\Phi_1(x)e^{-r(\cos(\alpha)x+\sin(\alpha)Y^+(x))}
 dx = 0.$$
For $\Phi_2$, we note that $$\int_{\Omega^-_{x,\alpha}} \sqrt{{\M}-x}\Phi_2(x) e^{-ax -bY^+(x)}dx = \int_{\Omega^+_{x,\alpha}} \sqrt{{\M}-x}\Phi_2(x) e^{-ax -bY^+(x)}dx$$ by the definition of integration over $\Omega_{\pm,\alpha}$ and the convention for the complex square root.
We parameterise the path $\Omega_{+,\alpha}$ as $\Omega_{+,\alpha}=\{x(t,\alpha) \mid t \in [\tau(\alpha), 0]\}$. Then,
\begin{equation}\label{intomega}
\int_{\Omega^+_{x,\alpha}} \sqrt{{\M}-x}\Phi_2(x) e^{-r(\cos(\alpha)x +\sin(\alpha)Y^+(x))}dx
\end{equation} $$= -ie^{-r(\cos(\alpha)x(\alpha) + \sin(\alpha)y(\alpha))} \int_{\tau(\alpha)}^0 \sqrt{x(t,\alpha) - {\M}}\Phi_2(x(t,\alpha)) e^{rt^2}x'_\omega(t,\alpha)dt.$$
Moreover, for any $t\in(\tau(\alpha),0)$, there exists $\xi_{t,\alpha} \in (\tau(\alpha),t)$ such that $$x(t,\alpha) - {\M} = x(t,\alpha) - x(\tau(\alpha),\alpha) = x'_\omega(\xi_{t,\alpha}, \alpha)(t-\tau(\alpha)).$$
Therefore,
\begin{equation}\label{inégalité_qui_simplifie}
\left|\int_{\tau(\alpha)}^0 \sqrt{x(t,\alpha) - {\M}}\Phi_2(x(t,\alpha)) e^{rt^2}x'_\omega(t,\alpha)dt - \Phi_2(x(0,\alpha))x'_\omega(0,\alpha)^{3/2}\int_{\tau(\alpha)}^0 \sqrt{t - \tau(\alpha)} e^{rt^2}dt\right|\end{equation}
$$\leq \e(\alpha) \int_{\tau(\alpha)}^0 \sqrt{t-\tau(\alpha)} e^{rt^2}dt $$
where
%
\begin{align*}
\e(\alpha) &:= \sup_{t\in[\tau(\alpha),0]} \left|\sqrt{x'_\omega(\xi_{t,\alpha},\alpha)}\Phi_2(x(t,\alpha))x'_\omega(t,\alpha) - \sqrt{x'_\omega(0,\alpha)}\Phi_2(x(0,\alpha))x'_\omega(0,\alpha)  \right|\\
&\leq \sup_{t\in[\tau(\alpha),0]\atop\xi\in[\tau(\alpha),0]}\left|\sqrt{x'_\omega(\xi,\alpha)}\Phi_2(x(t,\alpha))x'_\omega(t,\alpha) - \sqrt{x'_\omega(0,\alpha)}\Phi_2(x(0,\alpha))x'_\omega(0,\alpha)  \right|  \underset{\alpha \to \alpha_b}{\longrightarrow} 0.
\end{align*}

Using the change of variables $t = \tau(\alpha)s$, we have \begin{equation}\label{eq:inttau(alpha)}
\int_{\tau(\alpha)}^0 \sqrt{t-\tau(\alpha)}\, e^{rt^2}dt =|\tau(\alpha)|^{3/2}\int_{0}^1 \sqrt{1-s}\, e^{r\tau(\alpha)^2s^2}ds.
\end{equation}

Assume first that $r\tau(\alpha)^2$ remains bounded.Then, \eqref{eq:inttau(alpha)} is also bounded. By \eqref{inégalité_qui_simplifie}, \eqref{intomega} is then bounded by $|\tau(\alpha)|^{3/2} e^{-r(\cos(\alpha)x(\alpha) + \sin(\alpha)y(\alpha))} \leq M\sqrt{|\tau(\alpha)|}\frac{e^{-r(\cos(\alpha)x(\alpha) + \sin(\alpha)y(\alpha))}}{\sqrt r} $ for some constant $M > 0$, which is negligible compared to \eqref{jujuju}. This proves $(ii)$. 

Now, suppose that $r\tau(\alpha)^2 \to +\infty$. By Lemma~\ref{lemme_technique}, the following asymptotics hold:
$$|\tau(\alpha)|^{3/2}\int_{0}^1 \sqrt{1-s} e^{r\tau(\alpha)^2s^2}ds \sim |\tau(\alpha)|^{3/2}\frac{\sqrt \pi}{4\sqrt 2}\frac{e^{r\tau(\alpha)^2}}{(r\tau(\alpha)^2)^{3/2}} = \frac{\sqrt \pi}{4\sqrt 2}\frac{e^{r\tau(\alpha)^2}}{(r|\tau(\alpha)|)^{3/2}} $$
Hence,
\begin{align*}
\frac{1}{2i\pi}\int_{\Omega_{x,\alpha}^-\cup\Omega^+_{x,\alpha}}\Phi(x)&e^{-r(\cos(\alpha)x+\sin(\alpha)Y^+(x))}
 dx \\
 &\sim  \frac{-i}{i\pi}\frac{\sqrt \pi}{4\sqrt 2} \Phi_2(\M)x'_\omega(0,\alpha_b)^{3/2}e^{-r(\cos(\alpha)x(\alpha) + \sin(\alpha)y(\alpha))} \frac{e^{r\tau(\alpha)^2}}{(r\tau(\alpha))^{3/2}}\\
&\sim  \frac{-\sin^{3/2}(\alpha_b)}{2\sqrt \pi} \Phi_2(\M)e^{-r(\cos(\alpha)x(\alpha) + \sin(\alpha)y(\alpha))} \frac{e^{r\tau(\alpha)^2}}{(r\sin(\alpha_b - \alpha))^{3/2}}
\end{align*}
since 
\begin{align*}
\frac{x'_\omega(0,\alpha_b)^3}{2\sqrt{2}|x'(\alpha_b)|^{3/2}}&= \frac{1}{|F''(x(\alpha_b),\alpha_b)x'(\alpha_b)|^{3/2}} = |\sin(\alpha_b)|^{3/2} 
\end{align*} (see \eqref{eq_link_cste} and \eqref{zop}). The conclusion follows from Lemma~\ref{tres_chiant}.
\end{proof}

\begin{proof}[Proof of Theorem~\ref{thm3}]
By Lemmas~\ref{double_simple} and~\ref{découpage}, the function $g^{v_0}$ can be written as $g^{v_0} = I_1^{v_0}(a,b) + I_2^{v_0}(a,b)$, with representation given by \eqref{I111}. According to Lemmas~\ref{lem:neglig1}, \ref{lem:colalphab}, and \ref{lem:contribomega}, the contributions of the integrals along the paths $S^\pm_{x,\alpha}$ are negligible compared to those along $\Gamma^\pm_{x,\alpha}$ and $\Omega^\pm_{x,\alpha}$. Thus, we are left with a competition between the contributions from $\Gamma^\pm_{x,\alpha}$ and $\Omega^\pm_{x,\alpha}$. 
\begin{itemize}
\item[$(i)$] If $\alpha > \alpha_b$, then the integrals along $\Omega_{\pm,\alpha}$ vanish (since $\Omega_{\pm,\alpha} = \{x(\alpha)\}$), and the result \eqref{jujuju} follows from Lemma~\ref{lem:colalphab}.
\item[$(ii)$] The statement in point $(ii)$ directly follows from Lemma~\ref{lem:contribomega}.
\item[$(iii) - (iv) - (v)$] Now suppose that $\alpha < \alpha_b$ and $r(\alpha - \alpha_b)^2 \to +\infty$. Then we have:
We have
\begin{align}\label{quotient}\frac{e^{r\tau^2(\alpha)}}{r|\tau(\alpha)|^{3/2}} &= \frac{e^{r\tau^2(\alpha)}}{r|\tau(\alpha)|^2}\sqrt{|\tau(\alpha)|} \\
&= \exp\left(r\tau^2(\alpha) - \ln(r\tau^2(\alpha)) + \frac{1}{2}\ln(|\tau(\alpha)| ) \right)\notag\\
&= \exp\left(K_+^2r(\alpha_b - \alpha)^2(1 + o(1)) + \frac{1}{2}\ln(\alpha_b - \alpha) +  \frac{1}{2}\ln(|K_+|) \right)\notag\\
&= \sqrt{|K_+|}\exp\left(K_+^2r(\alpha_b - \alpha)^2\left(\frac{\ln(\alpha_b - \alpha)}{2K_+^2r(\alpha_b - \alpha)^2}+ 1 + o(1) \right)\right).
\end{align}

If $\frac{\ln(\alpha_b - \alpha)}{2K_+^2r(\alpha_b - \alpha)^2} \to l \in [-\infty, -1) 
$, then \eqref{quotient} tends to $0$, and the asymptotics of $g(r\cos(\alpha),r\sin(\alpha))$ are governed by \eqref{jujuju}.  
Note that in this case,  $\frac{e^{K_+^2r(\alpha_b - \alpha)^2}}{r(\alpha_b - \alpha)^{3/2}} \to 0.$
If $\frac{\ln(\alpha_b - \alpha)}{2K_+^2r(\alpha_b - \alpha)^2} \to l \in (-1, 0]$, then \eqref{quotient} diverges to $+\infty$, and the asymptotics of $g(r\cos(\alpha),r\sin(\alpha))$ are given by \eqref{eq:contribomega}, which coincides with \eqref{cas_3.2} thanks to \eqref{eq:xalpha_tau}.  
In this case,  $\frac{e^{K_+^2r(\alpha_b - \alpha)^2}}{r(\alpha_b - \alpha)^{3/2}} \to +\infty.$
It remains to analyze the critical case where $\frac{\ln(\alpha_b - \alpha)}{2K_+^2r(\alpha_b - \alpha)^2} \to -1$. In this situation, we have $r\sim\frac{-\ln(\alpha_b - \alpha)}{2K_+^2(\alpha_b - \alpha)^2}$ so that
$$
r\tau^2(\alpha) = rK_+^2(\alpha_b - \alpha)^2 + O\left(r(\alpha_b - \alpha)^3\right)  = rK_+^2(\alpha_b - \alpha)^2 + o(1)
$$
and consequently,
\begin{equation}\label{eq:rtau2}
\frac{e^{r\tau^2(\alpha)}}{r|\tau(\alpha)|^{3/2}} \sim \frac{e^{K_+^2r(\alpha_b - \alpha)^2}}{r|K_+(\alpha - \alpha_b)|^{3/2}}.
\end{equation}
\nr{Hence, both contributions \eqref{jujuju} and \eqref{eq:contribomega} may be significant, and expression \eqref{eq:rtau2} provides the conclusion for points $(iii)$ to $(v)$, depending on the possible limiting behavior of $\frac{e^{K_+^2r(\alpha_b - \alpha)^2}}{r|K_+(\alpha - \alpha_b)|^{3/2}}$.}
\nb{If the quantity in \eqref{eq:rtau2} tends to zero, the contribution
in \eqref{jujuju} dominates that in \eqref{eq:contribomega},
yielding $(iii)$.
If it tends to $+\infty$, the contribution in \eqref{eq:contribomega}
dominates that in \eqref{jujuju}, yielding $(iv)$.
Finally, if it converges to a finite strictly positive constant,
both contributions are of the same order and must be added,
yielding $(v)$.}
\end{itemize}
\end{proof}

\subsection{Lemmas for the non-red-arc zones}\label{sub:apresalphab}
For angles outside of the red zones, we use a different changing path again. For $\delta > 0$, we set
\begin{equation}
R^+_{\delta} =\{\M + \delta + is \;|\; s > 0 \},\quad R^-_{\delta} =\{\M + \delta - is \;|\; s > 0 \}.
\end{equation}
The following lemma performs the appropriate path deformation; see Figure~\ref{fig:chgtcheminbranch}.
\begin{lemma}[Changing path for directions $0 \leq \alpha < \alpha_b$]\label{découpage2}
Suppose that $x_{max}^+ < x_{max}^-$ (i.e., $\alpha_b \in (0,\pi)$). Let $\delta > 0$ and let $v_0 \in \R^2$, $a \in \R$, and $b > 0$ such that $(a,b) \neq v_0$. Then, for $0 \leq \alpha < \alpha_b$,
\begin{equation}\label{I1111}
I_1^{v_0}(a,b) = \frac{1}{2i\pi}\int_{R^-_{\delta}\cup \Omega_{-,\M + \delta} \cup \Omega_{+,\M + \delta}\cup R^-_{\delta}}\phi^{v_0}(x)\frac{\gamma(x,Y^+(x),Z^+(x))}{\partial_y\gamma_+(x,Y^+(x))}e^{-ax-bY^+(x)}
 dx.
\end{equation}
Furthermore, if $b_0 > 0$ and $a > a_0$, then \eqref{I22''} holds. Finally, if $b_0 \leq 0$, then
\begin{equation}\label{I3333}
I_2^{v_0}(a,b) = \frac{1}{2i\pi}\int_{R^-_{\delta}\cup \Omega_{-,\M + \delta} \cup \Omega_{+,\M + \delta}\cup R^+_{\delta}}\frac{e^{(a_0-a)x+b_0Z^+(x) - bY^+(x)}}{\partial_y\gamma_+(x,Y^+(x))} dx.
\end{equation}
\end{lemma}

\begin{proof}
The proof is identical to that of Lemma~\ref{découpage}, except that the paths $\Gamma^\pm_{x,\alpha}$ are replaced by straight lines.
\end{proof}
\begin{figure}[t]
    \centering
    \includegraphics[scale=0.35]{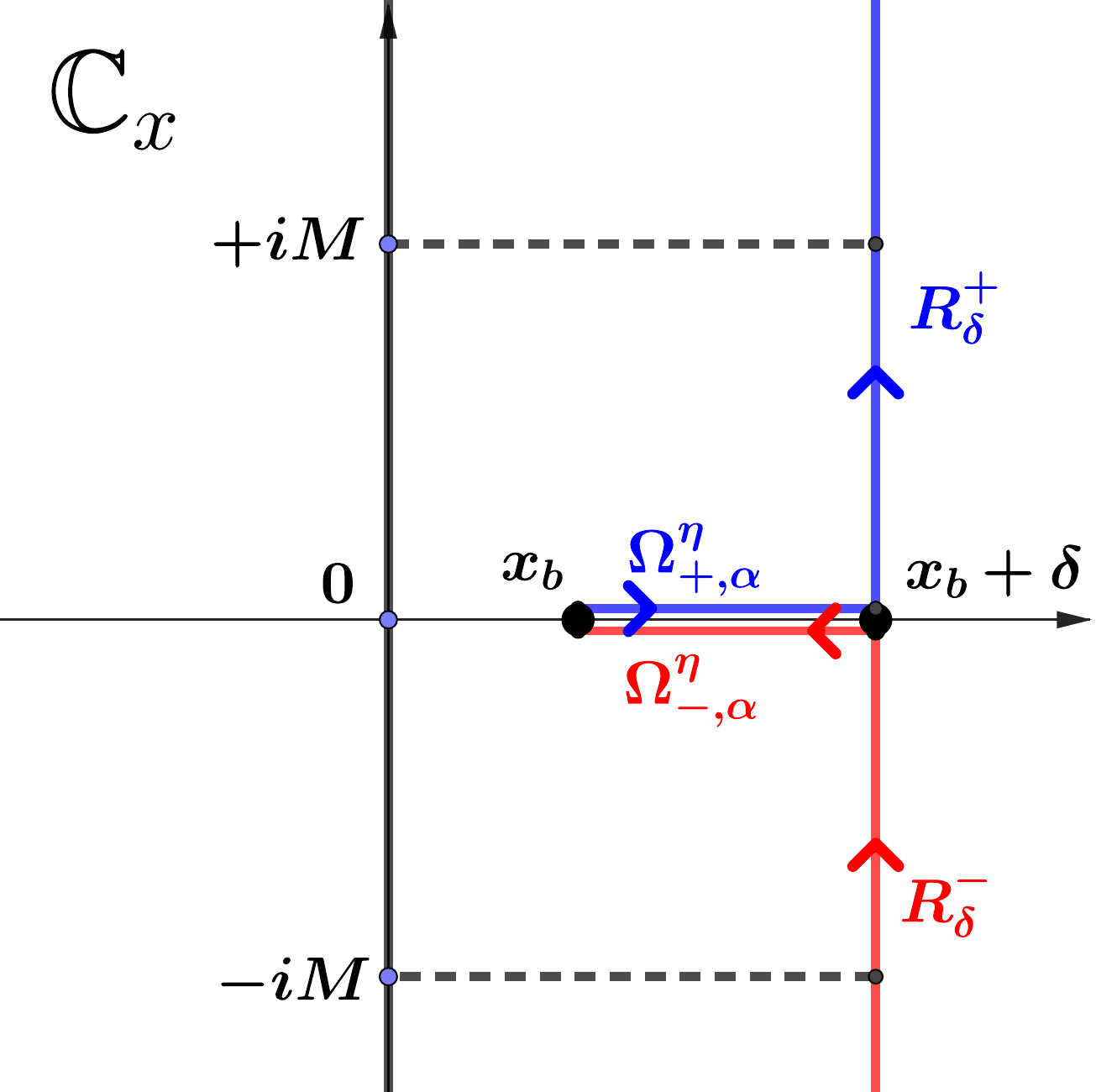}
    \caption{Path deformation in Lemma~\ref{découpage2}}
    \label{fig:chgtcheminbranch}
\end{figure}
Once again, we split the analysis of the asymptotics. We recall the notation from \eqref{Phi}. The following lemma provides the main asymptotic contribution.

\begin{lemma}[Contribution of $\Omega_{-,\M + \delta}$]\label{lem:contribOMEGAxbdelta}
Let $v_0 \in \R^2$. Suppose that $x_{max}^+ < x_{max}^-$ (i.e., $\alpha_b \in (0,\pi)$), and let $\alpha_0 \in [0, \alpha_b)$.  
If $b_0 \leq 0$, then
\begin{equation}\label{eq:contribOMEGAxbdelta}
\frac{1}{2i\pi}\int_{\Omega_{-,\M + \delta} \cup \Omega_{+,\M + \delta}}\Phi(x)e^{-r(\cos(\alpha)x+\sin(\alpha)Y^+(x))}
 dx
\underset{r\to\infty \atop\alpha\to\alpha_0, \alpha \geq 0}{\sim}  C_{br}f_0(v_0)\frac{e^{-r(\cos(\alpha)x(\alpha_b) + \sin(\alpha)y(\alpha_b))}}{r^{3/2}
\sin(\alpha_b - \alpha)^{3/2}}
\end{equation}
where $C_{br} = |\sin(\alpha_b)|^{3/2} C_0$ (see \eqref{eq:C_0}), and where $f_0(v_0)$ is defined in \eqref{f(v_0)branch}. If $b_0 > 0,$ then \begin{equation}\label{eq:ascas2}
\frac{1}{2i\pi}\int_{\Omega_{-,\M + \delta} \cup \Omega_{+,\M + \delta}}\phi^{v_0}(x)\frac{\gamma(x,Y^+(x),Z^+(x))}{\partial_y\gamma_+(x,Y^+(x))}e^{-r(\cos(\alpha)x+\sin(\alpha)Y^+(x))}
 dx
\end{equation}
$$\underset{r\to\infty \atop\alpha\to\alpha_0, \alpha \geq 0}{\sim}  C_{br}f_0(v_0)\frac{e^{-r(\cos(\alpha)x(\alpha_b) + \sin(\alpha)y(\alpha_b))}}{r^{3/2}
\sin(\alpha_b - \alpha)^{3/2}}.$$
\end{lemma}

\begin{proof}
First, assume that $b_0 \leq 0$. Once again, by the convention on the complex square root,
\begin{align*}
\frac{1}{2i\pi}\int_{\Omega_{-,\M + \delta} \cup \Omega_{+,\M + \delta}}&\Phi(x)e^{-ax-bY^+(x)}dx \\
&= \frac{2}{2i\pi}\int_{\Omega_{+,\M + \delta}}\sqrt{{\M}-x}\Phi_2(x)e^{-r(\cos(\alpha) x + \sin(\alpha)Y^+(x))}dx\\
&=\frac{e^{-r(\cos(\alpha)x(\alpha_b) + \sin(\alpha)y(\alpha_b))}}{i\pi}\int_{\Omega_{+,\M + \delta}}\sqrt{{\M}-x}\Phi_2(x)e^{-rU(x,\alpha)}dx
\end{align*}
where \begin{equation}
U(x,\alpha) = \cos(\alpha)(x-x(\alpha_b)) + \sin(\alpha)(Y^+(x) -y(\alpha_b)).
\end{equation}
By the definition of $(x(\alpha_b), y(\alpha_b))$, we have $(Y^+)'(\M) = -\frac{\cos(\alpha_b)}{\sin(\alpha_b)}$. Therefore,
$$U'_x(\M, \alpha_0) = \cos(\alpha_0) - \sin(\alpha_0) \frac{\cos(\alpha_b)}{\sin(\alpha_b)} = \frac{\sin(\alpha_b - \alpha_0)}{\sin(\alpha_b)} > 0 .$$
Moreover, $U(\M, \alpha_0) = 0$. Then, by the implicit function theorem, there exists a neighborhood $V \subset \R \times [0,2\pi]$ of $(0, \alpha_0)$ given by
\begin{equation}\label{def:V}
    V = \{(t, \alpha), \quad |t|\leq L, |\alpha - \alpha_0| \leq \eta \}
\end{equation}
and a map $z : V \rightarrow \R$ such that, for all $(t, \alpha) \in V$,
\begin{equation}\label{equivalenceU=0}
    U(x , \alpha)  = t \Longleftrightarrow  x = z(t,\alpha)
\end{equation}
with $x_b=z(0, \alpha_0)$.
This means that $U(z(t,\alpha), \alpha) = t$. In particular, we obtain $U(z(0,\alpha), \alpha) = 0$.  Moreover, by the definition of $U(x,\alpha)$, we have $U(\M, \alpha) = 0$.  
By the uniqueness stated in \eqref{equivalenceU=0}, we then deduce that  
$z(0,\alpha) = x_b$
for every $\alpha$ satisfying $|\alpha - \alpha_0| \leq \eta$. Furthermore,
\begin{equation}\label{z'}z'_t(0,\alpha_0) = \frac{1}{U'_x(x_b, \alpha_0)} = \frac{\sin(\alpha_b)}{\sin(\alpha_b - \alpha_0)} > 0.\end{equation}
Again by the implicit function theorem, if $\delta > 0$ is sufficiently small, there exists a smooth function $t(\alpha)$ (possibly requiring a smaller $\eta$ in \eqref{def:V}), defined for $|\alpha - \alpha_0| \leq \eta$, such that
$$z(t(\alpha) ,\alpha) = \M + \delta\quad \forall |\alpha - \alpha_0| \leq \eta. $$
By choosing $L > 0$ and $\eta > 0$ sufficiently small in \eqref{def:V}, we can assume that $z'_t(t,\alpha) > 0$ for all $(t,\alpha) \in V$. Since $z(0,\alpha) = x_b$, it follows that $t(\alpha) > 0$. In particular, $t(\alpha) \to t(\alpha_0) > 0$ as $\alpha \to \alpha_0$.
We now parametrize the segment $\Omega_{+,\M + \delta}$ by $\{ z(t,\alpha) \;|\; t \in [0, t(\alpha)] \}$. Then,
$$\int_{\Omega_{+,\M + \delta}}\sqrt{{\M}-x}\Phi_2(x)e^{-rU(x,\alpha)}dx = -i\int_0^{t(\alpha)} \sqrt{z(t,\alpha)-\M}\Phi_2(z(t,\alpha))z'_t(t,\alpha)e^{-rt}dt.$$
Then, by the Taylor–Lagrange inequality,
$$\left|z(t,\alpha) - z(0,\alpha) - tz'_t(0,\alpha)\right| \leq C_1t^2,\quad \forall (t,\alpha) \in V$$
where $C_1 = \frac{1}{2}\sup_{(t,\alpha) \in V} |z''_t(0,\alpha)|$. Since $|\sqrt a - 1| = \frac{|a-1|}{\sqrt a +1} \leq |a -1|$ for $a \geq 0$, we obtain 
$$\left|\sqrt{\frac{z(t,\alpha) - z(0,\alpha)}{tz'_t(0,\alpha)}} - 1\right| \leq C_2|t|$$
for some constant $C_2 > 0$ (since $V$ can be chosen so that $1/z'_t(0,\alpha)$ is bounded on $V$). Hence,
$$\left|\sqrt{z(t,\alpha)-z(0,\alpha)} - \sqrt{tz'_t(0,\alpha)}\right| \leq C_3|t|^{3/2} \quad \forall (t,\alpha), \; 0\leq t\leq L, \; \forall|\alpha - \alpha_0| \leq \eta$$
for some constant $C_3 > 0$. Similarly, there exists a constant $C_4 > 0$ such that
$$\left|\Phi_2(z(t,\alpha))z'_t(t,\alpha) - \Phi_2(\M)z'_t(0,\alpha) \right|\leq C_4|t| \quad \forall (t,\alpha), \; 0\leq t\leq L, \; \forall|\alpha - \alpha_0| \leq \eta.$$
Thus, there exists $C > 0$ such that for all $(t,\alpha)$ satisfying $0\leq t\leq K, \;|\alpha - \alpha_0| \leq \eta$, we have $$\left|\sqrt{z(t,\alpha) - z(0,\alpha)} \Phi_2(z(t,\alpha))z'_t(t,\alpha) - \Phi_2(\M)z'_t(0,\alpha)\sqrt{tz'_t(0,\alpha)} \right|\leq C|t|^{3/2}$$
It follows that $$-i\int_0^{t(\alpha)} \sqrt{z(t,\alpha)-\M}\Phi_2(z(t,\alpha))z'_t(t,\alpha)e^{-rt}dt = -i\Phi_2(\M)|z'_t(0,\alpha)|^{3/2}\int_0^{t(\alpha)} \sqrt{t}e^{-rt}dt + R(r,\alpha)$$
where $$ |R(r,\alpha)| \leq C\int_0^{t(\alpha)}t^{3/2}e^{-rt}dt = r^{-5/2}\int_0^{rt(\alpha)}t^{3/2}e^{-t}dt = O(r^{-5/2})$$ as $r\to+\infty$ and $\alpha \to \alpha_0$. 
Furthermore, $$\int_0^{t(\alpha)} \sqrt{t}e^{-rt}dt \underset{r\to\infty \atop\alpha\to\alpha_0}{\sim} r^{-3/2}\int_0^{+\infty}\sqrt{t}e^{-t}dt = \frac{\sqrt \pi}{2r^{3/2}}.$$ 
Hence,
\begin{align}\label{gltt}
\frac{1}{2i\pi}\int_{\Omega_{-,\M + \delta} \cup \Omega_{+,\M + \delta}}&\Phi(x)e^{-r(\cos(\alpha)x+\sin(\alpha)Y^+(x))}
 dx\\
&\underset{r\to\infty \atop\alpha\to\alpha_0}{\sim} \frac{-ie^{-r(\cos(\alpha)x(\alpha_b) + \sin(\alpha)y(\alpha_b))}}{i\pi}\Phi_2(\M)|z'_t(0,\alpha_0)|^{3/2}\frac{\sqrt \pi}{2r^{3/2}}\\
&\underset{r\to\infty \atop\alpha\to\alpha_0}{\sim} C_{br}f_0(v_0)\frac{e^{-r(\cos(\alpha)x(\alpha_b) + \sin(\alpha)y(\alpha_b))}}{r^{3/2}\sin(\alpha_b - \alpha_0)^{3/2}}.
\end{align}
by \eqref{z'} and Lemma~\ref{tres_chiant}.

If $b_0 > 0$, the same argument applies, replacing $\Phi(x)$ with  $\phi^{v_0}(x)\frac{\gamma(x,Y^+(x),Z^+(x))}{\partial_y\gamma_+(x,Y^+(x))}$ which yields the result in \eqref{eq:ascas2}.
\end{proof}

To show that the asymptotics of the previous integrals yield the leading term in the asymptotics of the Green's function, we establish the following lemma.

\begin{lemma}[Negligibility of remaining integrals] \label{neglig2} Suppose $x_{max}^+ < x_{max}^-$ (i.e., $\alpha_b \in (0, \pi)$). Let $0 \leq \alpha_0 < \alpha_b$. Then, for $\delta > 0$ small enough, there exist constants $D > 0$, $\e' > 0$, $\eta > 0$, and $r_0 > 0$ such that for all $r \geq r_0$ and all $\alpha \geq 0$ satisfying $|\alpha - \alpha_0| < \eta$, the following estimates hold:
\begin{equation}\label{Neg'1}
\Big|\int_{R^\pm_{\delta}}\phi^{v_0}(x)\frac{\gamma(x,Y^+(x),Z^+(x))}{\partial_y\gamma_+(x,Y^+(x))}e^{-r(\cos(\alpha)x+\sin(\alpha)Y^+(x))}
 dx\Big| \leq De^{-r(\cos(\alpha)x(\alpha_b) + \sin(\alpha)y(\alpha_b) + \e')},
\end{equation}
and:
\begin{itemize}
    \item If $b_0 > 0,$ then \begin{equation}\label{Neg'3}I_2(r\cos(\alpha),r\sin(\alpha)) \leq De^{-r(\cos(\alpha)x(\alpha_b) + \sin(\alpha)y(\alpha_b) + \e')}
    \end{equation}
    \item If $b_0 \leq 0,$ then
    \begin{equation}\label{Neg'2}
\Big|\int_{R^\pm_{\delta}}\frac{e^{(a_0-a)x+b_0Z^+(x) - bY^+(x)}}{\partial_y\gamma_+(x,Y^+(x))}e^{-r(\cos(\alpha)x+\sin(\alpha)Y^+(x))}
 dx\Big| \leq De^{-r(\cos(\alpha)x(\alpha_b) + \sin(\alpha)y(\alpha_b) + \e')}
\end{equation}
\end{itemize}
\end{lemma}

\begin{proof}
We begin by proving \eqref{Neg'1}. The argument is analogous to that used in the proof of Lemma~\ref{lem:neglig1}.
Suppose first that $\alpha_0 > 0$. Then \eqref{Neg'1} can be rewritten as
\begin{equation}\label{nenene}\Big|\int_0^{+\infty}\phi^{v_0}(\M+\delta + is)\frac{\gamma(\M+\delta + is,Y^+(\M+\delta + is),Z^+(\M+\delta + is))}{\partial_y\gamma_+(\M+\delta + is,Y^+(\M+\delta + is))}e^{-r(F(\M, \alpha) - F(\M +\delta + is) - \e')}
 dx\Big| \leq D.
\end{equation}
We decompose the exponent as follows: $$F(\M, \alpha) - F(\M +\delta + is) - \e' = (F(\M, \alpha) - F(\M + \delta, \alpha) - \e') +(F(\M + \delta, \alpha) - F(\M +\delta + is)).$$
Now, observe that the function $x \mapsto F(x,\alpha_0)$ is strictly decreasing on $[-\pi, x(\alpha_0)]$ and strictly increasing on $[x(\alpha_0), x_{max}^+]$. Since $\alpha_0 < \alpha_b$, it follows that $x(\alpha_0) > \M$. Therefore, for some $\delta > 0$ and $\e' > 0$ small enough, there exists $\eta > 0$ such that $F(\M, \alpha) - F(\M + \delta, \alpha) - \e' \geq 0$ for all $\alpha \geq 0$ such that $|\alpha - \alpha_0|<\eta$. 
Moreover, using the same reasoning as in the proof of Lemma~\ref{lem:neglig1}, we find that there exists $c > 0$ such that for all $s \geq 0$,
\begin{equation}
    \label{c'v}
F(\M + \delta, \alpha) - \Re F(\M +\delta + is) \geq c \sin(\alpha) s.
\end{equation}
Then, by applying \eqref{zlm2}, the left-hand side of \eqref{nenene} is bounded by $\frac{C}{r \sin(\alpha)} \leq D$, for some $D > 0$ independent of $r$ and $\alpha \geq \alpha_0 - \eta > 0$ (for $\eta > 0$ sufficiently small). This proves \eqref{nenene} when $\alpha_0 > 0$. 

If $\alpha_0 = 0$, \eqref{Neg'1} can be obtained similarly by integration by parts, see \cite[Lemma 7.3]{Franceschi_2024} for similar considerations. The proof of \eqref{Neg'2} is symmetrical. 

We now turn to the proof of \eqref{Neg'3}. Suppose $b_0 \geq 0$ and consider the representation of $I_2^{v_0}(a,b)$ given by \eqref{I22''}. Applying the saddle-point method to \eqref{I22''} (see Lemmas~\ref{lem:contribsaddle}, \ref{lem:colalpha=0}, and \ref{lem:neglig1}), we obtain $$I_2^{v_0}(r\cos(\alpha), r\sin(\alpha)) = O\left(\frac{e^{-r(\cos(\alpha)x(\alpha) + \sin(\alpha)y(\alpha))}}{\sqrt r}\right) = o\left(e^{-r(\cos(\alpha)x(\alpha_b) + \sin(\alpha)y(\alpha_b) + \e')}\right)$$
as $r \to +\infty$ and $\alpha \to \alpha_0$, for $\e' > 0$ small enough by the definition of $(x(\alpha), y(\alpha))$. This proves \eqref{Neg'3} for some constant $D > 0$.
\end{proof}

\begin{proof}[Proof of Theorem~\ref{thm:4}]
If $b_0 \leq 0$, we consider the representations of $I_1^{v_0}(a,b)$ and $I_2^{v_0}(a,b)$ given by \eqref{I1111} and \eqref{I3333}, respectively, in Lemma~\ref{découpage2}. The contributions of these integrals along $\Omega_{\pm,\M+\delta}$ are exactly given by \eqref{eq:contribOMEGAxbdelta} from Lemma~\ref{lem:contribOMEGAxbdelta}. By Lemma~\ref{neglig2}, and more precisely by \eqref{Neg'1} and \eqref{Neg'2}, the asymptotic contributions of the integrals along $R^\pm_\delta$ in the representations \eqref{I1111} and \eqref{I3333} are negligible compared to those along $\Omega_{\pm,\M+\delta}$. It follows that the asymptotics of $g^{v_0}(r\cos(\alpha), r\sin(\alpha))$ are given by \eqref{eq:contribOMEGAxbdelta}, as claimed.

If $b_0 > 0$, we consider the representation of $I_1^{v_0}(a,b)$ given by \eqref{I1111}, and the inequality for $I_2^{v_0}(a,b)$ provided in \eqref{Neg'3}. The contribution of the integral \eqref{I1111} along $\Omega_{\pm,\M+\delta}$ is given by \eqref{eq:ascas2}. By \eqref{Neg'1}, the integrals along $R^\pm_\delta$ from the representation \eqref{I1111} are asymptotically negligible compared to \eqref{eq:ascas2}. Moreover, by \eqref{Neg'3}, the asymptotic contribution of $I_2^{v_0}(r\cos(\alpha), r\sin(\alpha))$ is also negligible compared to \eqref{eq:ascas2}. Therefore, the asymptotics of $g^{v_0}(r\cos(\alpha), r\sin(\alpha))$ are given by \eqref{eq:ascas2}, which completes the proof.
\end{proof}

\section{Proof of Theorem~\ref{thm5}: full and minimal Martin boundaries}\label{sec:5}
In this section, we prove the result concerning the Martin boundary using Theorems~\ref{alpha=0} to~\ref{thm:4}. In the general theory of Martin boundaries, Martin functions are usually \textit{excessive} (see \cite[Theorem~3]{kunitaWatanabe1965}). Nevertheless, for Markov chains with a finite number of steps from each state, one can immediately verify that Martin functions are harmonic. They are also harmonic for (drifted) Brownian motion in $\R^d$. In the next lemmas, we show that the functions $f_0$, $f_\pi$, and $h_\alpha$, $\alpha \in \mathcal M \setminus \{0,\pi\}$, in our model are harmonic as well. Due to the fundamental role of $v_0$ in this section, we write $g(v_0, v)$ for $g^{v_0}(v)$.

\begin{lemma}[Harmonicity]\label{cestbienharmonique}
The functions $f_0$, $f_\pi$, and $h_\alpha$, $\alpha \in \mathcal M\backslash\{0,\pi\}$, are harmonic.
\end{lemma}
\begin{proof}
Let $U$ be an open set that is relatively compact in $\R^2$. By \cite[Proposition~6.2]{kunitaWatanabe1965}, if $v \in \R^2$, then $g(\cdot,v)$ is harmonic on $\R^2 \setminus \{v\}$. Hence the Martin kernel $K(\cdot,v)$ defined by $K(v_0,v) = \frac{g(v_0,v)}{g(0,v)}$ is harmonic as well on $\R^2 \setminus \{v\}$. Therefore, if $v \notin U$ and $v_0 \in \R^2$ satisfy $v \neq v_0$, then 
\begin{equation}\label{eq:Kharmonique}
K(v_0,v) = \E_{v_0}\left[K(V_{\tau_{U^c}}, v)\right].
\end{equation} 
Let $\alpha_0 \in \mathcal M$. We consider $v = (r\cos(\alpha), r\sin(\alpha))$ with $\alpha \in \mathcal M$, and let $r \to \infty$ and $\alpha \to \alpha_0$.
By Theorems~\ref{thm:1} to \ref{thm3}, we have 
$$K(v_0,(r\cos(\alpha), r\sin(\alpha))) = \frac{h_{\alpha_0}(v_0) + \varepsilon(v_0, r, \alpha)}{h_{\alpha_0}(0) + \varepsilon(0, r, \alpha)}$$
where $\varepsilon(v_0, r, \alpha) \underset{r\to+\infty\atop\alpha\to\alpha_0}{\longrightarrow} 0$ (and where, for consistency of notation, we have set $h_0(v_0) = f_0(v_0)$ and $h_\pi(v_0) = f_\pi(v_0)$).
If $\varepsilon(v_0, r, \alpha)$ converges to zero 
as $r\to+\infty,\alpha\to\alpha_0$ {\it uniformly in $v_0 \in U$}, the dominated convergence theorem applied to \eqref{eq:Kharmonique} yields 
$$h_{\alpha_0}(v_0) = \E_{v_0}\left[h_{\alpha_0}(V_{\tau_{U^c}})\right]$$ and completes the proof. The uniform convergence in $v_0\in U$ can be established using arguments similar to those in \cite[Lemmas 5.3 and Section 6]{petit2024}. The only directions specific to our model to which these arguments do not apply directly are $\alpha_b$ and $\tilde\alpha_b$ corresponding to branching points, cf. \eqref{def:xb}, \eqref{def:alphabb}, and \eqref{def:alphabbtilde}.
It remains to establish uniform convergence for these directions.
Let us consider case~(i) of Theorem~\ref{thm3}. We prove that the equivalence \eqref{jujujuju} is uniform in $v_0 \in U$ and that the constant $M$ in \eqref{ineq_M_indepz0} can be chosen independently of $v_0 \in U$.  Recall the definitions \eqref{Phi1Phi2} of $\Phi_1$ and $\Phi_2$, which we denote by $\Phi_1^{v_0}$ and $\Phi_2^{v_0}$. For \eqref{jujujuju}, it suffices (see the proof of \cite[Lemma~8.1]{Franceschi_2024}) to show that  
\begin{equation}\label{casuffitpourbornesunif}
\sup_{v_0 \in U\atop x\in K}|\Phi_1^{v_0}(x)| <+\infty
\end{equation}
for some compact neighborhood $K \subset \C$ of $x_b$. For \eqref{ineq_M_indepz0}, it suffices to prove the analogue of \eqref{casuffitpourbornesunif} with $\Phi_2^{v_0}$ instead of $\Phi_1^{v_0}$. To do this, we use the explicit expression \eqref{DLL} of $\Phi^{v_0}(x)$. Note that $\Phi^{v_0}(x)$ can be written as $(\Psi_1(x,v_0) + \Psi_2(x,v_0)\sqrt{x_b-x})e^{f(x)\sqrt{\M - x}b_0\fc_{b_0<0}}$ where $\Psi_i(x,v_0), i=1,2$ are locally bounded in $(x,v_0)$ and both $\Psi_i(x,v_0), i=1,2$ and $f(x)$ are holomorphic in $x$ in a neighborhood of $x_b$. Expanding the exponential into a Taylor series gives
$$e^{f(x)\sqrt{\M - x}b_0\fc_{b_0<0}} = \underbrace{\sum_{k=0}^{+\infty} \frac{f(x)^{2k}}{(2k)!}(\M-x)^k b_0^{2k}\fc_{b_0<0}}_{=\Xi_1(x,b_0)}\;  +\; \sqrt{\M-x}\underbrace{\sum_{k=0}^{+\infty} \frac{f(x)^{2k+1}}{(2k+1)!}(\M-x)^k b_0^{2k+1}\fc_{b_0<0}}_{=\Xi_2(x,b_0)}.
 $$ 
A straightforward argument shows that $\Xi_1(x,b_0)$ and $\Xi_2(x,b_0)$ are bounded when $x$ lies in a sufficiently small neighborhood of $x_b$ and $b_0$ ranges over a compact set. Since $\Phi_1(x,v_0)=\Psi_1(x,v_0)\,\Xi_1(x,b_0) + (x_b-x)\Psi_2(x,v_0)\,\Xi_2(x,b_0)$ and $\Phi_2(x,v_0)=\Psi_1(x,v_0)\,\Xi_2(x,b_0) + \Psi_2(x,v_0)\,\Xi_1(x,b_0)$, this yields the desired uniform bounds.
\end{proof}
%

To prove Theorem~\ref{thm5}, we state the following technical lemma.

\begin{lemma}[Asymptotic inequalities for $h_\alpha, f_0, f_\pi$]\label{lem:ineqmini}
Let $\alpha, \beta \in \mathcal M$ with $\alpha \neq \beta$. Then: \begin{itemize}
    \item Suppose that $\beta\in\mathcal M \cap (0, \pi)$. Then there exist $\e > 0$, $\eta > 0$, $r_0 > 0$, and $C > 0$ such that for all $\gamma \in \mathcal M \cap (\beta - \e, \beta + \e)$ and all $r \geq r_0$,
    \begin{equation}\label{ineq:hgamma1}
        h_\gamma(r\cos(\beta), r\sin(\beta)) \geq C^{-1}e^{r(x(\beta)\cos(\alpha) + y(\beta)\sin(\beta) - \eta)} 
    \end{equation}
    and such that for all $r\geq r_0$,
    \begin{equation}\label{ineq:halpha1}
        h_\alpha(r\cos(\beta), r\sin(\beta)) \leq Ce^{r(x(\beta)\cos(\alpha) + y(\beta)\sin(\beta) - 2\eta)}.
    \end{equation}
    \item Suppose that $\beta \in \mathcal M \cap (\pi, 2\pi)$. Then there exist $\e > 0$, $\eta > 0$, $r_0 > 0$, and $C > 0$ such that for all $\gamma \in \mathcal M \cap (\beta - \e, \beta + \e)$ and all $r \geq r_0$,
    \begin{equation}\label{ineq:hgamma2}
        h_\gamma(r\cos(\beta), r\sin(\beta)) \geq C^{-1}e^{r(x(\beta)\cos(\alpha) + z(\beta)\sin(\beta) - \eta)} 
    \end{equation}
    and such that for all $r\geq r_0$
    \begin{equation}\label{ineq:halpha2}
        h_\alpha(r\cos(\beta), r\sin(\beta)) \leq Ce^{r(x(\beta)\cos(\alpha) + z(\beta)\sin(\beta) - 2\eta)}.
    \end{equation}
\end{itemize}
Inequalities \eqref{ineq:halpha1} and \eqref{ineq:halpha2} also hold if $\alpha$ is replaced by $0$ (resp. $\pi$) and $h_\alpha$ by $f_0$ (resp. $f_\pi$).
\end{lemma}
\begin{proof}
This is a direct consequence of expressions \eqref{f(v_0)branch} and \eqref{Harm1}.
\end{proof}

\begin{proof}[Proof of Theorem~\ref{thm5}] By Theorems~\ref{alpha=0} to~\ref{thm:4} and Lemma~\ref{cestbienharmonique}, it remains to show that functions $h_\alpha$ and $f_0, f_\pi$ are minimal. The main idea of the proof is that $h_\beta(r\cos(\beta), r\sin(\beta))$ is asymptotically of greater order than  $h_\alpha(r\cos(\beta), r\sin(\beta))$ as $r\to+\infty$ for any
fixed $\beta$ such that $\beta \neq \alpha$.

Let $\alpha \in \mathcal M\backslash\{0,\pi\}$. We now rigorously prove that $h_\alpha$ is minimal. To do so, we show that if $c_0, c_\pi \geq 0$ and if $\mu$ is a Borel measure on $\mathcal M \backslash \{0,\pi\}$ such that
\begin{equation}\label{eq:represlemme}
    h_\alpha = \int_{\mathcal M\backslash\{0,\pi\}} h_\alpha \, d\mu(\alpha) + c_0 f_0 + c_\pi f_\pi,
\end{equation}
then $c_0 = c_\pi = 0$ and $\mu = \delta_\alpha$ (we excluded $0$ and $\pi$ in the integral term to avoid problems of notation). Let $\beta \in \mathcal M \cap (0, \pi)$ with $\alpha \neq \beta$, and let $\e, \eta, C, r_0 > 0$ be the constants provided by Lemma~\ref{lem:ineqmini}. Then, using inequalities \eqref{ineq:hgamma1}, \eqref{ineq:halpha1}, and the representation \eqref{eq:represlemme}, for all $r \geq r_0$ we have
\begin{align*}
    Ce^{r(x(\beta)\cos(\alpha) + y(\beta)\sin(\beta) - 2\eta)} &\geq h_\alpha(r\cos(\beta), r\sin(\beta))\\
    &\geq \int_{ (\beta-\e, \beta+\e)\cap\mathcal M}h_\beta(r\cos(\gamma), r\sin(\gamma))d\mu(\gamma)\\
    &\geq C^{-1}\mu( (\beta-\e, \beta+\e)\cap\mathcal M)e^{r(x(\beta)\cos(\alpha) + y(\beta)\sin(\beta) - \eta)}.
\end{align*}
Hence, for $r\geq r_0,$
$$C^2e^{-\eta r} \geq \mu( (\beta-\e, \beta+\e)\cap\mathcal M).$$
Taking the limit as $r \to +\infty$, we obtain $\mu((\beta - \e, \beta + \e) \cap \mathcal M) = 0$. By a standard property of Borel measures, it follows that $\mu(\mathcal M\cap (0,\pi) \setminus \{\alpha\}) = 0$. Using symmetric arguments along with inequalities \eqref{ineq:hgamma2} and \eqref{ineq:halpha2}, we similarly get $\mu(\mathcal M \cap (\pi,2\pi) \setminus \{\alpha\}) = 0$. Hence, $\mu = c\delta_\alpha$, where $\delta_\alpha$ denotes the Dirac measure at $\alpha$. By \eqref{eq:represlemme}, we get $h_\alpha = ch_\alpha + c_0 f_0 + c_\pi f_\pi$. Since $h_\alpha$, $f_0$, and $f_\pi$ are clearly linearly independent, it follows that $c_0 = c_\pi = 0$ and $c = 1$. The proofs of the minimality of $f_0$ and $f_\pi$ are analogous.
\end{proof}

Let us now prove Corollary~\ref{cor:escape+-infty}.
\begin{proof}[Proof of Corollary~\ref{cor:escape+-infty}]
By standard arguments, the mapping
\[
h : (a_0,b_0) \longmapsto \P_{(a_0,b_0)}\bigl(B_t \underset{t\to+\infty}{\longrightarrow} +\infty\bigr) 
\]
is harmonic. Moreover, note that $h_{\alpha_{\mu^+}}$ and $h_{\alpha_{\mu^-}}$ are the only bounded harmonic functions among $f_0$, $f_\pi$, and $h_\alpha$ for $\alpha \in \mathcal M \backslash\{0, \pi\}$ (see \eqref{f(v_0)col}, \eqref{f(v_0)branch}, \eqref{Harm1}, and \eqref{Harm1'}). Hence, applying the representation \eqref{eq:representation} to the bounded harmonic function $h$ yields $h = c_+ h_{\alpha_{\mu^+}} + c_- h_{\alpha_{\mu^-}}$ for some nonnegative constants $c_+$ and $c_-$. By properties of standard Brownian motion (denoted here by $W$), we have
$$h(a_0, b_0) \geq \P\left[\left\{b_0 +\sqrt{\Sigma_{22}}W_t + \mu_2^+t\underset{t\to\infty}{\longrightarrow} +\infty\right\} \cap \left\{b_0 + \sqrt{\Sigma_{22}}W_t + \mu_+t > 0\;\;\; \forall t \geq 0\right\}\right] \underset{b_0\to+\infty}\longrightarrow 1$$
by Hypothesis~\eqref{hyp:drift} on the drift.
Since $h_{\alpha_{\mu^-}}(a_0, b_0) \underset{b_0\to+\infty}\longrightarrow 0$ and $h_{\alpha_{\mu^+}}(a_0, b_0)\underset{b_0\to+\infty}\longrightarrow  1$ (see \eqref{Harm1} and \eqref{Harm1'}), we obtain $c_+ = 1$. Considering now the limit as $b_0 \to -\infty$, we deduce that $c_- = 0$, and the desired result follows.
\end{proof}
%
%


\section{Generalization to skew diffusions acting in any oblique direction \nr{$q$}}\label{sec:skewdiff}
\nb{The goal of this section is to study the extension of our results to general oblique direction $q \in \R \times (-1,1)$, and, in particular, to exhibit new types of asymptotic behavior that are absent in the case of divergence form.} \nr{Up to this point, we have imposed the condition $q = q_0$ (see \eqref{eq:q}) to have a divergence form generator. The goal of this section is to study the extension of our results for $q \in \R \times (-1,1)$.}

In Section~\ref{sub:conjectures}, we state  
hypotheses~\ref{conj1} and \ref{aronsonlike} on $q$ under which this analysis can be fruitful. Section~\ref{sub:analogous} describes new properties satisfied by the Laplace transforms. In particular, the Laplace transform $\phi^{v_0}$ may now have a pole. In Section~\ref{sub:discussion}, we introduce a new angular region (arising from the pole of $\phi^{v_0}$) and generalise Theorems~\ref{thm:1} -- \ref{thm:4} to this broader range of parameters $q$. Sections~\ref{sub:asafterpole} and~\ref{sub:as:pole} are devoted to the study of the asymptotic behaviour of the Green's functions $g^{v_0}$ within this new region. Finally, in Section~\ref{sub:martingeneral}, we identify the Martin boundary and the minimal Martin boundary for the generalised processes.

Let $\Sigma^+, \Sigma^-$ be covariance matrices, and let $\mu^+, \mu^- \in \R^2$ satisfy \eqref{hyp:drift}. Let $q \in \R \times (-1,1)$. Using the same method as in Section~\ref{sub:existence}, one can show that for every $(a_0,b_0) \in \R^2$, there exists a unique pathwise solution $V = (A,B)$ to equation \eqref{EDS}, driven by a Brownian motion $(W_t)_{t \geq 0}$.

\subsection{Hypotheses}\label{sub:conjectures}
We now state two technical hypotheses on the parameter $q \in \R \times (-1,1)$ under which our extended results hold. 

\begin{hypothesis}[Markov property and density]\label{conj1}
The process $V$ defines a strong Markov process. Moreover, for any $v_0 \in \R^2$ and any $t > 0$, the law of $V_t$ starting from $v_0$ admits a density $p_t^{v_0}(v)$ with respect to the Lebesgue measure on $\R^2$.
Finally, the function $v\longmapsto p_t^{v_0}(v)$ has limits at the top and bottom of the axis $\{y=0\}$.
\end{hypothesis}

\begin{hypothesis}[Aronson-like inequality]\label{aronsonlike}
There exists a constant $M > 0$ such that for all $t > 0$ and all $v_0 \in \R^2$,
\begin{equation}
p_t^{v_0}(v) \leq \frac{M}{t}e^{-|v-v_0|^2/Mt + Mt}.
\end{equation}
\end{hypothesis}

Assuming Hypothesis~\ref{conj1}, we may define the Green's functions by $g^{v_0}(v) = \int_0^{+\infty} p_t^{v_0}(v) \, dt$ for $v = (a, b) \in \{b \neq 0\}$ and $v_0 \in \R^2$. Under Hypothesis~\ref{aronsonlike}, inequality~\eqref{utile!2} holds, and the functions $g^{v_0}$ satisfy inequalities~\eqref{utile!}.

\subsection{Analogous results for Laplace transforms.}\label{sub:analogous}
In this section, we investigate properties of the Laplace transforms for a general parameter $q$. The functional equation \eqref{def:phi} and the continuation formula for $\phi^{v_0}$ given in \eqref{laplace explicitee} remain valid and can be proved using the same arguments.

\begin{proposition}[Functional equation and continuation of $\varphi^{v_0}$] \nr{There exists $\eta > 0$ such that for all $x \in (-\eta, 0)$,} \nb{For all $x \in i\R$,} $Re(y) < 0$, and $Re(z) > 0$, the functions $\phi^{v_0}_-(x, z)$, $\phi^{v_0}_+(x, y)$, and $\phi^{v_0}(x)$ are finite and satisfy \eqref{eq fonctionnelle}. Furthermore, $\phi^{v_0}$ admits a meromorphic continuation on $\C \setminus \big((-\infty, \m] \cup [\M, +\infty)\big)$ given by \eqref{laplace explicitee}.
\end{proposition}

The main difference with the case $q = q_0$ is that $\phi^{v_0}$ may now have a pole.

\begin{proposition}[Poles of $\phi^{v_0}$ on $(\m,\M)$] On $[\m,\M]$, the equation $(E) : \gamma(x,Y^-(x),Z^+(x)) = 0$ has at most one solution, denoted by $x^*$. If this equation has no solution, we adopt the convention $x^* = +\infty$. Then, on $(\m,\M)$, the extended function $\phi^{v_0}$ has a pole if and only if $(E)$ has a solution, and the pole is located at $x^*$ whenever it exists. In that case, $(x^*,\, Y^-(x^*),\, Z^+(x^*))$ solves the system
\begin{equation} \label{syst_pole}
\begin{cases}   \gamma(x,y,z) = 0 \\
\gamma_+(x,y) = 0\\
\gamma_-(x,z) = 0
\end{cases}
\end{equation}
and $$res_{x=x^*}\phi^{v_0}(x) = -\frac{e^{x^*a_0 + Y^-(x^*)b_0\fc_{b_0> 0} +  Z^+(x^*)b_0\fc_{b_0<0}}}{\frac{d}{dx}\gamma(x,Y^-(x), Z^+(x))\big|_{x=x^*}}.$$ Furthermore, if $x^* > 0$ (resp. $x^* < 0$), then $x^*$ is the unique pole of $\phi^{v_0}$ in the strip $\{\m - \e <\Re(x) <x^* + \e \}$ (resp. $\{x^* - \e <\Re(x) <\M + \e \}$) for sufficiently small $\e > 0.$
\end{proposition}

\begin{proof}
From \eqref{laplace explicitee}, it is clear that $x$ is a pole of $\phi^{v_0}$ if and only if $(x, Y^-(x), Z^+(x))$ satisfies \eqref{syst_pole}. Note that the function $\gamma(x, Y^-(x), Z^+(x))$ is convex, being the sum of convex functions. Moreover, its explicit expression shows that at least one of the inequalities $\gamma(x, Y^-(x), Z^+(x))|_{x=\m} < 0$ or $\gamma(x, Y^-(x), Z^+(x))|_{x=\M} < 0$ must hold. Hence, $\gamma(x, Y^-(x), Z^+(x))$ has at most one zero on $(\m, \M)$, and $\phi^{v_0}$ has at most one pole in this interval. The last statement of the lemma follows from a straightforward analysis of $\gamma(x, Y^-(x), Z^+(x))$.
\end{proof}

Note that $x^* \neq 0$ since $\gamma(0, Y^-(0), Z^+(0)) = \left(\frac{1+q_2}{2}\right) Y^-(0) + \left(\frac{q_2 - 1}{2}\right) Z^+(0) < 0$.

\begin{rem}[Values of Green's functions at the boundary] 
In the case $q = q_0$, the Green's function is continuous on the axis, meaning that the upper and lower limits $g^{v_0}(a, 0^+)$ and $g^{v_0}(a, 0^-)$ are equal. When $q\neq q_0$, this is no longer the case.
Using the initial value theorem for Laplace transforms together with \eqref{eq fonctionnelle}, we obtain for all $u \in \R$ with $v_0 \neq (a, 0),$
\begin{equation}
\Sigma_{22}^+\left(\frac{1-q_2}{2}\right)g^{v_0}(a,0^+) = \Sigma_{22}^-\left(\frac{1+q_2}{2}\right)g^{v_0}(a,0^-).
\end{equation}
By the same argument as in the proof of Lemma~\ref{Tlocal_densite}, we also have
\begin{equation}
\phi^{v_0}(x) = \int_\R \frac{\Sigma_{22}^-g^{v_0}(a,0^-) + \Sigma_{22}^+g^{v_0}(a,0^+)}{2}e^{-xu}du
\end{equation}
which can be interpreted as the Laplace transform of the Green's functions on the axis.
In fact, if we adopt the convention $g^{v_0}(a,0) = \frac{\Sigma_{22}^-g^{v_0}(a,0^-) + \Sigma_{22}^+g^{v_0}(a,0^+)}{\Sigma_{22}^- + \Sigma_{22}^+}$, then $\phi^{v_0}$ is the Laplace transform of $\frac{\Sigma_{22}^- + \Sigma_{22}^+}{2} g^{v_0}(a, 0)$ exactly as in the case $q = q_0$. Theorem~\ref{alpha=0} for $\alpha = \alpha_0 \in \{0, \pi\}$ could be extended to the case of any $q\in \R \times (-1,1)$ in this sense. 
\end{rem}

\subsection{\nb{Introduction of a new angular region}\nr{Generalisation of Theorems~\ref{thm:1} to \ref{thm:4} and introduction of a new region for angles}}\label{sub:discussion}
\subsubsection{Case without pole}\label{subsub:touttout}If $x^* = +\infty$, then all the arguments remain unchanged, and Theorems~\ref{thm:1} to \ref{thm:4} hold for general $q \in \R \times (-1, 1)$.

\subsubsection{Notation if $\phi^{v_0}$ has a pole} Now suppose that $\phi^{v_0}$ has a pole $x^*$, and let us introduce a new region of angles for the asymptotic analysis. For simplicity, we assume that $x^* \neq \m$ and $x^* \neq \M$.

We recall the notation $\alpha_{\mu^+} = \arctan(\mu_2^+/\mu_1^+) \in (0, \pi)$ and $\alpha_{\mu^-} = \arctan(\mu_2^-/\mu_1^-) + 2\pi \in (3\pi/2, 2\pi)$. Note that $x(\alpha) > 0$ for $\alpha \in (0, \alpha_{\mu^+}) \cup (\alpha_{\mu^-}, 2\pi)$, and that $x(\alpha) < 0$ for $\alpha \in (\alpha_{\mu^+}, \pi) \cup (\pi, \alpha_{\mu^-})$ (see Section~\ref{sub:theorems}). Let $\alpha^*_+ \in (0, \pi)$ and $\alpha^*_- \in (\pi, 2\pi)$ be the angles defined by
\begin{equation}
    (x(\alpha^*_+), y(\alpha^*_+)) = (x^*, Y^+(x^*)), \quad (x(\alpha^*_-), z(\alpha^*_-)) = (x^*, Z^-(x^*)).
\end{equation}
(see Figure~\ref{fig:alphaxetoile}). 
\begin{figure}[t]
    \centering
    \includegraphics[scale=0.7]{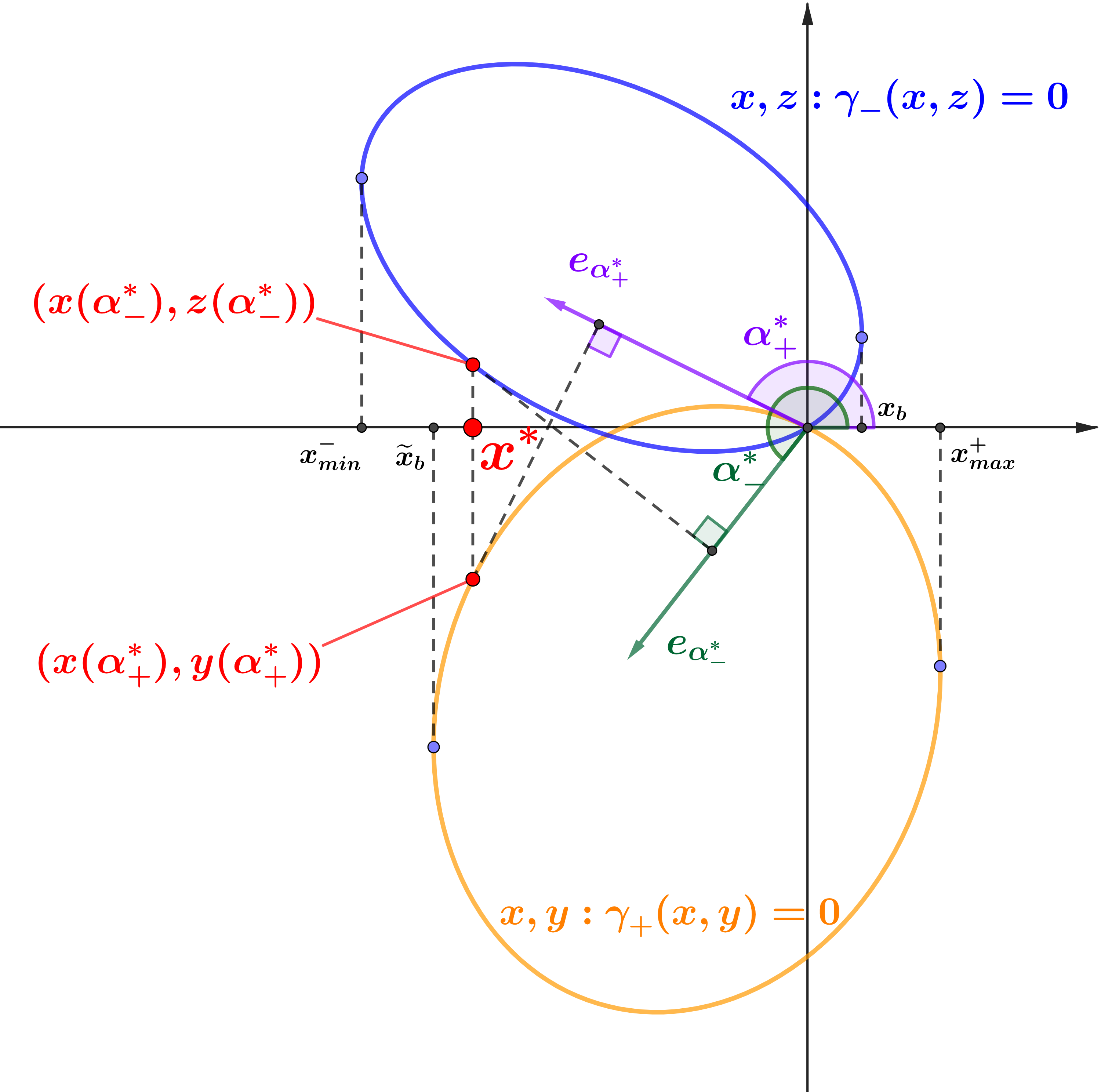}
    \caption{Definition of $\alpha^*_\pm$ from the data $x^*.$ In this case, $x^* \in (\m,0).$}
    \label{fig:alphaxetoile}
\end{figure}
From the above remarks on the sign of $x(\alpha)$, we observe that:
\begin{itemize}
    \item If $x^* \in (0,\M),$ then $\alpha^*_+ \in (0, \alpha_{\mu^+})$ and $\alpha^*_- \in (\alpha_{\mu^-},2\pi)$.
    \item If $x^* \in (\m,0),$ then $\alpha^*_+ \in (\alpha_{\mu^+},\pi)$ and $\alpha^*_- \in (\pi,\alpha_{\mu^-})$.
\end{itemize}
This defines a new (pink) region of angles; see Figure~\ref{fig:soleil3} for the case $x^* \in (0, \M)$.

\subsubsection{Generalisation of Theorems~\ref{thm:1} to \ref{thm:4} when $\phi^{v_0}$ has a pole}\label{subsub:tout}Outside the pink region, the same method as in the case $q = q_0$ applies. In these areas, the contour deformations do not involve the pole of $\phi^{v_0}$, so the asymptotics in Theorems~\ref{thm:1} to \ref{thm:4} remain valid in the corresponding non-pink regions (see Figure~\ref{fig:soleil3} for the case $x^* > 0$ and $\m = x_{min}^+$).

\subsection{Study of asymptotics (strictly) inside the dotted pink region}\label{sub:asafterpole}
Without loss of generality, we focus on the case where $x^* \in (0, \M)$. In this section, we study the asymptotics of the Green's functions $g^{v_0}$ in the directions $\alpha_0 \in (\alpha_-^*, 2\pi) \cup [0, \alpha_+^*)$.

Note that Lemma~\ref{double_simple}, and the constructions in Section~\ref{sub:pathsaddlepoint} still hold, namely that $g^{v_0} = I_1^{v_0} + I_2^{v_0}$.


\begin{lemma}[Changing path for directions $\alpha_0 \in [0,\alpha^*_+)$]\label{lem:direcccccc}
Suppose $x^* \in (0, \M)$. Let $x^* < x_c < \M$. Then, for $a \in \R$, $b > 0$, and $v_0 \in \R^2$ with $(a,b) \neq v_0$, we have:
\begin{equation}\label{eq:residuI1}
    I_1^{v_0}(a,b)  = res_{x=x^*}\phi^{v_0}(x)\frac{\gamma(x^*,Y^+(x^*),Z^+(x^*))}{\partial_y\gamma_+(x^*,Y^+(x^*))}e^{-ax^*-bY^+(x^*)}. 
\end{equation}
$$+\,\frac{1}{2i\pi}\int_{x_c + i\R}\frac{\phi^{v_0}(x)\gamma(x,Y^+(x),Z^+(x))e^{-ax-bY^+(x)}}{\partial_y\gamma_+(x,Y^+(x))} dx$$
where $\gamma(x^*,Y^+(x^*),Z^+(x^*)) \neq 0.$
Furthermore,
\begin{equation}\label{I2pole}
I_2^{v_0}(a,b) = \frac{1}{2i\pi}\int_{x_c + i\R}\frac{e^{(a_0-a)x}\left(
e^{(b_0-b)Y^+(x)}\fc_{b_0>0} 
+ e^{b_0Z^+(x) - bY^+(x)}\fc_{b_0 < 0}
\right)}{\partial_y\gamma_+(x,Y^+(x))} dx \quad \textnormal{if}\; b > b_0
\end{equation}
 and
\begin{equation}\label{I2popole}
I_2^{v_0}(a,b) = \frac{1}{2i\pi}\int_{Y^+(x_c) + i\R}\frac{e^{(a_0-a)X^+(y) + (b_0-b)y}}{\partial_x\gamma_+(X^+(y),y)} dy \quad \textnormal{if}\; a > a_0.
\end{equation}
\end{lemma}

\begin{proof}
The deformation of the integration path is justified using the same arguments as in Lemma~\ref{découpage2}. Equation $\gamma(x^*,Y^+(x^*),Z^+(x^*)) \neq 0$ follows from the fact that it is not possible to have both $\gamma(x,Y^+(x),Z^+(x)) = 0$ and $\gamma(x,Y^-(x),Z^+(x)) = 0$ simultaneously, except possibly at $x = x_{min}^+$ and $x = x_{max}^+$.
\end{proof}

The term $I_1^{v_0}$ is decomposed into the sum of two terms, namely the residue at the pole and a remaining integral term.  
In the next lemma, we show that this integral term and also $I_2(a,b)$ are negligible in the asymptotics compared with the residue at the pole.

\begin{lemma}[Negligibility of integrals compared to the pole]\label{lem:neglig12}
Suppose $x^* \in (0, \M)$. Let $v_0 \in \R^2$, $\alpha_0 \in [0,\alpha^*_+)$, and let $x_c$ be such that $x^* < x_c < \min(x(\alpha_0),\M)$. Let $r = \sqrt{a^2 + b^2}$ and let $\alpha(a,b) \in (0,\pi)$ denote the angle satisfying $\cos(\alpha) = \frac{a}{\sqrt{a^2 + b^2}}$ and $\sin(\alpha) = \frac{b}{\sqrt{a^2 + b^2}}$. Then, for $\eta > 0$ small enough, there exist constants $r_0 > 0$ and $D > 0$ such that for all $(a,b)$ satisfying $r > r_0$, $\alpha(a,b) > 0$, and $|\alpha(a,b) - \alpha_0| < \eta$, we have:
$$\Big|\frac{1}{2i\pi}\int_{x_c + i\R}\frac{\phi^{v_0}(x)\gamma(x,Y^+(x),Z^+(x))e^{-ax-bY^+(x)}}{\partial_y\gamma_+(x,Y^+(x))} dx\Big| + |I_2(a,b)| \leq D e^{-r(\cos(\alpha)x_c + \sin(\alpha)Y^+(x_c))}$$
\end{lemma}

\begin{proof}
The proof is identical to that of Lemma~\ref{neglig2}.
\end{proof}

We now prove the corresponding theorem for the dotted pink region (see Figure~\ref{fig:soleil3}).

\begin{theo}[Asymptotics strictly inside the dotted pink region]\label{thm:7}
 Suppose that $x^* \in (0, \M)$. Let $v_0 \in \R^2$ and $\alpha_0 \in [0,\alpha^*_+)$. Then: 
\begin{equation}
    g^{v_0}(r\cos(\alpha), r\sin(\alpha)) \underset{r\to+\infty\atop \alpha\to \alpha_0, \alpha > 0}{\sim} C^+_{*}f_{*}(v_0) e^{-r(\cos(\alpha)x^* + \sin(\alpha)Y^+(x^*))}
\end{equation}
where \begin{equation}
    f_{*}(a_0,b_0) = \begin{cases}
e^{x^*a_0 + Y^-(x^*)b_0} & \text{if } b_0 \geq 0 \\
e^{x^*a_0 + Z^+(x^*)b_0} & \text{if } b_0 < 0.
\end{cases}
\end{equation}
and 
\begin{equation}\label{eq:C+*}
    C^+_{*} = \frac{\gamma(x^*,Y^+(x^*),Z^+(x^*))}{\partial_y\gamma_+(x^*,Y^+(x^*))\frac{d}{dx}\gamma(x,Y^-(x), Z^+(x))\big|_{x=x^*}}.
\end{equation}
The symmetric result holds for angles $\alpha_0 \in (\alpha^*_-, 2\pi].$
\end{theo}
\begin{proof}
Choose $x^* < x_c < \min(x(\alpha_0),\M)$ and consider the representation $g^{v_0} = I_1^{v_0} + I_2^{v_0}$ from Lemma~\ref{lem:direcccccc}. By the choice of $x_c$, we have $\cos(\alpha_0)x^* + \sin(\alpha_0)Y^+(x^*) < \cos(\alpha_0)x_c + \sin(\alpha_0)Y^+(x_c)$. Then, by Lemma~\ref{lem:neglig12}, both the integral term in \eqref{eq:residuI1} and the integral in \eqref{I2pole} are negligible compared with the residue as $\alpha \to \alpha_0$ and $r \to +\infty$. The result follows immediately.
\end{proof}
\begin{figure}[t]
    \centering
    \includegraphics[width=0.8\linewidth]{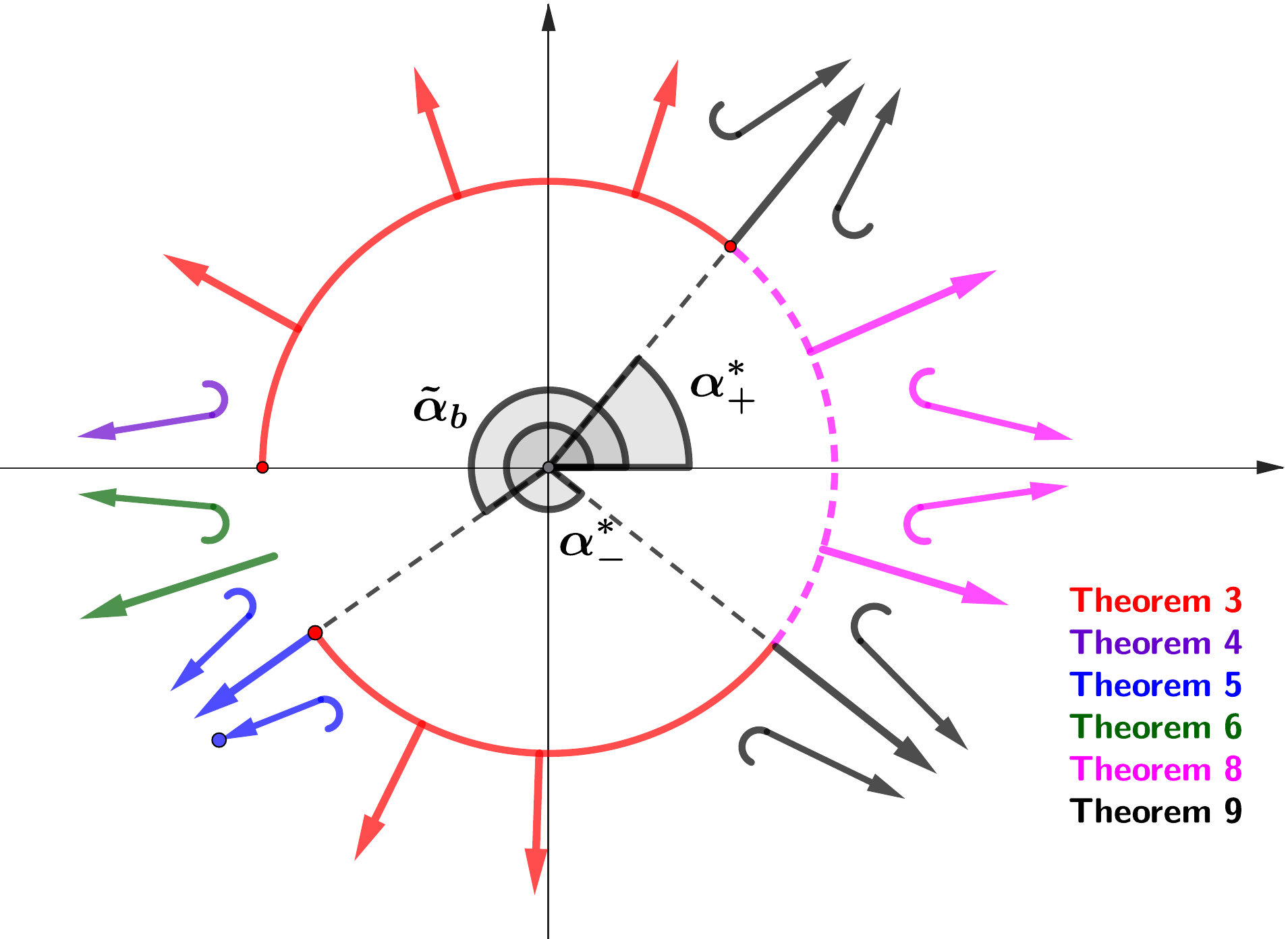}
    \caption{Summary of asymptotics for the extended results. In this case, $\m = x_{min}^+$ and $x^* \in (0,x_b)$. We omit $\alpha_b$ for readability.}
    \label{fig:soleil3}
\end{figure}

\begin{figure}[t]
    \centering
    \includegraphics[width=0.6\linewidth]{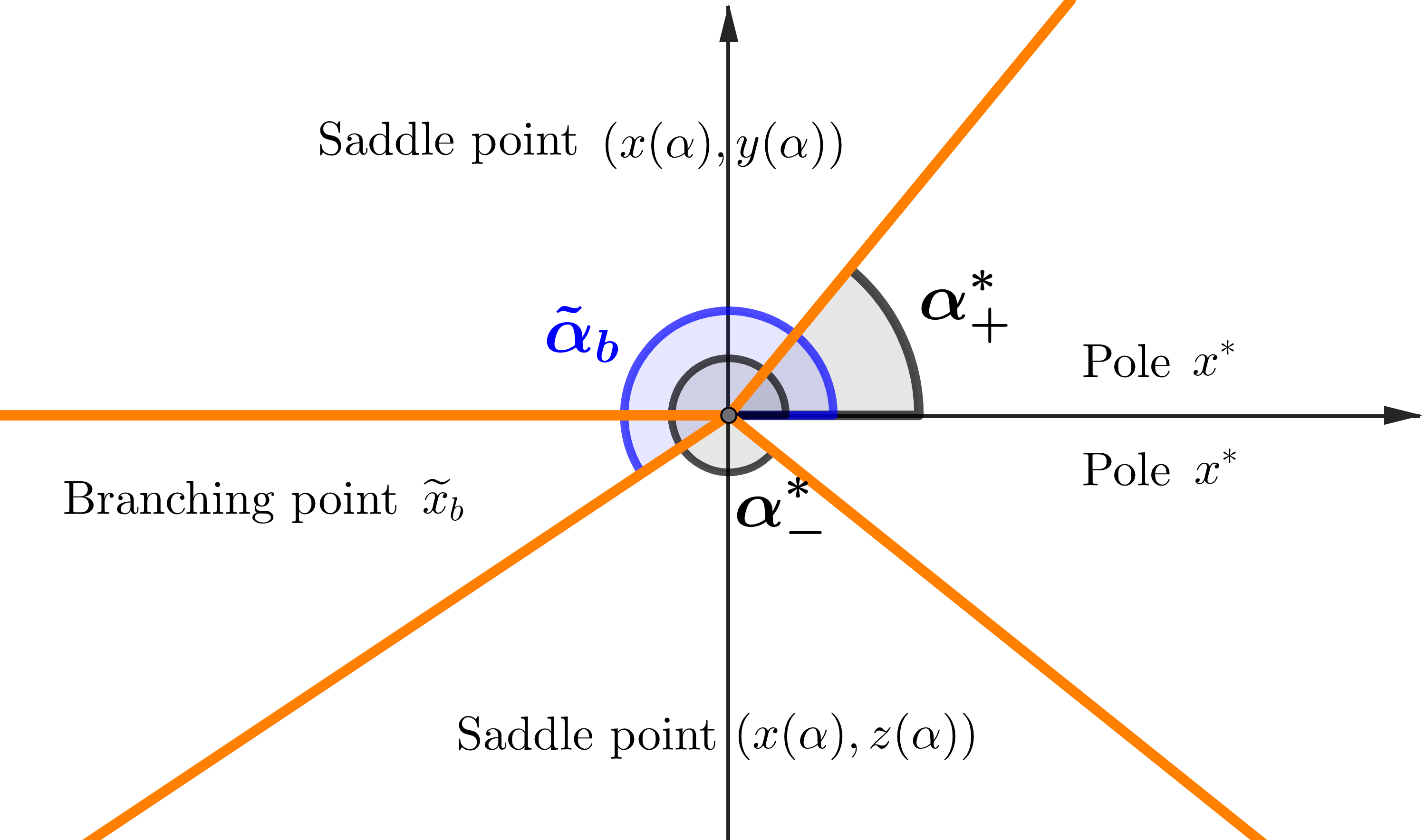}
    \caption{Type of asymptotics according to directions: saddle point type, pole or branching point.}
    \label{fig:typepasymptotique2}
\end{figure}

\subsection{Asymptotics in directions $\alpha^*_\pm$}\label{sub:as:pole}
We now consider the asymptotics in the directions $\alpha^*_\pm$ when $x^* \neq +\infty$. For simplicity, we assume that $x^* \in (0,\M)$. This “pole-type” asymptotics has been studied in detail in~\cite[Section 10]{Franceschi_2024}, and the same arguments apply here. Therefore, we state the following theorem without providing a proof. Let $\textnormal{erf}$ denote the  function defined by
$\textnormal{erf}(x) = \frac{2}{\sqrt{\pi}} \int_0^x \exp(-s^2)ds$.

\begin{theo}[Asymptotics as $\alpha \to \alpha^*_\pm$]\label{thm:8}
Suppose that $x^* \in (0, \M).$
Then, the following asymptotics hold: 
\begin{itemize} 

\item[$(i)$] If $\alpha > \alpha^*_+$ and $r(\alpha -\alpha^*_+)^2 \to \infty$, then
\begin{equation}
     g^{v_0}(r\cos(\alpha), r\sin(\alpha)) \underset{r\to+\infty\atop \alpha\to \alpha^*_+}{\sim}C_*'f_*(v_0) \frac{e^{-r(\cos(\alpha)x(\alpha) + \sin(\alpha)y(\alpha))}}{\sqrt{r}  (x(\alpha)-x^*)}
\end{equation}
where $$C_*'= -C^+(\alpha^*_+)\frac{\gamma(x^*,Y^+(x^*),Z^+(x^*))}
{\frac{d}{dx}\gamma(x,Y^-(x), Z^+(x))\big|_{x=x^*}} > 0$$   
and where $C^+(\alpha^*_+)$ is given by~\eqref{C(alpha)}.

\item[$(ii)$]  If $\alpha >\alpha^*_+$ and $r(\alpha -\alpha^*_+)^2 \to c>0$, then 
\begin{equation}
    g^{v_0}(r\cos(\alpha), r\sin(\alpha)) \underset{r\to+\infty\atop \alpha\to \alpha^*_+}{\sim} \frac{1}{2}C_*^+ \Big(1-\textnormal{erf}\Big(\sqrt{c} A(\alpha^*_+ )\Big)\Big) f_*(v_0)e^{ -r(\cos(\alpha)x^* + \sin(\alpha) y^*)} .
\end{equation}
where $C^+_*$ is given by \eqref{eq:C+*}, $A(\alpha^*_+) = \frac{-x'(\alpha^*_+)}{x'_\omega(0,\alpha^*_+)}$ and where $x'_\omega(0,\alpha^*_+)$ is defined in \eqref{zop}.

\item[$(iii)$] If $r (\alpha -\alpha^*_+)^2 \to 0$, then 
\begin{equation}
    g^{v_0}(r\cos(\alpha), r\sin(\alpha)) \underset{r\to+\infty\atop \alpha\to \alpha^*_+}{\sim} \frac{1}{2}C_*^+ f_*(v_0)e^{ -r(\cos(\alpha)x^* + \sin(\alpha) y^*)} .
\end{equation}

\item[$(iv)$]  If $\alpha <\alpha^*_+$ and $r (\alpha -\alpha^*_+)^2 \to c>0$, then 
\begin{equation}
    g^{v_0}(r\cos(\alpha), r\sin(\alpha)) \underset{r\to+\infty\atop \alpha\to \alpha^*_+}{\sim} \frac{1}{2}C_*^+ \Big(1+\textnormal{erf}\Big(\sqrt{c} A(\alpha^*_+ )\Big)\Big) f_*(v_0)e^{ -r(\cos(\alpha)x^* + \sin(\alpha) y^*)} .
\end{equation}

\item[$(v)$]  If  $\alpha<\alpha^*_+$ and $r (\alpha -\alpha^*_+)^2 \to \infty$, then
\begin{equation}
    g^{v_0}(r\cos(\alpha), r\sin(\alpha)) \underset{r\to+\infty\atop \alpha\to \alpha^*_+}{\sim} C_*^+ f_*(v_0)e^{ -r(\cos(\alpha)x^* + \sin(\alpha) y^*)} .
\end{equation}
\end{itemize} 
The symmetric asymptotics hold as $\alpha \to \alpha^*_-$. 
\end{theo}

\subsection{Martin boundary, general case}\label{sub:martingeneral}
The study of the asymptotics carried out in Sections~\ref{sub:asafterpole} and~\ref{sub:as:pole}, together with the results established in Sections~\ref{subsub:tout} and~\ref{subsub:touttout}, leads to the following theorem on the Martin boundary.

\begin{theo}[Martin Boundary and harmonic functions for general $q$]
\label{th3new}
If $x^* = +\infty$, then Theorem~\ref{thm5} applies. Assume now that $x^* \in (0, \M)$. 
Then, the Martin boundary $\Gamma$ is given by
\begin{equation}\label{front_Martin2}
\Gamma = \left[\{e^{i\alpha}\}_{\alpha \in [\alpha^*_+,\alpha^*_-] \cap \mathcal M} \cup \{ue^{i\tilde\alpha_b}\}_{u\in[1,2]}\right]/\mathcal{R}_1 \;\sim\; \mathbb{S}^1
\end{equation}
where $\mathcal{R}_1$ is the equivalence relation given by $e^{i\alpha^*_-}\,\mathcal{R}_1\,e^{i\alpha^*_+}$, $e^{i\pi}\,\mathcal{R}_1\,2e^{i\tilde\alpha_b}$, and $[x\mathcal{R}_1y \Longleftrightarrow x=y]$ for the other points (see Figure~\ref{fig:martingeneral}).
Furthermore, the minimal Martin boundary, denoted by $\Gamma_{min}$, is given by
\begin{equation}
\Gamma_{min} = \left[\{e^{i\alpha}\}_{\alpha \in [\alpha^*_+,\alpha^*_-] \cap \mathcal M}\right] /\mathcal{R}_2 \;\sim\; [0,1]
\end{equation}
where $\mathcal{R}_2$ is the equivalence relation given by $e^{i\alpha^*_-}\,\mathcal{R}_2\,e^{i\alpha^*_+}$ and $[x\mathcal{R}_2y \Longleftrightarrow x=y]$ for the other points (see again Figure~\ref{fig:martingeneral}). Moreover, for each positive harmonic function $h$ (with respect to the process), there exists a unique Radon measure $\nu$ on $[\alpha^*_+, \alpha^*_-)\cap \mathcal M$ such that
\begin{equation}\label{eq:representation2}
    \forall v \in \R^2, \quad h(v) = \int_{[\alpha^*_+, \alpha^*_-)\cap \mathcal M} \frac{h_\alpha(v)}{h_\alpha(0)} \, d\nu(\alpha) 
\end{equation}
where the conventions $\frac{h_{\alpha^*_+}(v)}{h_{\alpha^*_+}(0)} = \frac{f_*(v)}{f_*(0)}$ and $\frac{h_{\pi}(v)}{h_{\pi}(0)} = \frac{f_\pi(v)}{f_\pi(0)}$ are used, since 
$$\underset{\alpha \to \alpha^*_+\atop \alpha \in \mathcal M}{\lim} \frac{h_\alpha(v)}{h_\alpha(0)} \left(= \underset{\alpha \to \alpha^*_-\atop \alpha \in \mathcal M}{\lim} \frac{h_\alpha(v)}{h_\alpha(0)}\right) = \frac{f_*(v)}{f_*(0)}\quad \textnormal{and}\quad \frac{h_\alpha(v)}{h_\alpha(0)} \underset{\alpha \to \pi \atop \alpha \in \mathcal M}{\longrightarrow} \frac{f_\pi(v)}{f_\pi(0)}.$$
Finally, the case $x^* \in (\m, 0)$ is symmetric.
\end{theo}
\begin{figure}[t]
\centering
\begin{subfigure}[b]{0.47\linewidth}
  \centering
  \includegraphics[width=\linewidth]{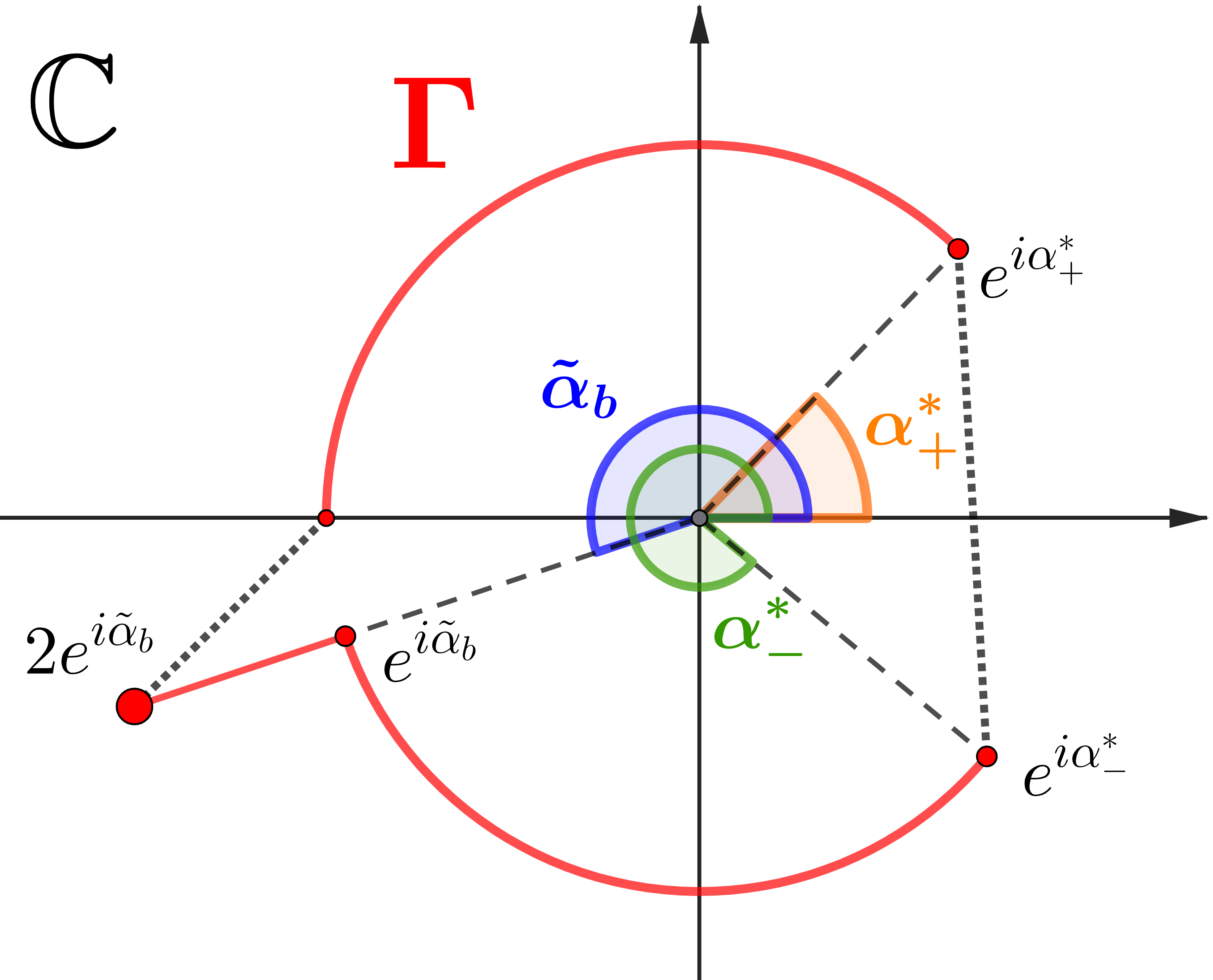}
  \caption{Martin Boundary.}
\end{subfigure}
\hfill
\begin{subfigure}[b]{0.47\linewidth}
  \centering
  \includegraphics[width=\linewidth]{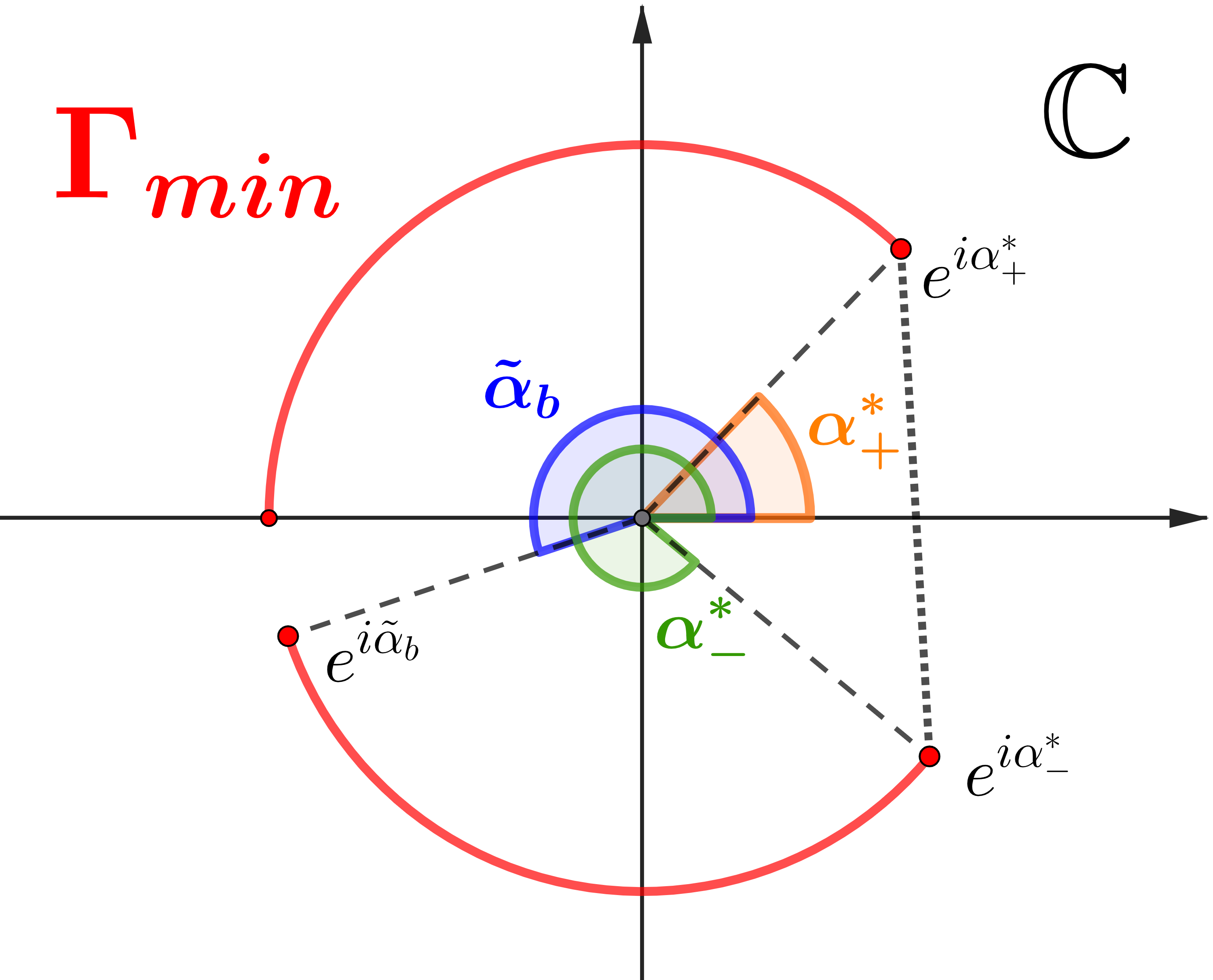}
  \caption{Minimal Martin Boundary.}
\end{subfigure}
\caption{Martin Boundary and Minimal Martin Boundary for the general process in the case $x^* \in (0, \M)$ and $x_{min}^- < x_{min}^+$.}
\label{fig:martingeneral}
\end{figure}
\newpage
\appendix

\section{Two-sided Tauberian theorem}
Tauberian theorems are known for Laplace transforms on $\R^+$, giving a link between the asymptotics of a function and its Laplace transform (cf \cite[Theorem 37.1]{doetsch_introduction_1974} or \cite[Lemma C.2]{dai_reflecting_2011}). In fact, the proof can be directly adapted to two-sided Laplace transforms without any technical issue. 
\begin{lemma}\label{lapinv}
Let $f :\R \longrightarrow \R^+$ be a continuous function and $\phi(z):= \int_\R e^{xz}f(x)dx$. Suppose that there are $a < 0 < b$ such that $\phi$ converges absolutely in $a<\Re(z)<b$. For $\delta > 0$, let $G^+_\delta(b) = \{z \in \C, \Re(z) > 0, z\neq b, |\arg(z-b)| > \delta\}$ (where the complex argument is taken in $(-\pi,\pi]$). Suppose that there exists $0 \leq \delta< \pi/2$ such that
\begin{itemize}
\item Function $\phi$ has a holomorphic continuation on $G^+_\delta(b)$.
\item The following asymptotics hold: $$\phi(z) \underset{|z|\to + \infty\atop z \in G^+_\delta(b)}\longrightarrow 0.$$
\item There exists $\lambda \in \R$ and constants $c \in \C$, $A \neq 0$ such that
\begin{equation}\label{aslap}(\phi(z) -c)(b-z)^\lambda \underset{z\to b\atop z\in G^+_\delta(b)}{\longrightarrow}A.\end{equation}
\end{itemize}
Then: $$f(x) \underset{x \to +\infty}{\sim}\frac{A}{\Gamma(\lambda)}e^{-bx}x^{\lambda - 1} $$
\end{lemma}
Obviously, the symmetric lemma for the asymptotics of $f$ in $-\infty$ holds.

\section{An integral asymptotic lemma}
\begin{lemma}\label{lemme_technique}
The following asymptotics hold as $q \to +\infty$
$$\int_0^1 \sqrt{1-s}e^{qs^2}ds \underset{q\to+\infty}{\sim} \frac{\sqrt \pi}{4\sqrt 2}\frac{e^{q}}{q^{3/2}}.$$
\end{lemma}

\begin{proof}
It suffices to establish that, for some $\epsilon>0$, the integral from
 over $[1-\epsilon, 1]$ has the same asymptotic. Since the asymptotic of the integral over $[0, 1-\epsilon ]$ is bounded by $O (\exp (q(1-\epsilon)^2)))$, as $q \to \infty.$ Let us look at
$$J(q, \epsilon)= \exp(-q)\int_{1-\epsilon} ^1 \sqrt{1-s} \exp(q s^2) ds$$
for some $\epsilon>0$ small enough. After the change of variables $t= q-qs^2$, it becomes
$$ J(q, \epsilon)= \frac{1}{2 q} \int_0^{q (1- (1-\epsilon)^2)}
  \frac{ \sqrt{ 1-\sqrt{  1-t/q}   }  }{\sqrt{1-t/q}} \exp(-t)dt. $$
For any $\epsilon>0$  small enough there exists a constant $C>0$ such that
$$ \Big|  \frac{ \sqrt{ 1-\sqrt{  1-s}   }  }{\sqrt{1-s}} -   \sqrt{\frac{s}{2}} \Big| \leq C s \sqrt{ \frac{s}{2}} , \  \  \forall s \in[ 0, 1-(1-\epsilon)^2]. $$
Then
$$ J(q, \epsilon)=  \frac{1}{2 q} \int\limits_0^{ q (1- (1-\epsilon)^2)} \sqrt{\frac{ t}{2q}}
  \exp(-t) dt + R(q, \epsilon)$$
  where
$$|R(q, \epsilon)| \leq  \frac{1}{2 q} \int_0^\infty \sqrt{\frac{ t}{2q}}  \frac{t}{q}\exp(-t)dt = O(q^{-5/2}), \  \  \ q \to \infty$$
  is of smaller order than the first term. It remains to compute
  $$    \int_0^{ q (1- (1-\epsilon)^2)} \sqrt{t}
  \exp(-t) dt \to \frac{\sqrt{\pi}}{2}, \  \ q \to \infty.$$
\end{proof}

\section{The case of dimension 1 when \texorpdfstring{$\mu_-<0<\mu_+$}{mu- < 0 < mu+}}
\label{annexe:dim1}

Consider the diffusion
\[
 dX_t=\sigma(X_t)\,dW_t+b(X_t)\,dt+q\,dL_t^0(X),
 \qquad -1<q<1,
\]
where $L^0(X)$ is symmetric local time at zero,
\[
 \sigma(x)=\sigma_-\ind_{\{x<0\}}+\sigma_+\ind_{\{x\geq0\}},
 \qquad b(x)=\mu_-\ind_{\{x<0\}}+\mu_+\ind_{\{x\geq0\}},
\]
and $\sigma_\pm>0$, $\mu_-<0<\mu_+$. In this transient regime,
$X_t$ tends almost surely to either $-\infty$ or $+\infty$.
We compute the Green density with respect to Lebesgue measure,
\[
 g(x_0,x)=\int_0^\infty p_t(x_0,x)\,dt,
\]
by deriving a functional equation and inverting its one-sided Laplace
transforms. Throughout, we retain the notation
\begin{equation}\label{app:kernels}
 \gamma_-(y)=\frac{\sigma_-^2}{2}y^2+\mu_-y,
 \qquad
 \gamma_+(z)=\frac{\sigma_+^2}{2}z^2+\mu_+z,
 \qquad
 y^*=-\frac{2\mu_-}{\sigma_-^2}>0,
 \quad z^*=-\frac{2\mu_+}{\sigma_+^2}<0.
\end{equation}
Thus $\gamma_-(y)=\sigma_-^2y(y-y^*)/2$ and
$\gamma_+(z)=\sigma_+^2z(z-z^*)/2$.  \\

\paragraph{\textbf{Functional equation.}}
\nb{For $y>0$ and $z<0$, define
\begin{align*}
 \varphi^-_{x_0}(y)
 &=\mathbb E_{x_0}\!\left[\int_0^\infty
       e^{yX_t}\ind_{\{X_t<0\}}\,dt\right]
   =\int_{-\infty}^0 e^{yx}g(x_0,x)\,dx,\\
 \varphi^+_{x_0}(z)
 &=\mathbb E_{x_0}\!\left[\int_0^\infty
       e^{zX_t}\ind_{\{X_t>0\}}\,dt\right]
   =\int_0^\infty e^{zx}g(x_0,x)\,dx.
\end{align*}
These transforms are finite and satisfy
\begin{equation}\label{app:fe}
 {\begin{aligned}
 \gamma_-(y)\varphi^-_{x_0}(y)
 +\gamma_+(z)\varphi^+_{x_0}(z)
 +\frac{y(q-1)+z(1+q)}{2}\mathbb E_{x_0}[L_\infty^0]
 ={}&-e^{yx_0}\ind_{\{x_0<0\}}-e^{zx_0}\ind_{\{x_0\geq0\}}.
 \end{aligned}}
\end{equation}
Indeed, apply It\^o--Tanaka to the bounded continuous function
$f(x)=e^{yx}\ind_{\{x<0\}}+e^{zx}\ind_{\{x\geq0\}}$.
Its derivatives at zero are $f'(0-)=y$ and $f'(0+)=z$.
The coefficient of symmetric local time is therefore
\[
 \frac{f'(0+)-f'(0-)}2
 +q\frac{f'(0+)+f'(0-)}2
 =\frac{y(q-1)+z(1+q)}2.
\]
The stochastic integral has mean zero, since its integrand is bounded.
Taking expectations gives the finite-time version of \eqref{app:fe},
with right-hand side $\mathbb E_{x_0}[f(X_t)]-f(x_0)$.
\\
First choose $(y,z)=(y^*,z^*)$. Both kernels vanish, while
$z^*(1+q)+y^*(q-1)<0$. Since $f\geq0$, the finite-time identity bounds
$\mathbb E_{x_0}[L_t^0]$ uniformly in $t$. Moreover,
$f(X_t)\to0$ almost surely, and $0\leq f\leq1$. Dominated and monotone
convergence yield
\begin{equation}\label{app:local-time}
 \mathbb E_{x_0}[L_\infty^0]
 =-\frac{2\bigl(e^{y^*x_0}\ind_{\{x_0<0\}}
                  +e^{z^*x_0}\ind_{\{x_0\geq0\}}\bigr)}
             {z^*(1+q)+y^*(q-1)}.
\end{equation}
To justify finiteness of the transforms before taking limits, choose
$0<y<y^*$ and $z^*<z<0$. Both kernels and the local-time coefficient
are negative. The finite-time identity consequently bounds the sum
of the two nonnegative occupation terms, multiplied by
$-\gamma_-(y)$ and $-\gamma_+(z)$, by $f(x_0)\leq1$.
Monotonicity in $y$ and $z$ then gives finiteness for every $y>0$,
$z<0$. Passing to the limit proves \eqref{app:fe}.}\\

\paragraph{\textbf{Solving the functional equation.}}
\nb{Set $y=y^*$ in \eqref{app:fe} and subtract the identity at
$(y^*,z^*)$. Factoring $\gamma_+(z)$ gives
\begin{equation}\label{app:phi-plus}
 \varphi^+_{x_0}(z)
 =\ind_{\{x_0\geq0\}}
       \frac{e^{z^*x_0}-e^{zx_0}}{\gamma_+(z)}
  -\frac{1+q}{\sigma_+^2z}\mathbb E_{x_0}[L_\infty^0].
\end{equation}
Likewise, setting $z=z^*$ and subtracting the same identity yields
\begin{equation}\label{app:phi-minus}
 \varphi^-_{x_0}(y)
 =\ind_{\{x_0<0\}}
       \frac{e^{y^*x_0}-e^{yx_0}}{\gamma_-(y)}
  +\frac{1-q}{\sigma_-^2y}\mathbb E_{x_0}[L_\infty^0].
\end{equation}
The apparent singularities at $z=z^*$ and $y=y^*$ are removable;
the formulas there are understood by continuity. Notice that the
argument uses $y>0$ and $z<0$: setting either variable to zero before
passing to the infinite-time limit would leave a nonzero terminal
contribution on one of the two escape events.}\\

\paragraph{\textbf{Laplace inversion.}}
\nb{The constant terms in \eqref{app:phi-plus}--\eqref{app:phi-minus}
are inverted using
\[
 \int_0^\infty e^{zx}\,dx=-\frac1z,
 \qquad
 \int_{-\infty}^0 e^{yx}\,dx=\frac1y.
\]
For the remaining terms, direct integration, splitting at $x=x_0$,
gives the transform pairs
\begin{align}
 \int_0^\infty e^{zx}
  \frac{e^{z^*(x_0-x)_+}-e^{z^*x_0}}{\mu_+}\,dx
 &=\frac{e^{z^*x_0}-e^{zx_0}}{\gamma_+(z)},
 &&x_0\geq0,\ z<0,\label{app:pair-plus}\\
 \int_{-\infty}^0 e^{yx}
  \frac{e^{-y^*(x-x_0)_+}-e^{y^*x_0}}{-\mu_-}\,dx
 &=\frac{e^{y^*x_0}-e^{yx_0}}{\gamma_-(y)},
 &&x_0<0,\ y>0.\label{app:pair-minus}
\end{align}
For example, the left-hand side of \eqref{app:pair-plus} equals
\[
 \frac1{\mu_+}\left[
 e^{z^*x_0}\frac{e^{(z-z^*)x_0}-1}{z-z^*}
 -\frac{e^{zx_0}}{z}+\frac{e^{z^*x_0}}{z}\right]
 =\frac{e^{z^*x_0}-e^{zx_0}}{\gamma_+(z)}.
\]
Similarly, splitting the negative half-line at $x_0$ gives
\[
 \frac1{-\mu_-}\left[
 \frac{e^{yx_0}}y+
 e^{y^*x_0}\frac{1-e^{(y-y^*)x_0}}{y-y^*}
 -\frac{e^{y^*x_0}}y\right]
 =\frac{e^{y^*x_0}-e^{yx_0}}{\gamma_-(y)}.
\]
Uniqueness of the one-sided Laplace transform now gives, for $x\ne0$,
\begin{equation}\label{app:green-local-time}
 g(x_0,x)=
 \begin{cases}
 \displaystyle
 \ind_{\{x_0<0\}}
 \frac{e^{-y^*(x-x_0)_+}-e^{y^*x_0}}{-\mu_-}
 +\frac{1-q}{\sigma_-^2}\mathbb E_{x_0}[L_\infty^0],&x<0,\\[10pt]
 \displaystyle
 \ind_{\{x_0\geq0\}}
 \frac{e^{z^*(x_0-x)_+}-e^{z^*x_0}}{\mu_+}
 +\frac{1+q}{\sigma_+^2}\mathbb E_{x_0}[L_\infty^0],&x>0.
 \end{cases}
\end{equation}
Substituting \eqref{app:local-time} produces the explicit formula
\begin{equation}\label{app:green-explicit}
 {\displaystyle
 g(x_0,x)=
 \begin{cases}
 \displaystyle
 -\frac1{\mu_-}\left(e^{-y^*(x-x_0)_+}
 -\frac{z^*(1+q)e^{y^*x_0}}{z^*(1+q)+y^*(q-1)}\right),
 &x<0,\ x_0<0,\\[11pt]
 \displaystyle
 -\frac{2(1-q)e^{z^*x_0}}
 {\sigma_-^2\bigl(z^*(1+q)+y^*(q-1)\bigr)},
 &x<0\leq x_0,\\[11pt]
 \displaystyle
 -\frac{2(1+q)e^{y^*x_0}}
 {\sigma_+^2\bigl(z^*(1+q)+y^*(q-1)\bigr)},
 &x_0<0<x,\\[11pt]
 \displaystyle
 \frac1{\mu_+}\left(e^{z^*(x_0-x)_+}
 -\frac{y^*(q-1)e^{z^*x_0}}{z^*(1+q)+y^*(q-1)}\right),
 &x>0,\ x_0\geq0.
 \end{cases}}
\end{equation}
The expressions agree as $x_0$ crosses zero. As a density with respect
to Lebesgue measure, $g(x_0,\cdot)$ need not be continuous at zero;
its one-sided limits satisfy
\[
 \sigma_+^2(1-q)g(x_0,0+)
 =\sigma_-^2(1+q)g(x_0,0-).
\]
The value of the density at \(x=0\) can be chosen arbitrarily.}\\

\paragraph{\textbf{Limits at infinity.}}
\nb{Formula \eqref{app:green-explicit} also yields
$g(x_0,x)\to h_\pm(x_0)$ as $x\to\pm\infty$, where
\begin{align*}
 h_+(x_0)&=
 \begin{cases}
 \displaystyle -\frac{2(1+q)e^{y^*x_0}}
 {\sigma_+^2\bigl(z^*(1+q)+y^*(q-1)\bigr)},&x_0<0,\\[8pt]
 \displaystyle \frac1{\mu_+}\left(1-
 \frac{y^*(q-1)e^{z^*x_0}}{z^*(1+q)+y^*(q-1)}\right),&x_0\geq0,
 \end{cases}\\[6pt]
 h_-(x_0)&=
 \begin{cases}
 \displaystyle -\frac1{\mu_-}\left(1-
 \frac{z^*(1+q)e^{y^*x_0}}{z^*(1+q)+y^*(q-1)}\right),&x_0<0,\\[8pt]
 \displaystyle -\frac{2(1-q)e^{z^*x_0}}
 {\sigma_-^2\bigl(z^*(1+q)+y^*(q-1)\bigr)},&x_0\geq0.
 \end{cases}
\end{align*}
In particular, $h_\pm$ are positive and continuous, and
$\mu_+h_+-\mu_-h_-=1$. }

{\subsection*{Acknowledgments}
S.F. and M.P. thank Antoine Lejay for very instructive discussions on stochastic processes in discontinuous layered media and porous barriers.
I.K. and M.P. are grateful to Lorenzo Zambotti for generously sharing his expertise on the links between PDEs and stochastic processes, and in particular for his useful suggestions regarding the analysis of transition densities.
I.K. is indebted to Irina Ignatiouk-Robert and Kilian Raschel for enriching discussions on discrete random walks analogous to the diffusion processes studied in the present paper. \nb{The authors would like to thank the anonymous referee for their valuable comments, which helped to improve the article.}
}


\end{document}